\documentclass[preprint,11pt]{elsarticle}

\usepackage{lmodern}
\usepackage{amsmath,amsthm,amssymb,enumerate}
\usepackage{color}
\usepackage{datetime}
\usepackage{fancyhdr}
\usepackage{graphicx}
\usepackage{array}
\usepackage{float}
\usepackage{eurosym}
\usepackage{pdfpages}
\usepackage{subcaption}
\usepackage[margin=1.0in]{geometry}
\usepackage[hidelinks]{hyperref}
\usepackage{lineno}

\newtheorem{theorem}{Theorem}[section]
\newtheorem{lemma}[theorem]{Lemma}
\newtheorem{proposition}[theorem]{Proposition}
\newtheorem{corollary}[theorem]{Corollary}
\newtheorem{definition}{Definition}[section]
\newtheorem{remark}{Remark}
\newtheorem{assumption}{Assumption}[section]

\journal{}

\begin{document}

\begin{frontmatter}

\title{Well-posedness and a sign-graph limit for an active
Cahn--Hilliard equation on Riemannian manifolds}

\author[1]{Darko Mitrovi\'c\corref{cor1}}
\ead{darkom@ucg.ac.me}

\author[2,3]{Andrej Novak}
\ead{andrej.novak@phy.hr}

\author[1]{Ajla \v{S}ukurica}
\ead{ajla.s@ucg.ac.me}

\cortext[cor1]{Corresponding author}

\address[1]{Faculty of Mathematics and Natural Sciences,
University of Montenegro, Cetinjski put bb,
81000 Podgorica, Montenegro}

\address[2]{Department of Physics, Faculty of Science,
University of Zagreb, Bijeni\v{c}ka cesta 32,
10000 Zagreb, Croatia}

\address[3]{Luxembourg School of Business,
46 C\^ote d'Eich, 1450 Luxembourg}

% Repaired version: q-power anchoring with q>6; beta_q(r)=|r|^{q-2}r.

\begin{abstract}
We establish a weak-solution theory and a singular constitutive limit for the
non-variational Cahn--Hilliard evolution
\[
\partial_t u
=
\Delta_g\!\left(
W'(u)-\mu\Delta_g u
-\nu a(\Delta_g u)|\nabla_g u|_g
\right)
+\lambda(x)\bigl(b(u_{\mathrm{ref}})-b(u)\bigr)
\]
on a compact Riemannian manifold, possibly with boundary. Here
\(W'(u)=u^3-u\), \(a\) is a bounded \(C^1\), non-decreasing classifier
with \(a(0)=0\), and \(b\) is a monotone coercive anchoring law with
\(q\)-growth, \(q>6\). The factor \(|\nabla_g u|_g\) concentrates the active response in diffuse transition layers, while \(a(\Delta_g u)\) makes it sensitive to the sign of the intrinsic second-order profile. The resulting chemical-potential correction is non-variational and therefore breaks the passive Cahn–Hilliard gradient-flow structure.
The natural estimates, however, give
only weak compactness of \(\Delta_g u\), which is insufficient to identify
\(a(\Delta_g u)|\nabla_g u|_g\).

For \(L^2\) initial data and \(L^q\) reference data, we construct weak
solutions in arbitrary dimension. A Galerkin-level one-sided comparison,
using the scalar monotonicity of \(a\) and \(b\) together with biharmonic
coercivity, yields
\[
\Delta_g u_m\longrightarrow\Delta_g u
\quad\text{strongly in }L^2(Q_T),
\]
and hence \(u_m\to u\) in \(L^2(0,T;H_D^2(M))\), together with strong
convergence of the active product in $L^2(Q_T)$. This identifies the constitutive law
in an ordinary weak formulation, without a second-derivative defect. In two
dimensions, weak solutions are unique, and continuous dependence holds in the
$L^2$ metric of the initial data and of \(b(u_{\mathrm{ref}})\), whenever
the latter distance is finite.

The estimates use only the bound and monotonicity of the classifier, not its
slope. They are therefore uniform for
\(a_N(r)=\frac{2}{\pi}\arctan(Nr)\), although $a_N'(0)\to\infty$. In any
dimension, every sequence of corresponding solutions admits a subsequence such
that $u_N\to u$ strongly in \(L^2(0,T;H_D^2(M))\) and
\[
 a_N(\Delta_g u_N)\stackrel{*}{\rightharpoonup}\xi
 \quad\text{in }L^\infty(Q_T),
 \qquad
\xi\in\operatorname{Sign}(\Delta_g u),
\]
and
\[
a_N(\Delta_g u_N)|\nabla_g u_N|_g
\rightharpoonup \xi|\nabla_g u|_g
\quad\text{in }L^2(Q_T).
\]
The limit is governed by the maximal-monotone sign graph. It is a
classifier-steepness limit at fixed interfacial scale, not a sharp-interface
limit. In two dimensions the limiting state is unique, although \(\xi\) may
be non-unique on \(\{\Delta_g u=0\}\). Consequently, the entire family \(u_N\) converges. Flat-torus computations solve the discrete graph inclusion directly and, under coupled refinement, identify a spinodal regime in which activity increases the characteristic length and reduces interfacial content
while worsening phase purity, consistently with local fourth-order damping.
\end{abstract}

\begin{keyword}
Cahn--Hilliard equation \sep Riemannian manifolds \sep active phase-field models
 \sep weak solutions  \sep differential inclusion \sep maxima monotone graph \sep singular constitutive limits 

\texttt{AMS subject classifications:}
35K35, 35K55, 35K65, 35B40, 58J35, 82C26.
\end{keyword}
\end{frontmatter}
%\linenumbers

\section{Introduction}
\label{sec:introduction}

\subsection{From passive Cahn--Hilliard dynamics to an active constitutive law}

The Cahn--Hilliard equation is one of the basic continuum models for phase
separation, spinodal decomposition, coarsening, and diffuse-interface motion.
For an order parameter $u$ on a Riemannian manifold $(M,g)$, the passive
chemical potential associated with the Ginzburg--Landau energy is
\[
\mathcal E_g(u)
=
\int_M
\left(
W(u)+\frac{\mu}{2}|\nabla_g u|_g^2
\right)\,dV_g,
\qquad
m_{\rm pas}[u]=W'(u)-\mu\Delta_g u.
\]
Combining the conservation law
\[
\partial_tu+\operatorname{div}_gJ=0,
\]
with the constitutive relation $J=-\nabla_gm_{\rm pas}$ gives
\[
\partial_tu=\Delta_g\bigl(W'(u)-\mu\Delta_gu\bigr).
\]
On a closed manifold, or under the corresponding no-flux conditions, this is
the $H^{-1}$-gradient flow of $\mathcal E_g$ and satisfies
\[
\frac{d}{dt}\mathcal E_g(u(t))
=
-\int_M|\nabla_gm_{\rm pas}[u]|_g^2\,dV_g.
\]
Thus the passive dynamics combines a conserved order parameter, fourth-order
regularisation, and free-energy dissipation. We refer to the original work of
Cahn and Hilliard \cite{Cahn}, the analysis of Elliott and Zheng
\cite{ElliottZheng86}, and the surveys and monographs
\cite{Mir_book,MirAIMS,Novi08}. Its relation to Model-B dynamics and critical
phenomena is classical \cite{HohenbergHalperin77}, while a derivation based on
microforce balance was given in \cite{Gurtin96}.

Many phase-separating systems are driven, coupled to an environment, or
maintained away from equilibrium. At the phase-field level, one way to encode
this is to add to the chemical potential a term that is not the first
variation of an equilibrium free energy. This is the precise sense in which
we use the term \emph{active}. It is the same broad nonequilibrium distinction that
underlies Active Model B and related active phase-field theories, where
interfacial contributions break detailed balance
\cite{Cate18,Nardini17,TjhungNardiniCates18,Wittkowski14}. The particular
constitutive law studied below is different from the standard Active Model B
closure and is not presented as a microscopic derivation for a specific
active material. It defines a class of fourth-order, non-variational
phase-field evolutions in which the additional chemical-potential response is
localised at interfaces and is sensitive to the sign of an intrinsic
second-order quantity.

Curved substrates provide a natural setting for such a problem. Phase
separation on membranes and surfaces couples interfacial structure to the
geometry of the supporting space. Related passive and coupled surface
Cahn--Hilliard systems have been studied in
\cite{AbelsKampmann20,CaetanoElliott21,
CaetanoElliottGrasselliPoiatti23,DuJuTian11,ElliottRanner15,GarckeKampmannRaetzRoeger16}, while active membrane models provide a broader physical
motivation \cite{Gov04,ProstBruinsma96,RamaswamyTonerProst00}. We therefore
formulate the evolution intrinsically on a compact Riemannian manifold,
possibly with boundary. The geometry enters both the constitutive law and the
analysis through the Riemannian gradient, the Laplace--Beltrami operator, its
spectral realisation, and the Sobolev and elliptic estimates on $M$ (see
also \cite{Chavel}).

\subsection{Nonequilibrium structure and singular response}

Let $(M,g)$ be a compact, connected, $d$-dimensional Riemannian manifold,
possibly with boundary. We study
\begin{equation}
\label{intro:main-equation}
\begin{aligned}
\partial_tu
={}&
\Delta_g
\Bigl(
W'(u)-\mu\Delta_gu
-\nu a(\Delta_gu)|\nabla_gu|_g
\Bigr)
+
\lambda(x)\bigl(b(u_{\mathrm{ref}})-b(u)\bigr),
\\
&\quad\qquad
u(0,\cdot)=u_{\mathrm{in}},
\end{aligned}
\end{equation}
where
\[
W(u)=\frac14(1-u^2)^2,
\qquad
W'(u)=u^3-u,
\qquad
\mu,\nu>0.
\]
The classifier $a:\mathbb R\to\mathbb R$ is bounded, $C^1$,
non-decreasing, and satisfies $a(0)=0$. The coefficient $\lambda$
satisfies
\[
0<\lambda_*\leq\lambda(x)\leq\lambda^*<\infty
\qquad\text{for a.e. }x\in M.
\]
The anchoring law $b$ is
continuous, non-decreasing, satisfies $b(0)=0$, and, for some $q>6$
and positive constants $c_0,c_1,c_2$, obeys
\[
b(r)r\geq c_0|r|^q-c_1,
\qquad
|b(r)|\leq c_2\bigl(1+|r|^{q-1}\bigr).
\]
The model example is
\[
b(r)=\beta_q(r):=|r|^{q-2}r.
\]
The data satisfy
\[
u_{\mathrm{in}}\in L^2(M),
\qquad
u_{\mathrm{ref}}\in L^q(M).
\]

Equation \eqref{intro:main-equation} contains three mechanisms with different
structural roles. The terms $W'(u)-\mu\Delta_gu$ form the passive
Cahn--Hilliard chemical potential. The contribution
\[
m_{\rm act}[u]
=
-\nu a(\Delta_gu)|\nabla_gu|_g
\]
is a non-variational interfacial response. Finally,
\[
\lambda(x)\bigl(b(u_{\mathrm{ref}})-b(u)\bigr)
\]
describes local exchange with, or feedback toward, a prescribed reference
profile. The reference state $u_{\mathrm{ref}}$ is independent of the
initial state $u_{\mathrm{in}}$, and $\lambda(x)$ specifies the local
strength of the coupling.

The two factors in $m_{\rm act}$ have complementary meanings.
The factor $|\nabla_g u|_g$ concentrates the response in diffuse transition
layers, whereas $a(\Delta_g u)$ makes it sensitive to the sign of the
Laplace--Beltrami operator. For a locally one-dimensional transition profile,
this sign distinguishes the concave and convex sides of the layer. To see how geometry enters this distinction, let $\Gamma\subset M$
be a smooth interface, let $r$ be a signed geodesic distance,
and consider the inner profile
\[
u(x)\simeq q_*\bigl(r/\varepsilon\bigr).
\]
Formally,
\[
|\nabla_gu|_g
\simeq
\varepsilon^{-1}
\left|q_*'\bigl(r/\varepsilon\bigr)\right|,
\]
while
\[
\Delta_gu
\simeq
\varepsilon^{-2}q_*''\bigl(r/\varepsilon\bigr)
+\varepsilon^{-1}H_\Gamma q_*'\bigl(r/\varepsilon\bigr)
+\text{lower-order terms},
\]
up to the sign convention for the geodesic mean curvature $H_\Gamma$.
The leading response therefore resolves the two sides of the transition
profile through the sign of $q_*''$, while the mean curvature of the interface within \(M\) appears at the next order. In this precise sense the Riemannian
formulation is part of the constitutive model (only a change of notation).

The form of $m_{\rm act}$ is inherited from shock filtering
\cite{AlvarezMazorra94,Nov22,Osh} and is related to edge-selective and
anisotropic diffusion \cite{NovakCMS2026,PeronaMalik90,Weickert98}, but its
dynamical action is different from that of a second-order shock filter. Here
the classifier-dependent quantity is inserted into the chemical potential
and is then acted upon by $\Delta_g$. In particular, for
\[
a_N(r)=\frac{2}{\pi}\arctan(Nr),
\]
if $G=|\nabla_gu|_g$ is frozen locally and treated as spatially constant,
then in the regime $N|\Delta_gu|\ll1$,
\[
\Delta_g\bigl[-\nu a_N(\Delta_gu)G\bigr]
\simeq
-\frac{2\nu NG}{\pi}\Delta_g^2u.
\]
This has the sign of additional fourth-order damping. The shock-filter origin of the classifier therefore does not imply a universal sharpening
effect for \eqref{intro:main-equation}, the actual response depends on the state and parameter regime.

Although the full dynamics is not a gradient flow, it retains a useful
one-sided structure. Since $a$ is non-decreasing and $a(0)=0$,
\[
a(r)r\geq0,
\]
and hence the active term contributes
\[
-\nu\int_M
a(\Delta_gu)|\nabla_gu|_g\Delta_gu\,dV_g
\leq0
\]
to the basic $L^2$-balance. This is an analytical monotonicity mechanism,
not a dissipation identity for the Ginzburg--Landau energy. The anchoring
has a second monotone structure:
\[
\bigl(b(u_{\mathrm{ref}})-b(u)\bigr)
\bigl(u-u_{\mathrm{ref}}\bigr)
\leq0.
\]
Its coercivity controls large amplitudes and supplies the $L^q(Q_T)$
estimate required by the cubic double-well term. In the present compactness
argument, the strict inequality $q>6$ allows one to obtain strong
$L^6(Q_T)$-convergence and hence strong $L^2(Q_T)$-convergence of
$W'(u_m)$.

The restoring term also has a separate variational interpretation. If
\[
B(r):=\int_0^r b(s)\,ds,
\]
then $B$ is convex and, formally, the functional
\[
\mathcal A_{u_{\mathrm{ref}}}(v)
:=
\int_M\lambda(x)
\bigl(B(v)-b(u_{\mathrm{ref}})v\bigr)\,dV_g
\]
satisfies
\[
-\frac{\delta\mathcal A_{u_{\mathrm{ref}}}}{\delta v}
=
\lambda(x)\bigl(b(u_{\mathrm{ref}})-b(v)\bigr).
\]
The passive Cahn--Hilliard flux and the anchoring source therefore have
separate relaxational structures, in different metrics. It is the
classifier-dependent correction to the chemical potential that breaks the
equilibrium gradient-flow structure.
Source and proliferation terms in Cahn--Hilliard equations, and
feedback-stabilisation mechanisms for phase-field systems, have been studied
in \cite{Fakih2015,BarbuColliGilardiMarinoschi2017,MirAIMS}.

The source also determines the balance law. On a closed manifold, the
fourth-order part has zero mean and
\[
\frac{d}{dt}\int_Mu\,dV_g
=
\int_M\lambda(x)
\bigl(b(u_{\mathrm{ref}})-b(u)\bigr)\,dV_g.
\]
Thus the differential Cahn--Hilliard part remains conservative, whereas the
full anchored system does not. If the source is formally switched off, the
usual conservation law is recovered on a closed manifold. If
$\partial M\neq\emptyset$, we use the Dirichlet Laplacian and its associated
spectral/Navier fourth-order realisation. In that case the constant function
is not an admissible test function, and the formal mean balance also contains
the boundary flux of the full chemical potential. No conservation or
source-only balance is asserted for this boundary realisation.

The steep-classifier family introduces a second constitutive question.
Although
\[
a_N(r)\longrightarrow\operatorname{sgn}(r)
\qquad\text{for }r\neq0,
\]
the limiting law cannot be obtained by choosing a value for
$\operatorname{sgn}(0)$. If $z_N=c/N$, then $z_N\to0$, while
\[
a_N(z_N)=\frac{2}{\pi}\arctan(c).
\]
The stable closure is instead the maximal monotone graph
\[
\operatorname{Sign}(r)
=
\begin{cases}
\{1\}, & r>0,\\
[-1,1], & r=0,\\
\{-1\}, & r<0,
\end{cases}
\qquad
\operatorname{Sign}=\partial|\,\cdot\,|;
\]
see \cite{Brezis1973}. The interval at zero records the response that may
remain when the regularised Laplacians approach zero on the scale $N^{-1}$.
The singular limit is therefore a constitutive differential inclusion, not
the original equation with a single-valued convention inserted at zero.
For a fixed nonnegative interfacial weight, this graph is monotone in the
Laplacian variable; the dependence of that weight on $u$ means that the
complete fourth-order evolution remains non-variational.

Nonsmooth bulk free energies in classical Cahn--Hilliard theory lead to
obstacle-type variational inequalities, equivalently to maximal-monotone
subdifferential graphs acting on the order parameter itself
\cite{BloweyElliott91}. Here the graph acts on $\Delta_gu$, is
weighted by $|\nabla_gu|_g$, and enters a non-variational chemical
potential. It is therefore located at the highest derivative controlled by
the basic estimate and is not a standard additive monotone perturbation.
Moreover, $N\to\infty$ changes the steepness of this constitutive response.
It is not a sharp-interface limit: no interfacial-width parameter is sent to zero, and in particular \(\mu\) remains fixed. 
The problem addressed below is whether these smooth, geometrically intrinsic constitutive laws select a closed and stable evolution when their response
becomes discontinuous. This requires compactness at the level of $\Delta_gu$, identification of the limiting graph, and stability of the
state even when the multiplier itself is not uniquely determined.

Singular limits and compactness problems for Cahn--Hilliard equations
also arise in kinetic and nonlocal models; see, for instance,
\cite{ElbarPerthameSkrzeczkowski2024} and the earlier kinetic derivation
\cite{ElbarMasonPerthameSkrzeczkowski2023}. In those works, the limit changes
the scale or effective description of the model. In the present paper, the
diffuse-interface thickness is kept fixed, and the singular parameter only
makes the active constitutive response steeper.

\subsection{Analytical obstruction and compactness mechanism}

The main analytical issue is the closure of the constitutive product
\[
a(\Delta_gu)|\nabla_gu|_g.
\]
For Galerkin approximations $u_m$, the basic estimates give
\[
u_m
\quad\text{bounded in}\quad
L^\infty(0,T;L^2(M))
\cap L^2(0,T;H_D^2(M))
\cap L^q(Q_T).
\]
Aubin--Lions compactness yields strong convergence of $\nabla_gu_m$ in
$L^2(Q_T)$, but initially only weak $L^2(Q_T)$-compactness of
$\Delta_gu_m$. Nonlinear composition with $a$ is not weakly continuous,
so these convergences do not identify $a(\Delta_gu_m)$, and boundedness of
$a$ gives only an unidentified weak limit of the active product.

We obtain the missing second-order compactness by a one-sided comparison at
the Galerkin level. After replacing the preliminary weak limit $u$ by a
simultaneous time--space regularisation $u^\zeta$, we compare $u_m$ with
its spectral projection $P_mu^\zeta$ and test the Galerkin equation with
\[
u_m-P_mu^\zeta.
\]
The decisive point is the pointwise inequality
\[
\bigl(a(\Delta_gu_m)-a(\Delta_gP_mu^\zeta)\bigr)
\bigl(\Delta_gu_m-\Delta_gP_mu^\zeta\bigr)
|\nabla_gu_m|_g
\geq0.
\]
The monotonicity of $b$ gives an analogous one-sided estimate for the
anchoring term, while the passive biharmonic part controls
$\Delta_g(u_m-P_mu^\zeta)$. Passing first $m\to\infty$ and then
$\zeta\downarrow0$ yields
\[
\Delta_gu_m\longrightarrow\Delta_gu
\qquad\text{strongly in }L^2(Q_T).
\]
Elliptic regularity then gives strong convergence in
$L^2(0,T;H_D^2(M))$, and the active product is identified strongly in
$L^2(Q_T)$. This strong second-order convergence is the closure mechanism
for the active chemical potential (it is not a direct consequence of the usual Aubin--Lions argument). Notice also that the full map
$u\mapsto a(\Delta_gu)|\nabla_gu|_g$ is not asserted to be monotone. The proof uses the scalar monotonicity of $a$ within the one-sided comparison.

The same mechanism is consistent under steepening of the classifier. For
\[
a_N(r)=\frac{2}{\pi}\arctan(Nr),
\qquad
a_N'(0)=\frac{2N}{\pi},
\]
the slopes are not uniformly bounded as $N\to\infty$. Our estimates use
instead only $|a_N|\leq1$ and monotonicity and are therefore uniform in
$N$. This is the bridge from the smooth equations to the sign-graph
inclusion.

\subsection{Main results and their significance}

The results can be summarized as follows.

\begin{enumerate}[(i)]
\item
For every bounded $C^1$ non-decreasing classifier $a$ with $a(0)=0$,
Theorem~\ref{thm:existence-weak-solution} gives a weak solution for
\[
u_{\mathrm{in}}\in L^2(M),
\qquad
u_{\mathrm{ref}}\in L^q(M),
\qquad q>6,
\]
in any dimension. More importantly, the Galerkin construction satisfies
\[
u_m\longrightarrow u
\quad\text{strongly in }L^2(0,T;H_D^2(M)),
\]
and
\[
a(\Delta_gu_m)|\nabla_gu_m|_g
\longrightarrow
a(\Delta_gu)|\nabla_gu|_g
\quad\text{strongly in }L^2(Q_T).
\]
The limit is therefore an ordinary weak solution with a pointwise identified active flux.

\item
In dimension two, Theorem~\ref{thm:uniqueness-beta} proves uniqueness and
continuous dependence for the smooth-classifier problem. If $u$ and $v$
correspond to the data pairs
$(u_{\mathrm{in}},u_{\mathrm{ref}})$ and
$(v_{\mathrm{in}},v_{\mathrm{ref}})$, and
\[
b(u_{\mathrm{ref}})-b(v_{\mathrm{ref}})\in L^2(M),
\]
then
\begin{align*}
&\sup_{t\in[0,T]}\|u(t)-v(t)\|_{L^2(M)}^2
+\int_0^T\|\Delta_g(u-v)(t)\|_{L^2(M)}^2\,dt
\\
&\qquad\leq
C_T\left(
\|u_{\mathrm{in}}-v_{\mathrm{in}}\|_{L^2(M)}^2
+\|b(u_{\mathrm{ref}})-b(v_{\mathrm{ref}})\|_{L^2(M)}^2
\right).
\end{align*}
Agmon's inequality controls the cubic double-well difference in dimension
two. The active estimate uses boundedness and
monotonicity of $a$, but not a bound for $a'$. This distinction is
important for the steep-classifier limit.

\item
Theorem~\ref{thm:steep-classifier-limit} identifies the singular limit of the
specific family
\[
a_N(r)=\frac{2}{\pi}\arctan(Nr).
\]
In arbitrary dimension, every sequence of corresponding weak solutions has
a subsequence for which
\[
u_N\longrightarrow u
\quad\text{strongly in }L^2(0,T;H_D^2(M)),
\]
\[
a_N(\Delta_gu_N)\stackrel{*}{\rightharpoonup}\xi
\quad\text{in }L^\infty(Q_T),
\qquad
\xi\in\operatorname{Sign}(\Delta_gu)
\quad\text{a.e. in }Q_T,
\]
and
\[
a_N(\Delta_gu_N)|\nabla_gu_N|_g
\rightharpoonup
\xi|\nabla_gu|_g
\quad\text{weakly in }L^2(Q_T).
\]
The limit is the differential inclusion
\[
\partial_tu
=
\Delta_g\Bigl(
W'(u)-\mu\Delta_gu-\nu\xi|\nabla_gu|_g
\Bigr)
+\lambda(x)\bigl(b(u_{\mathrm{ref}})-b(u)\bigr),
\qquad
\xi\in\operatorname{Sign}(\Delta_gu).
\]
In dimension two, Theorem~\ref{thm:sign-graph-uniqueness} shows that the
state $u$ is unique and continuously dependent on the data, and
Corollary~\ref{cor:full-steep-convergence} upgrades subsequential convergence
to convergence of the entire family $u_N$. The multiplier $\xi$ may
remain non-unique on $\{\Delta_gu=0\}$, but this does not affect
uniqueness of the state $u$.

\item
Section~\ref{sec:numerics} tests the structures used in the analysis. A prescribed exact solution verifies the complete smooth-classifier scheme. A weighted convex formulation solves the sign-graph inclusion directly at the discrete level and is compared with finite classifiers under coupled refinement. Paired spinodal experiments, a multi-direction
continuous-dependence test, and an audited reconstruction stress test then examine observable consequences of the active response. In the resolved regime studied here, activity increases the characteristic spinodal length and reduces interfacial content, while it does not improve phase purity or reconstruction error. These computations support the local damping interpretation above.
\end{enumerate}

These results show that a discontinuous, concavity-sensitive
constitutive law can be obtained as a controlled limit of smooth active
phase-field equations. Strong compactness closes the second-derivative
response, while monotonicity selects a unique two-dimensional state even
when the limiting constitutive multiplier is not uniquely determined. This
separation between state selection and multiplier selection is one of the
main features of the singular limit.

\subsection{Relation to earlier work}

Two related, but mathematically different, lines of work should first be
distinguished. Classical shock filters are second-order morphological
evolutions in which the sign of a second derivative selects between dilation
and erosion \cite{Osh}. Edge-selective and anisotropic diffusion provide a
complementary class of image evolutions \cite{PeronaMalik90,Weickert98}, and
combined shock-filter/diffusion models were studied in
\cite{AlvarezMazorra94}. In the present equation,
the classifier-dependent term is instead part of the chemical potential and
is acted upon by an additional Laplacian. A related two-scale shock--diffusion mechanism was considered in \cite{NovakCMS2026}, it belongs
to a second-order morphological class and does not contain a Cahn--Hilliard chemical potential or the sign-graph limit studied here.

Shock-dependent classifiers were incorporated into modified Euclidean Cahn--Hilliard equations for image inpainting in \cite{Nov22}. The closest
(analytical) predecessor is the Euclidean equation studied in \cite{ARMA},
\[
\partial_tu
=
\Delta\bigl(-\nu\arctan(\Delta u)|\nabla u|-\mu\Delta u\bigr)
+\lambda(x)(u_0-u).
\]
At the level of the differential expression, this corresponds to removing
the double-well term from \eqref{intro:main-equation}, choosing a linear
fidelity law, and fixing an arctangent classifier, up to normalisation. The
hypotheses are not literal special cases of those imposed here: the
inpainting coefficient in \cite{ARMA} may vanish on the damaged region,
whereas we assume $\lambda\geq\lambda_*>0$, and the reference profile in
the present paper is prescribed independently of the initial state.

The main compactness obstruction in \cite{ARMA} was already the nonlinear
dependence on the Laplacian. There the regularised solutions converged
strongly at the first-derivative level, while their Laplacians were available
only weakly in $L^2$. The limiting active response was consequently encoded
by an entropy admissibility condition and, in variational form, by a Young
measure associated with the approximate Laplacians. Existence and uniqueness
were obtained within that framework, but the nonlinear term was not
identified as $\arctan(\Delta u)|\nabla u|$ in an ordinary weak
formulation without an additional concentration assumption.

The present argument resolves this loss of second-order compactness. For
every bounded smooth non-decreasing classifier, the Galerkin-level comparison
identifies the active product directly in the weak equation. The model also
retains the cubic double-well force, separates the initial and reference
states, allows a general monotone coercive anchoring law, and is formulated intrinsically on a compact Riemannian manifold. Taken together, these changes produce a different weak-solution theory that can not be obtained by formal replacement of Euclidean derivatives by their geometric counterparts. 

The steep-classifier limit is also different from the spatial regularisation
used in \cite{ARMA}. Here the constitutive laws themselves vary with $N$
and converge to a discontinuous graph. Strong convergence of the
corresponding Laplacians identifies
\[
\xi\in\operatorname{Sign}(\Delta_gu)
\qquad\text{a.e. in }Q_T,
\]
and two-dimensional state uniqueness upgrades subsequential compactness to
convergence of the complete family, despite possible nonuniqueness of the
multiplier on the zero-Laplacian set.

Cahn--Hilliard image inpainting and higher-order reconstruction
\cite{Bert07d,Bert07,Burg09,Mira16,Cher15} remain relevant to the origin of
the model and to one numerical stress test. In this contribiution the focus is different:
the central object is a nonequilibrium geometric phase-field equation, and the main questions are closure of its second-derivative-dependent chemical potential, stability of the resulting weak evolution, and state selection in a singular constitutive limit. The numerical section accordingly illustrates the discrete graph law and phase-separation response in addition to a reconstruction example.

\subsection{Organization of the paper}

Section~\ref{sec:model} introduces the geometric setting, the boundary
realisation, the assumptions on the classifier and anchoring law, and the
weak formulation. Section~\ref{sec:stability} constructs weak solutions for
bounded smooth monotone classifiers, proves the strong second-order
compactness needed to identify the active term, and establishes uniqueness
and continuous dependence in dimension two. Section~\ref{sec:steep-classifier}
treats the steep-classifier limit, identifies the maximal-monotone sign-graph
inclusion, and proves two-dimensional state uniqueness and full-family
convergence. Section~\ref{sec:numerics} develops the smooth and direct-graph
discretisations, verifies the implementation, and reports the graph-limit,
spinodal, stability, and reconstruction experiments together with their
numerical error budgets.

\section{Model and preliminaries}
\label{sec:model}

\subsection{Geometric setting and notation}

Let $(M,g)$ be a compact, connected, $d$-dimensional Riemannian manifold,
possibly with non-empty boundary $\partial M$. We assume that the boundary is
sufficiently regular, for instance of class $C^{1,1}$, so that the usual trace,
elliptic regularity, and spectral results for the Laplace--Beltrami operator
are available. We fix $T>0$ and write
\[
Q_T:=(0,T)\times M.
\]
The Riemannian volume measure will be denoted by $dV_g$. If $\partial M\neq
\emptyset$, the induced surface measure on $\partial M$ is denoted by
$dS_g$, and $\mathbf n$ stands for the outward unit normal.

We briefly recall the geometric notation used throughout the paper. The metric
$g$ induces a pointwise scalar product on tangent vectors, denoted by
\[
\langle X,Y\rangle_g=g(X,Y),
\]
and the associated norm is written as
\[
|X|_g=\sqrt{\langle X,X\rangle_g}.
\]
In local coordinates $(x^1,\ldots,x^d)$, the metric is represented by the
matrix $(g_{ij})$, its inverse by $(g^{ij})$, and the volume element is
\[
dV_g=\sqrt{|g|}\,dx,
\qquad |g|:=\det(g_{ij}).
\]

The Levi--Civita connection associated with $g$ will be denoted by
$\nabla^g$. It is the unique connection on $TM$ which is torsion-free and
compatible with the metric, namely
\[
\nabla^g_XY-\nabla^g_YX=[X,Y],
\qquad
X\langle Y,Z\rangle_g
=
\langle \nabla^g_XY,Z\rangle_g
+
\langle Y,\nabla^g_XZ\rangle_g.
\]
For a scalar function $u$, the Riemannian gradient $\nabla_g u$ is the vector
field defined by
\[
\langle \nabla_g u,X\rangle_g=Xu
\qquad\text{for every vector field }X.
\]
In local coordinates,
\[
(\nabla_g u)^i=g^{ij}\partial_j u.
\]
The Hessian of $u$ is the covariant two-tensor
\[
\nabla_g^2 u:=\nabla^g(du),
\]
and the divergence of a vector field $X$ is defined by
\[
\operatorname{div}_g X
=
\frac{1}{\sqrt{|g|}}\partial_i\bigl(\sqrt{|g|}X^i\bigr).
\]
We use the sign convention
\[
\Delta_g u:=\operatorname{div}_g(\nabla_g u).
\]
Thus, in local coordinates,
\[
\Delta_g u
=
\frac{1}{\sqrt{|g|}}
\partial_i\left(\sqrt{|g|}\,g^{ij}\partial_j u\right).
\]
With this convention, the operator $-\Delta_g$ is non-negative under Dirichlet
boundary conditions.

For sufficiently smooth functions $u,\varphi$, Green's formula reads
\[
\int_M \varphi\,\Delta_g u\,dV_g
=
-\int_M \langle \nabla_g \varphi,\nabla_g u\rangle_g\,dV_g
+
\int_{\partial M}\varphi\,\partial_{\mathbf n,g}u\,dS_g,
\]
where
\[
\partial_{\mathbf n,g}u:=\langle \nabla_g u,\mathbf n\rangle_g.
\]
In particular, if $\varphi=0$ on $\partial M$, the boundary term vanishes.

The unknown
\[
u=u(t,x)
\]
is a scalar order parameter on the manifold. It may represent, for instance, the relative concentration of two phases constrained to a curved surface.
Values close to two preferred states correspond to the two pure phases, while transition layers represent diffuse interfaces. The equation considered below combines passive Cahn--Hilliard relaxation with a non-variational, gradient-localised concavity response.  Although its constitutive form is
inherited from shock-filter constructions, no unconditional interfacial sharpening property is assumed. Monotonicity of the classifier gives a
favourable sign in the basic $L^2$-estimate, while its morphological action depends on the parameter regime and the evolving state.

\subsection{The active Cahn--Hilliard family on a manifold}

Let $\mu,\nu>0$ be fixed. We consider the equation
\begin{equation}
\label{eq:active-CH-manifold}
\partial_t u
=
\Delta_g\Bigl(
W'(u)
-\mu\Delta_g u
-\nu a(\Delta_g u)|\nabla_g u|_g
\Bigr)
+\lambda(x)(b(u_{\mathrm{ref}})-b(u))
\qquad\text{in }Q_T.
\end{equation}
The initial condition is
\begin{equation}
\label{ic-general}
u(0,x)=u_{\mathrm{in}}(x),
\qquad x\in M.
\end{equation} 

We assume that
\begin{equation}
\label{ass:initial-reference-data}
u_{\mathrm{in}}\in L^2(M),
\qquad
u_{\mathrm{ref}}\in L^q(M),
\qquad
q>6.
\end{equation}
The assumption on $u_{\mathrm{in}}$ is sufficient for the initial trace and
the basic energy estimate.  On the other hand,
$u_{\mathrm{ref}}\in L^q(M)$ implies
\[
b(u_{\mathrm{ref}})
\in L^{q'}(M),
\qquad
q'=\frac{q}{q-1},
\]
which guarantees that the prescribed anchoring term is well defined in the
weak formulation.

If $\partial M\neq\emptyset$, we use the Dirichlet realization
$A=-\Delta_g$ with $D(A)=H^2(M)\cap H_0^1(M)$, and the fourth-order
operator is interpreted through the spectral/Navier form associated with $A^2$.
Thus $u=0$ is imposed in the trace sense. The second Navier condition
$\Delta_g u=0$ is built into the chosen spectral/Navier realization of the
fourth-order operator and is recovered only for sufficiently regular
solutions (not as a separate trace condition at the weak level). For this reason, when writing the model in its formal strong form,
we shall use the shorthand notation
\begin{equation}
\label{bc-navier}
u(t,x)=0,
\qquad
\Delta_g u(t,x)=0,
\qquad (t,x)\in(0,T)\times\partial M.
\end{equation}
The condition \eqref{bc-navier} is therefore only a formal boundary notation
for the strong problem and is not part of the weak formulation below. If
$M$ is closed, that is, if $\partial M=\emptyset$, no boundary condition is
imposed.

The quantity
\[
W'(u)-\mu\Delta_g u
\]
is the usual Cahn--Hilliard chemical potential. The function $W$ is a
double-well potential. The model example is
\[
W(u)=\frac14(1-u^2)^2,
\qquad
W'(u)=u^3-u.
\]
The anchoring law
\[
b:\mathbb R\to\mathbb R
\]
is assumed to be continuous, non-decreasing, and coercive of supercubic
growth. More precisely, we assume that there exist an exponent $q>6$ and
constants $c_0,c_1,c_2>0$ such that
\[
b(r)r\geq c_0|r|^q-c_1,
\qquad
|b(r)|\leq c_2\bigl(1+|r|^{q-1}\bigr),
\qquad r\in\mathbb R.
\]
We also assume that $b(0)=0$, when symmetry between the two phases is
desired, $b$ may additionally be taken to be odd. The associated primitive
\[
B(r):=\int_0^rb(s)\,ds
\]
is convex, and the monotonicity of $b$ gives
\[
\bigl(b(r)-b(s)\bigr)(r-s)\geq0,
\qquad r,s\in\mathbb R.
\]
The coercivity assumption provides the $L^q(Q_T)$ control needed to identify
the cubic double-well term in the weak limit. Moreover, the growth condition
implies that
\[
u_{\mathrm{ref}}\in L^q(M)
\quad\Longrightarrow\quad
b(u_{\mathrm{ref}})\in L^{q'}(M),
\qquad
q'=\frac{q}{q-1}.
\]
The principal example is the power-law anchoring
\[
b(r)=\beta_q(r):=|r|^{q-2}r,
\qquad q>6.
\]

If globally Lipschitz control of the double-well derivative is required, the
potential $W$ may instead be replaced by a smooth truncated potential $W_R$.
The additional term
\[
-\nu a(\Delta_g u)|\nabla_g u|_g
\]
is a non-variational, gradient-localised concavity response.  Its rigorous
structural property is monotonicity rather than unconditional anti-diffusion:
since $a(r)r\geq0$,
\[
-\nu\int_M
a(\Delta_g u)|\nabla_g u|_g\Delta_g u\,dV_g
\leq0,
\]
so it contributes with a favourable sign to the basic $L^2$-balance.
Moreover, if $G=|\nabla_g u|_g\geq0$ is frozen locally, treated as
spatially constant, and $|\Delta_g u|$ is small, then for
$a_N(r)=2\arctan(Nr)/\pi$,
\[
\Delta_g\!\left[-\nu a_N(\Delta_g u)G\right]
\simeq
-\frac{2\nu NG}{\pi}\Delta_g^2u.
\]
Thus, in this locally linearised regime, the active term has the same sign as
additional fourth-order damping.  Sharpening is not a structural consequence
of the classifier, and if it occurs in another regime, it must be demonstrated through a specified observable.  The computations in
Section~\ref{sec:numerics} identify instead a reduced-interface/coarsening
regime consistent with this local damping calculation.
The relaxation term
\[
\lambda(x)(b(u_{\mathrm{ref}})-b(u))
\]
is a nonlinear anchoring term. It has the same qualitative effect as the
standard linear fidelity term $\lambda(x)(u_{\mathrm{ref}}-u)$: where
$\lambda$ is large, the solution is strongly tied to the prescribed state
$u_{\mathrm{ref}}$, while where
$\lambda$ is small the dynamics is governed mainly by the differential
operator. The q-power form is chosen because it is slightly supercritical with respect to the
standard cubic double-well potential. Indeed, testing the equation by $u$ gives a
coercive contribution controlling $u$ in $L^q(Q_T)$ with $q>6$, which is stronger than the
integrability needed for the cubic term $W'(u)=u^3-u$.

We assume throughout that the anchoring coefficient satisfies
\begin{equation}
\label{ass-lambda}
\lambda\in L^\infty(M),
\qquad
0<\lambda_*\leq \lambda(x)\leq \lambda^*
\quad\text{for a.e. }x\in M,
\end{equation}
for some constants $\lambda_*,\lambda^*>0$.

The initial and reference data have different analytical roles, as recorded
in \eqref{ass:initial-reference-data}. The assumption
$u_{\mathrm{in}}\in L^2(M)$ supplies the initial trace and the
initial-difference term in the stability estimate, whereas
$u_{\mathrm{ref}}\in L^q(M)$ implies
$b(u_{\mathrm{ref}})\in L^{q'}(M)$ and makes the prescribed anchoring
functional meaningful.

\begin{assumption}[Classifier]
\label{ass:aclass}
The classifier $a:\mathbb R\to\mathbb R$ satisfies:
\begin{enumerate}[(i)]
\item $a\in C^1(\mathbb R)$;
\item $a$ is non-decreasing and such that $a(0)=0$;
\item there exists a constant $M_a>0$ such that
\[
|a(r)|\leq M_a
\qquad\text{for all }r\in\mathbb R.
\]
\end{enumerate}
\end{assumption}

A basic family of smooth classifiers is
\begin{equation}
\label{classifier-family}
a_N(r)=\frac{2}{\pi}\arctan(Nr),
\qquad N\in\mathbb N.
\end{equation}
For each fixed $N$, $a_N$ is smooth, odd (and thus $a_N(0)=0$), bounded,
and non-decreasing. Moreover,
\[
\lim_{N\to\infty}a_N(r)=\operatorname{sgn}(r)
\qquad\text{for every }r\neq 0.
\]
Thus the family \eqref{classifier-family} approximates a sign-type classifier
in the steep-classifier regime.

The Euclidean shock-filter Cahn--Hilliard equation studied in \cite{ARMA} is
recovered from \eqref{eq:active-CH-manifold} by taking $M$ to be a Euclidean
domain, replacing $\nabla_g$ and $\Delta_g$ by the usual Euclidean operators,
removing the double-well term, and using the corresponding linear anchoring
term. The present paper treats the more general geometric and double-well
model while working with ordinary weak solutions rather than
Young-measure-valued solution concepts.

\subsection{Functional framework}

For $k\in\mathbb N$, we denote by $H^k(M)$ the usual Sobolev space on the
Riemannian manifold $M$. The space $H_0^1(M)$ is defined as the closure of
$C_c^\infty(M^\circ)$ in $H^1(M)$, where $M^\circ$ denotes the interior
of $M$. We set
\[
H_D^2(M):=
\begin{cases}
H^2(M)\cap H_0^1(M), & \partial M\neq\emptyset,\\
H^2(M), & \partial M=\emptyset.
\end{cases}
\]
equipped with the norm
\[
\|u\|_{H_D^2(M)}:=\|u\|_{H^2(M)}.
\]
The trace condition $u=0$ on $\partial M$ is encoded in the definition of
$H_D^2(M)$. The second Navier condition $\Delta_g u=0$ on $\partial M$
will be understood in the weak sense associated with the formulation below.

Since the equation contains the nonlinear anchoring term $b(u)$, it is useful
to choose the test space so that both the fourth-order part and the anchoring
term are meaningful in arbitrary dimension. We therefore define
\[
\mathcal Y:=H_D^2(M)\cap L^q(M),
\]
with norm
\[
\|\varphi\|_{\mathcal Y}
:=
\|\varphi\|_{H_D^2(M)}+\|\varphi\|_{L^q(M)}.
\]
Then $\mathcal Y$ is a reflexive Banach space, continuously embedded into
$L^2(M)$. We denote its dual by
\[
\mathcal Y':=(\mathcal Y)'.
\]
The duality pairing between $\mathcal Y'$ and $\mathcal Y$ will be denoted
by
\[
\langle\cdot,\cdot\rangle_{\mathcal Y',\mathcal Y},
\]
or simply by $\langle\cdot,\cdot\rangle$ when no confusion can arise.

The reason for introducing $\mathcal Y$ is the following. If
$u\in L^q(M)$, then $b(u)\in L^{q'}(M)$, where
$q'=q/(q-1)$, and hence
\[
\left|\int_M b(u)\varphi\,dV_g\right|
\leq
\|b(u)\|_{L^{q'}(M)}\|\varphi\|_{L^q(M)}
\leq
C\bigl(1+\|u\|_{L^q(M)}^{q-1}\bigr)
\|\varphi\|_{L^q(M)}
\leq
C\bigl(1+\|u\|_{L^q(M)}^{q-1}\bigr)
\|\varphi\|_{\mathcal Y}.
\]
Thus $b(u)$ defines an element of $\mathcal Y'$. At the same time, since
$\mathcal Y\subset H_D^2(M)$, we have
\[
\Delta_g\varphi\in L^2(M)
\qquad\text{for every }\varphi\in\mathcal Y,
\]
which is exactly what is needed for the fourth-order terms.

The embedding $H_D^2(M)\hookrightarrow L^q(M)$ holds only under the usual
Sobolev restrictions. More precisely, it holds for every finite $q$ if
$d\le4$, while for $d>4$ it holds only for
\[
q\le \frac{2d}{d-4}.
\]
Thus, when $q$ is outside this range, the intersection space
$\mathcal Y=H_D^2(M)\cap L^q(M)$ must be kept explicitly.

\noindent For time-dependent functions we use the spaces
\[
L^2(0,T;H_D^2(M)),
\qquad
L^6(Q_T),
\qquad
W^{1,q'}(0,T;\mathcal X').
\]
The natural energy class is
\[
u\in L^\infty(0,T;L^2(M))
\cap L^2(0,T;H_D^2(M))
\cap L^q(Q_T).
\]
Fix $s\geq 2$ so large that
\[
\mathcal X:=D((I+A)^{s/2})\hookrightarrow H_D^2(M)\cap L^q(M).
\] 
The space $\mathcal X$ is equipped with the norm
\[
\|\varphi\|_{\mathcal X}
:=
\|(I+A)^{s/2}\varphi\|_{L^2(M)}.
\]
Since $s\ge2$, spectral calculus gives
\[
\mathcal X\hookrightarrow H_D^2(M),
\]
and since $s>d(1/2-1/q)$, Sobolev embedding gives
\[
\mathcal X\hookrightarrow L^q(M).
\] Consequently,
\[
\mathcal X\hookrightarrow \mathcal Y
=
H_D^2(M)\cap L^q(M)
\]
continuously. Thus the fourth-order terms and the $q$-power anchoring term
can both be tested against functions in $\mathcal X$. We denote the dual of
$\mathcal X$ by $\mathcal X'$.

\noindent
Furthermore, we require
\[
\partial_t u\in L^{q'}(0,T;\mathcal X').
\]

The condition $u\in L^6(Q_T)$ is needed in the untruncated double-well case
in order to ensure
\[
W'(u)=u^3-u\in L^2(Q_T).
\]
It is also compatible with the nonlinear anchoring term
$\lambda(x)(b(u_{\mathrm{ref}})-b(u))$. If a smooth truncated potential
$W_R$ with
globally Lipschitz derivative is used instead, the $L^6$-condition is not
needed for the double-well term, although it is still natural when the
q-power anchoring is present.

We shall repeatedly use the elliptic estimate
\begin{equation}
\label{elliptic-estimate}
\|v\|_{H^2(M)}
\leq
C\bigl(
\|\Delta_g v\|_{L^2(M)}
+
\|v\|_{L^2(M)}
\bigr),
\qquad v\in H_D^2(M),
\end{equation}
where $C>0$ depends only on $(M,g)$.

Let $A=-\Delta_g$ denote the following self-adjoint realization of the
Laplace--Beltrami operator on $L^2(M)$. If $\partial M\neq\emptyset$, then
$A$ is the Dirichlet Laplace--Beltrami operator, with
\[
D(A)=H^2(M)\cap H_0^1(M).
\]
If $M$ is closed, then $A$ is the non-negative self-adjoint
Laplace--Beltrami operator with
\[
D(A)=H^2(M).
\]
In both cases $A$ has compact resolvent. Hence there exist eigenvalues
\[
0\leq \lambda_1\leq \lambda_2\leq\cdots,
\qquad \lambda_k\to\infty,
\]
and an orthonormal basis $\{e_k\}_{k\geq1}$ of $L^2(M)$ such that
\[
A e_k=\lambda_k e_k .
\]
In the boundary case, the eigenfunctions satisfy the Dirichlet condition
\[
e_k|_{\partial M}=0 ,
\]
whereas in the closed case no boundary condition is imposed. Moreover,
$\lambda_1>0$ if $\partial M\neq\emptyset$, while $\lambda_1=0$ in the
closed case. If $M$ is closed and connected, then the constants span
$\ker A$.

For $m\in\mathbb N$, we denote by
\[
P_m:L^2(M)\to \mathcal V_m:=\operatorname{span}\{e_1,\ldots,e_m\}
\]
the $L^2$-orthogonal projection,
\[
P_m v=\sum_{k=1}^m (v,e_k)_{L^2(M)}e_k.
\]
The projection $P_m$ will be used in the Galerkin approximation and, in
particular, to make comparison functions admissible as Galerkin test
functions. Since $u_m\in \mathcal V_m$ implies $P_m u_m=u_m$, the
classifier in the Galerkin equation is simply evaluated at $\Delta_g u_m$.
When comparing $u_m$ with a smooth function $u^\zeta$, we use
$P_m u^\zeta$, so that $u_m-P_m u^\zeta\in\mathcal V_m$.

\begin{definition}[Weak solution]
\label{def:weak-solution}
Let $q>6$,
\[
u_{\mathrm{in}}\in L^2(M),
\qquad
u_{\mathrm{ref}}\in L^q(M),
\]
and assume \eqref{ass-lambda}, Assumption~\ref{ass:aclass}, and the standing
continuity, monotonicity, coercivity, and growth assumptions on $b$. A function
\[
u\in L^\infty(0,T;L^2(M))
\cap L^2(0,T;H_D^2(M))
\cap L^q(Q_T),
\qquad
\partial_tu\in L^{q'}(0,T;\mathcal X'),
\qquad
q'=\frac{q}{q-1},
\]
is called a weak solution of
\eqref{eq:active-CH-manifold}--\eqref{bc-navier} if, after redefining $u$
on a set of times of measure zero,
\[
u\in C_{\mathrm w}([0,T];L^2(M)),
\qquad
u(0)=u_{\mathrm{in}}
\quad\text{in }L^2(M),
\]
and, for every $\varphi\in\mathcal X$, the identity
\begin{align}
\label{weak-form-manifold}
\langle \partial_tu(t),\varphi\rangle_{\mathcal X',\mathcal X}
&=
\int_M
\Bigl(
W'(u)
-\mu\Delta_gu
-\nu a(\Delta_gu)|\nabla_gu|_g
\Bigr)
\Delta_g\varphi\,dV_g
\nonumber\\
&\quad+
\int_M
\lambda(x)\bigl(b(u_{\mathrm{ref}})-b(u)\bigr)
\varphi\,dV_g
\end{align}
holds for a.e. $t\in(0,T)$.

Equivalently, for every
\[
\varphi\in C^1([0,T];\mathcal X)
\]
and every $t\in[0,T]$, one has
\begin{align}
\label{weak-form-integrated}
&\int_M u(t)\varphi(t)\,dV_g
-
\int_M u_{\mathrm{in}}\varphi(0)\,dV_g
-
\int_0^t\int_M u\,\partial_t\varphi\,dV_g\,ds
\nonumber\\
&=
\int_0^t\int_M
\Bigl(
W'(u)
-\mu\Delta_gu
-\nu a(\Delta_gu)|\nabla_gu|_g
\Bigr)
\Delta_g\varphi\,dV_g\,ds
\nonumber\\
&\quad+
\int_0^t\int_M
\lambda(x)\bigl(b(u_{\mathrm{ref}})-b(u)\bigr)
\varphi\,dV_g\,ds.
\end{align}
\end{definition}

The definition is meaningful for the following reasons. Since
\[
u\in L^2(0,T;H_D^2(M)),
\]
and $H_D^2(M)\hookrightarrow H^1(M)$ continuously on a compact manifold,
we have
\[
\Delta_gu\in L^2(Q_T),
\qquad
\nabla_gu\in L^2(Q_T).
\]
The boundedness of $a$ therefore gives
\[
a(\Delta_gu)|\nabla_gu|_g\in L^2(Q_T).
\]
Moreover, since $q>6$ and $Q_T$ has finite measure,
\[
u\in L^q(Q_T)\hookrightarrow L^6(Q_T),
\]
and hence, for the standard double-well potential,
\[
W'(u)=u^3-u\in L^2(Q_T).
\]
Consequently,
\[
W'(u)-\mu\Delta_gu
-\nu a(\Delta_gu)|\nabla_gu|_g
\in L^2(Q_T).
\]
For every $\varphi\in\mathcal X$, we have
$\Delta_g\varphi\in L^2(M)$, and therefore
\[
\varphi\longmapsto
\int_M
\Bigl(
W'(u)-\mu\Delta_gu
-\nu a(\Delta_gu)|\nabla_gu|_g
\Bigr)
\Delta_g\varphi\,dV_g
\]
defines an element of $\mathcal X'$ for a.e. $t$, and belongs to
$L^2(0,T;\mathcal X')$.

For the anchoring term, the growth assumption
\[
|b(r)|\leq c_b\bigl(1+|r|^{q-1}\bigr),
\qquad r\in\mathbb R,
\]
together with
\[
u\in L^q(Q_T),
\qquad
u_{\mathrm{ref}}\in L^q(M),
\]
implies
\[
b(u)\in L^{q'}(Q_T),
\qquad
b(u_{\mathrm{ref}})\in L^{q'}(M).
\]
Indeed,
\[
|b(r)|^{q'}
\leq C\bigl(1+|r|^q\bigr).
\]
Since $\lambda\in L^\infty(M)$ and
$\mathcal X\hookrightarrow L^q(M)$, the map
\[
\varphi\longmapsto
\int_M
\lambda(x)\bigl(b(u_{\mathrm{ref}})-b(u)\bigr)
\varphi\,dV_g
\]
defines an element of $\mathcal X'$ and belongs to
$L^{q'}(0,T;\mathcal X')$. Therefore the entire right-hand side of
\eqref{weak-form-manifold} belongs to
$L^{q'}(0,T;\mathcal X')$, since
\[
L^2(0,T;\mathcal X')
\hookrightarrow
L^{q'}(0,T;\mathcal X')
\]
for $q'<2$ and $T<\infty$.

The integrated formulation \eqref{weak-form-integrated} follows from
\eqref{weak-form-manifold} by testing in time and integrating by parts.

Equivalently, this can be justified by the standard Steklov regularisation in
time: for $h>0$,
\[
z^h(t):=\frac1h\int_t^{t+h}z(s)\,ds,
\qquad 0<t<T-h,
\]
and \(z^h\to z\) strongly in \(L^p(0,T;B)\) whenever
\(z\in L^p(0,T;B)\), \(1\le p<\infty\). Moreover,
\[
\partial_t z^h(t)=\frac{z(t+h)-z(t)}{h}
\]
in the corresponding Bochner space. Applying this to the present weak identity,
and then letting \(h\downarrow0\), gives the time integration by parts in
\eqref{weak-form-integrated}. In later energy estimates, when the test
function depends on the solution itself, one first tests with the regularised
and spatially projected functions \(P_N z^h\in\mathcal X\), and then lets
\(N\to\infty\) and \(h\downarrow0\).

We shall use the following form of the Aubin--Lions compactness lemma.

\begin{theorem}[Aubin--Lions]
\label{AuLi}
Let \(X_0\), \(X\), and \(X_1\) be Banach spaces such that
\[
X_0\subset X\subset X_1,
\]
where \(X_0\) is compactly embedded into \(X\) and \(X\) is continuously
embedded into \(X_1\). For \(1\leq p,q\leq\infty\), set
\[
W_{p,q}
=
\{v\in L^p(0,T;X_0):\partial_t v\in L^q(0,T;X_1)\}.
\]
Then:
\begin{enumerate}[(i)]
\item if \(p<\infty\), the embedding
\[
W_{p,q}\hookrightarrow L^p(0,T;X)
\]
is compact;

\item if \(p=\infty\) and \(q>1\), the embedding
\[
W_{p,q}\hookrightarrow C([0,T];X)
\]
is compact.
\end{enumerate}
\end{theorem}

In the sequel we apply Aubin--Lions with
\[
X_0=H_D^2(M),
\qquad
X=H_D^1(M),
\qquad
X_1=\mathcal X',
\]
where
\[
H_D^1(M):=
\begin{cases}
H_0^1(M), & \partial M\neq\emptyset,\\
H^1(M), & \partial M=\emptyset.
\end{cases}
\]
We apply Aubin--Lions with
\[
X_0=H_D^2(M),\qquad X=H_D^1(M),\qquad X_1=\mathcal X'.
\]
The embedding \(H_D^2(M)\hookrightarrow\hookrightarrow H_D^1(M)\) is compact,
and \(H_D^1(M)\hookrightarrow \mathcal X'\) is continuous because
\(H_D^1(M)\hookrightarrow L^2(M)\) and \(\mathcal X\hookrightarrow L^2(M)\).
Thus boundedness in \(L^2(0,T;H_D^2(M))\) and of the time derivative in
\(L^{q'}(0,T;\mathcal X')\) implies compactness in
\(L^2(0,T;H_D^1(M))\).

We shall also use the standard weak time-continuity fact: if
\[
u\in L^\infty(0,T;L^2(M)),
\qquad
\partial_t u\in L^{p}(0,T;\mathcal X')
\quad\text{for some }p>1,
\]
then, after redefining \(u\) on a set of times of measure zero,
\[
u\in C_w([0,T];L^2(M)).
\]
Thus the initial condition \(u(0)=u_{\mathrm{in}}\) is meaningful in the weak
\(L^2(M)\)-sense.

Whenever an \(L^2\)-energy identity is used below under weaker time
regularity, it is justified by a standard time-regularisation argument using Steklov averages; see
\cite[Chapter~III, Section~1]{LSU}.
If, in addition, one has the stronger estimate
\[
\partial_t u\in L^2(0,T;H^{-2}(M)),
\]
then the Lions--Magenes theorem yields
\[
u\in C([0,T];L^2(M)).
\]

\noindent
We shall also use the following interpolation consequence. If \(q>6\),
\(z_n\to z\) strongly in \(L^2(Q_T)\), and \((z_n)_n\) is bounded in
\(L^q(Q_T)\), then
\[
z_n\to z
\qquad\text{strongly in }L^6(Q_T).
\]
Indeed,
\[
\|z_n-z\|_{L^6(Q_T)}
\le
\|z_n-z\|_{L^2(Q_T)}^{\theta}
\|z_n-z\|_{L^q(Q_T)}^{1-\theta},
\qquad
\theta=\frac{q-6}{3(q-2)}>0.
\]
Consequently, for the standard double-well potential,
\[
W'(z_n)=z_n^3-z_n\to z^3-z=W'(z)
\qquad\text{strongly in }L^2(Q_T).
\]

\section{The smooth-classifier problem: existence, compactness, and
two-dimensional well-posedness}
\label{sec:stability}

In this section we prove existence of weak solutions to
\eqref{eq:active-CH-manifold} for a fixed smooth classifier \(a\) satisfying
Assumption~\ref{ass:aclass}. Throughout the section,
\[
q'=\frac{q}{q-1},
\qquad
\mathcal X=D((I+A)^{s/2}),
\]
where \(s\ge2\) is chosen so that
\[
\mathcal X\hookrightarrow H_D^2(M)\cap L^q(M).
\]

\subsection{The Galerkin problem}
\label{subsec:Galerkin-problem}

Let \(\{e_k\}_{k\geq1}\) be the eigenbasis of the self-adjoint realization
\(A=-\Delta_g\) specified above: the Dirichlet realization when
\(\partial M\neq\emptyset\), and the Laplace--Beltrami realization, including
its constant zero mode, when \(M\) is closed. We set
\[
{\cal V}_m:=\operatorname{span}\{e_1,\ldots,e_m\},
\qquad
P_m:L^2(M)\to{\cal V}_m
\]
for the corresponding spectral projection.

We look for an approximate solution of the form
\begin{equation}
\label{galerkin-app}
u_m(t,x)=\sum_{k=1}^m c_k^m(t)e_k(x).
\end{equation}
Then \(u_m\in{\cal V}_m\), and hence
\[
P_m u_m=u_m.
\]
Thus, at the Galerkin level, the classifier may be written either as
\(a(\Delta_g u_m)\) or as \(a(\Delta_gP_m u_m)\). These two expressions are
identical for \(u_m\in{\cal V}_m\).

If \(\partial M\neq\emptyset\), then
\[
u_m|_{\partial M}=0,
\qquad
\Delta_g u_m|_{\partial M}=0,
\]
because
\[
\Delta_g u_m=-\sum_{k=1}^m\lambda_k c_k^m e_k
\]
and \(e_k|_{\partial M}=0\). Thus the Navier boundary conditions are
satisfied at the Galerkin level.

\begin{proposition}[Galerkin existence and uniform estimates]
\label{prop:Galerkin-existence}
Let $q>6$, and assume that
\[
u_{\mathrm{in}}\in L^2(M),
\qquad
u_{\mathrm{ref}}\in L^q(M).
\]
Assume further that $\lambda$ satisfies \eqref{ass-lambda}, that $a$ satisfies
Assumption~\ref{ass:aclass}, and that $b$ satisfies the standing continuity,
monotonicity, coercivity, and growth assumptions.

Then, for every $m\in\mathbb N$, there exists
\[
u_m\in C^1([0,T];\mathcal V_m)
\]
such that
\[
u_m(0)=P_m u_{\mathrm{in}},
\]
and, for every $\varphi\in\mathcal V_m$ and every $t\in[0,T]$,
\begin{align}
\label{weak-form-galerkin-m}
(\partial_tu_m(t),\varphi)_{L^2(M)}
&=
\int_M
\Bigl(
W'(u_m)
-\mu\Delta_gu_m
-\nu a(\Delta_gu_m)|\nabla_gu_m|_g
\Bigr)
\Delta_g\varphi\,dV_g
\nonumber\\
&\quad+
\int_M
\lambda(x)\bigl(b(u_{\mathrm{ref}})-b(u_m)\bigr)
\varphi\,dV_g.
\end{align}

Moreover, there exists a constant $C_T>0$, independent of $m$, such that
\begin{equation}
\label{galerkin-uniform-estimate}
\begin{aligned}
&
\|u_m\|_{L^\infty(0,T;L^2(M))}^2
+
\|u_m\|_{L^2(0,T;H_D^2(M))}^2
+
\|u_m\|_{L^q(Q_T)}^q
\\
&\qquad\leq
C_T
\Bigl(
1
+\|u_{\mathrm{in}}\|_{L^2(M)}^2
+\|u_{\mathrm{ref}}\|_{L^q(M)}^q
\Bigr).
\end{aligned}
\end{equation}
The constant $C_T$ may depend on $T$, the geometry of $(M,g)$, the
parameters $\mu$ and $\nu$, the bounds on $\lambda$ and $a$, and the
structural constants in the assumptions on $b$, but not on $m$.

Finally, choose $s\geq2$ sufficiently large that
\[
\mathcal X:=D\bigl((I+A)^{s/2}\bigr)
\hookrightarrow H_D^2(M)\cap L^q(M).
\]
Then, with
\(
q'=\frac{q}{q-1},
\)
the family of time derivatives satisfies
\begin{equation}
\label{galerkin-time-derivative-estimate}
\sup_{m\in\mathbb N}
\|\partial_tu_m\|_{L^{q'}(0,T;\mathcal X')}
<\infty.
\end{equation}
\end{proposition}
\begin{proof}
We divide the proof into three steps.

\medskip
\noindent
\textbf{Step 1: finite-dimensional system.}

We seek an approximate solution of the form
\eqref{galerkin-app}. The coefficient vector
\[
c^m(t):=(c_1^m(t),\ldots,c_m^m(t))\in\mathbb R^m
\]
is required to satisfy
\begin{align}
\label{galerkin-m-system}
(\partial_tu_m,e_j)_{L^2(M)}
&=
\int_M
\Bigl(
W'(u_m)
-\mu\Delta_gu_m
-\nu a(\Delta_gu_m)|\nabla_gu_m|_g
\Bigr)
\Delta_ge_j\,dV_g
\nonumber\\
&\quad+
\int_M
\lambda(x)
\bigl(
b(u_{\mathrm{ref}})-b(u_m)
\bigr)e_j\,dV_g,
\qquad j=1,\ldots,m,
\end{align}
with initial condition
\[
u_m(0)=P_mu_{\mathrm{in}}.
\]

Equivalently, \eqref{galerkin-m-system} is an autonomous system of ordinary
differential equations
\[
\frac{d}{dt}c^m(t)=F_m(c^m(t)),
\qquad
c^m(0)=c_{\mathrm{in}}^m,
\]
for a vector field \(F_m:\mathbb R^m\to\mathbb R^m\).

We claim that \(F_m\) is continuous. Indeed, the maps
\[
c\longmapsto u_c:=\sum_{k=1}^m c_ke_k,
\qquad
c\longmapsto \nabla_gu_c,
\qquad
c\longmapsto \Delta_gu_c
\]
are linear and continuous from \(\mathbb R^m\) into every finite-dimensional
spectral norm. The continuity of \(W'\), \(a\), and \(b\), together with the
continuity of the norm map
\[
X\longmapsto |X|_g,
\]
therefore implies the continuity of all nonlinear terms in
\eqref{galerkin-m-system}. Moreover,
\[
b(u_{\mathrm{ref}})\in L^{q'}(M),
\qquad
e_j\in L^q(M),
\]
so the prescribed reference-state term is well defined.

Peano's existence theorem consequently yields a local solution
\[
c^m\in C^1([0,T_m);\mathbb R^m)
\]
for some \(T_m>0\), and hence
\[
u_m\in C^1([0,T_m);\mathcal V_m).
\]
The a priori estimate established below bounds
\(\|u_m(t)\|_{L^2(M)}\). Since the Galerkin basis is orthonormal in
\(L^2(M)\),
\[
|c^m(t)|_{\mathbb R^m}^2
=
\|u_m(t)\|_{L^2(M)}^2.
\]
Thus the coefficient vector cannot become unbounded in finite time. If the
maximal interval of existence had a finite endpoint, the boundedness of
\(c^m\) and the continuity of \(F_m\) would allow the solution to be
continued beyond that endpoint. Hence the Galerkin solution exists on the
whole interval \([0,T]\).

\medskip
\noindent
\medskip
\noindent
\textbf{Step 2: basic energy estimate.}

We test \eqref{galerkin-m-system} by $u_m$, equivalently multiply the
$j$-th equation by $c_j^m$ and sum over $j=1,\ldots,m$. Since
$u_m\in\mathcal V_m$, this gives
\begin{align}
\label{energy-m-start}
\frac12\frac{d}{dt}\|u_m\|_{L^2(M)}^2
&=
\int_M W'(u_m)\Delta_gu_m\,dV_g
-\mu\|\Delta_gu_m\|_{L^2(M)}^2
\nonumber\\
&\quad
-\nu\int_M
a(\Delta_gu_m)|\nabla_gu_m|_g\Delta_gu_m\,dV_g
\nonumber\\
&\quad+
\int_M
\lambda(x)\bigl(b(u_{\mathrm{ref}})-b(u_m)\bigr)u_m\,dV_g.
\end{align}

For the standard double-well potential,
\[
W'(r)=r^3-r,
\qquad
W''(r)=3r^2-1\geq-1.
\]
If $M$ is closed, integration by parts produces no boundary term. If
$\partial M\neq\varnothing$, then $u_m=0$ on $\partial M$ and
$W'(0)=0$, so the boundary term again vanishes. Hence
\[
\begin{aligned}
\int_M W'(u_m)\Delta_gu_m\,dV_g
&=
-\int_M W''(u_m)|\nabla_gu_m|_g^2\,dV_g
\\
&\leq
\|\nabla_gu_m\|_{L^2(M)}^2.
\end{aligned}
\]
Using the interpolation estimate
\[
\|\nabla_gz\|_{L^2(M)}^2
\leq
\varepsilon\|\Delta_gz\|_{L^2(M)}^2
+
C_\varepsilon\|z\|_{L^2(M)}^2,
\qquad
z\in H_D^2(M),
\]
we obtain
\begin{equation}
\label{doublewell-m-bound}
\int_M W'(u_m)\Delta_gu_m\,dV_g
\leq
\varepsilon\|\Delta_gu_m\|_{L^2(M)}^2
+
C_\varepsilon\|u_m\|_{L^2(M)}^2.
\end{equation}

The active contribution has a favourable sign. Indeed, since $a$ is
non-decreasing and $a(0)=0$,
\[
a(r)r\geq0
\qquad\text{for every }r\in\mathbb R.
\]
Since $|\nabla_gu_m|_g\geq0$, it follows that
\begin{equation}
\label{classifier-m-sign}
-\nu\int_M
a(\Delta_gu_m)|\nabla_gu_m|_g\Delta_gu_m\,dV_g
\leq0.
\end{equation}

We next estimate the anchoring term. By the coercivity assumption on $b$,
\[
b(r)r\geq c_0|r|^q-c_1,
\qquad r\in\mathbb R,
\]
and the bounds on $\lambda$, we have
\begin{align}
-\int_M\lambda(x)b(u_m)u_m\,dV_g
&\leq
-\lambda_*c_0\|u_m\|_{L^q(M)}^q
+
\lambda^*c_1\operatorname{Vol}_g(M).
\label{anchoring-coercive-part}
\end{align}
For the prescribed reference contribution, H\"older's and Young's inequalities
give
\begin{align}
\int_M\lambda(x)b(u_{\mathrm{ref}})u_m\,dV_g
&\leq
\lambda^*
\|b(u_{\mathrm{ref}})\|_{L^{q'}(M)}
\|u_m\|_{L^q(M)}
\nonumber\\
&\leq
\frac{\lambda_*c_0}{2}
\|u_m\|_{L^q(M)}^q
+
C
\|b(u_{\mathrm{ref}})\|_{L^{q'}(M)}^{q'}.
\label{anchoring-reference-part}
\end{align}
The growth assumption
\[
|b(r)|\leq c_2\bigl(1+|r|^{q-1}\bigr)
\]
implies
\[
|b(r)|^{q'}
\leq
C\bigl(1+|r|^q\bigr),
\qquad
q'=\frac{q}{q-1},
\]
and therefore
\[
\|b(u_{\mathrm{ref}})\|_{L^{q'}(M)}^{q'}
\leq
C\bigl(
1+\|u_{\mathrm{ref}}\|_{L^q(M)}^q
\bigr).
\]
Combining \eqref{anchoring-coercive-part} and
\eqref{anchoring-reference-part}, we obtain
\begin{equation}
\label{anchoring-m-bound}
\begin{aligned}
&\int_M
\lambda(x)\bigl(b(u_{\mathrm{ref}})-b(u_m)\bigr)u_m\,dV_g
\\
&\qquad\leq
-\frac{\lambda_*c_0}{2}
\|u_m\|_{L^q(M)}^q
+
C\bigl(
1+\|u_{\mathrm{ref}}\|_{L^q(M)}^q
\bigr).
\end{aligned}
\end{equation}

Choosing $\varepsilon=\mu/2$ in \eqref{doublewell-m-bound}, and using
\eqref{classifier-m-sign} and \eqref{anchoring-m-bound} in
\eqref{energy-m-start}, we find
\begin{equation}
\label{energy-m-final}
\begin{aligned}
\frac{d}{dt}\|u_m\|_{L^2(M)}^2
&+
\mu\|\Delta_gu_m\|_{L^2(M)}^2
+
\lambda_*c_0\|u_m\|_{L^q(M)}^q
\\
&\leq
C\|u_m\|_{L^2(M)}^2
+
C\bigl(
1+\|u_{\mathrm{ref}}\|_{L^q(M)}^q
\bigr).
\end{aligned}
\end{equation}

Gronwall's inequality, together with
\[
\|P_mu_{\mathrm{in}}\|_{L^2(M)}
\leq
\|u_{\mathrm{in}}\|_{L^2(M)},
\]
now yields
\begin{equation}
\label{uniform-L2-H2-Lq-m}
\begin{aligned}
&
\|u_m\|_{L^\infty(0,T;L^2(M))}^2
+
\|\Delta_gu_m\|_{L^2(Q_T)}^2
+
\|u_m\|_{L^q(Q_T)}^q
\\
&\qquad\leq
C_T
\Bigl(
1
+\|u_{\mathrm{in}}\|_{L^2(M)}^2
+\|u_{\mathrm{ref}}\|_{L^q(M)}^q
\Bigr).
\end{aligned}
\end{equation}
Finally, the elliptic estimate
\[
\|z\|_{H^2(M)}
\leq
C\bigl(
\|\Delta_gz\|_{L^2(M)}
+
\|z\|_{L^2(M)}
\bigr),
\qquad z\in H_D^2(M),
\]
gives
\[
\|u_m\|_{L^2(0,T;H_D^2(M))}^2
\leq
C_T
\Bigl(
1
+\|u_{\mathrm{in}}\|_{L^2(M)}^2
+\|u_{\mathrm{ref}}\|_{L^q(M)}^q
\Bigr).
\]
All constants are independent of $m$. In particular, the coefficient vector
$c^m(t)$ remains bounded on every finite time interval, so the local
Galerkin solution extends to the whole interval $[0,T]$.

\medskip
\noindent
\medskip
\noindent
\textbf{Step 3: estimate of the time derivative.}

Let
\[
\mathcal X:=D\bigl((I+A)^{s/2}\bigr),
\]
where $s\geq2$ is chosen so that
\[
\mathcal X\hookrightarrow H_D^2(M)\cap L^q(M).
\]
For every $\varphi\in\mathcal X$, since
$\partial_tu_m\in\mathcal V_m$, we have
\[
\langle\partial_tu_m,\varphi\rangle_{\mathcal X',\mathcal X}
=
(\partial_tu_m,P_m\varphi)_{L^2(M)}.
\]
Using the Galerkin equation with the test function $P_m\varphi$, we obtain
\begin{align}
\bigl|
\langle\partial_tu_m,\varphi\rangle_{\mathcal X',\mathcal X}
\bigr|
&\leq
\Bigl(
\|W'(u_m)\|_{L^2(M)}
+
\mu\|\Delta_gu_m\|_{L^2(M)}
\Bigr)
\|\Delta_gP_m\varphi\|_{L^2(M)}
\nonumber\\
&\quad+
\nu\|a\|_{L^\infty(\mathbb R)}
\|\nabla_gu_m\|_{L^2(M)}
\|\Delta_gP_m\varphi\|_{L^2(M)}
\nonumber\\
&\quad+
\lambda^*
\Bigl(
\|b(u_{\mathrm{ref}})\|_{L^{q'}(M)}
+
\|b(u_m)\|_{L^{q'}(M)}
\Bigr)
\|P_m\varphi\|_{L^q(M)}.
\label{time-derivative-pointwise-estimate}
\end{align}
Because $P_m$ is a spectral projection, it commutes with $A$ and is
uniformly bounded on $\mathcal X$. Hence
\[
\|\Delta_gP_m\varphi\|_{L^2(M)}
+
\|P_m\varphi\|_{L^q(M)}
\leq
C\|\varphi\|_{\mathcal X},
\]
with a constant independent of $m$. It follows that
\begin{align}
\|\partial_tu_m\|_{\mathcal X'}
&\leq
C\Bigl(
\|W'(u_m)\|_{L^2(M)}
+
\|\Delta_gu_m\|_{L^2(M)}
+
\|\nabla_gu_m\|_{L^2(M)}
\nonumber\\
&\qquad\qquad
+
\|b(u_{\mathrm{ref}})\|_{L^{q'}(M)}
+
\|b(u_m)\|_{L^{q'}(M)}
\Bigr).
\label{time-derivative-Xdual-estimate}
\end{align}

We now show that every term on the right-hand side is bounded in
$L^{q'}(0,T)$, uniformly in $m$.

For the standard double-well potential,
\[
W'(u_m)=u_m^3-u_m,
\]
and therefore
\[
\|W'(u_m)\|_{L^2(M)}
\leq
\|u_m\|_{L^6(M)}^3
+
\|u_m\|_{L^2(M)}.
\]
Since $q>6$ and $M$ has finite volume,
\[
\|u_m\|_{L^6(M)}
\leq
C\|u_m\|_{L^q(M)}.
\]
Moreover,
\[
3q'\leq q,
\qquad
q'=\frac{q}{q-1}.
\]
Consequently,
\begin{align*}
\int_0^T
\|u_m(t)\|_{L^6(M)}^{3q'}\,dt
&\leq
C\int_0^T
\|u_m(t)\|_{L^q(M)}^{3q'}\,dt
\\
&\leq
C_T\left(
1+\|u_m\|_{L^q(Q_T)}^q
\right).
\end{align*}
Together with the uniform
$L^\infty(0,T;L^2(M))$-bound, this shows that
\[
\bigl(W'(u_m)\bigr)_m
\quad\text{is bounded in}\quad
L^{q'}(0,T;L^2(M)).
\]

The families $(\Delta_gu_m)_m$ and $(\nabla_gu_m)_m$ are bounded in
$L^2(Q_T)$. Since $q'<2$ and $T<\infty$, they are therefore also bounded
in
\[
L^{q'}(0,T;L^2(M)).
\]

For the anchoring law, the growth assumption
\[
|b(r)|
\leq
c_2\bigl(1+|r|^{q-1}\bigr)
\]
implies
\[
|b(r)|^{q'}
\leq
C\bigl(1+|r|^q\bigr).
\]
Hence
\[
\|b(u_{\mathrm{ref}})\|_{L^{q'}(M)}^{q'}
\leq
C\Bigl(
1+\|u_{\mathrm{ref}}\|_{L^q(M)}^q
\Bigr),
\]
and
\[
\|b(u_m)\|_{L^{q'}(Q_T)}^{q'}
\leq
C\Bigl(
1+\|u_m\|_{L^q(Q_T)}^q
\Bigr).
\]
The uniform estimate \eqref{uniform-L2-H2-Lq-m} therefore shows that
\[
\bigl(b(u_m)\bigr)_m
\quad\text{is bounded in}\quad
L^{q'}(Q_T).
\]

Combining these bounds with
\eqref{time-derivative-Xdual-estimate}, we conclude that
\[
\sup_{m\in\mathbb N}
\|\partial_tu_m\|_{L^{q'}(0,T;\mathcal X')}
<\infty.
\]
This proves \eqref{galerkin-time-derivative-estimate} and completes the
proof of the proposition. \end{proof}

\subsection{Passage to the limit in the Galerkin approximation}
\label{subsec:passage-m-infty}

In this subsection we let $m\to\infty$ in the Galerkin problem constructed
above. Since $u_m\in\mathcal V_m$, one has $P_mu_m=u_m$. Thus, at the
Galerkin level,
\[
a(\Delta_gu_m)=a(\Delta_gP_mu_m).
\]
The main point is to prove the strong convergence of $\Delta_gu_m$, which
allows us to identify the active term in the limit. We begin by deriving the
preliminary limit equation and establishing the strong time trace of its
solution.

\begin{lemma}[Time trace for the preliminary limit equation]
\label{lem:trace-preliminary-flux}
Let $q>6$, and assume that
\[
u_{\mathrm{in}}\in L^2(M),
\qquad
u_{\mathrm{ref}}\in L^q(M).
\]
Assume the standing hypotheses on $\lambda$, $a$, and $b$, and let $(u_m)$
be the Galerkin solutions constructed in
Proposition~\ref{prop:Galerkin-existence}, so that
\[
u_m(0)=P_mu_{\mathrm{in}}.
\]
Suppose that, after extraction of a subsequence,
\[
u_m\rightharpoonup u
\quad\text{weakly in }L^2(0,T;H_D^2(M)),
\]
\[
u_m\to u
\quad\text{strongly in }L^2(0,T;H_D^1(M))
\quad\text{and a.e. in }Q_T,
\]
and
\[
u_m\rightharpoonup^\ast u
\quad\text{weakly-* in }L^\infty(0,T;L^2(M)).
\]
Assume moreover that
\[
u_m\to u
\quad\text{strongly in }L^r(Q_T)
\quad\text{for every }1\leq r<q,
\]
and that $(u_m)$ is bounded in $L^q(Q_T)$.
Define
\[
\chi_m
:=
a(\Delta_gu_m)|\nabla_gu_m|_g.
\]
Then $(\chi_m)$ is bounded in $L^2(Q_T)$. Consequently, after passing to a
further subsequence,
\[
\chi_m\rightharpoonup\chi
\quad\text{weakly in }L^2(Q_T)
\]
for some $\chi\in L^2(Q_T)$.
Set
\[
\mathfrak m_\chi
:=
W'(u)-\mu\Delta_gu-\nu\chi.
\]
Then
\(
\mathfrak m_\chi\in L^2(Q_T),
\)
and $u$ satisfies the preliminary equation
\begin{equation}
\label{eq:preliminary-limit-equation}
\partial_tu
=
\Delta_g\mathfrak m_\chi
+
\lambda(x)\bigl(b(u_{\mathrm{ref}})-b(u)\bigr)
\end{equation}
in the sense of distributions on $(0,T)\times M$. 
More precisely, for every
\(
\varphi\in C^1([0,T];\mathcal X) \),
\(
\varphi(T)=0,
\)
one has
\begin{align}
\label{eq:preliminary-limit-formulation}
&-\int_0^T
(u,\partial_t\varphi)_{L^2(M)}\,dt
-
(u_{\mathrm{in}},\varphi(0))_{L^2(M)}
\nonumber\\
&\qquad=
\int_0^T\int_M
\mathfrak m_\chi\,\Delta_g\varphi\,dV_g\,dt
+
\int_0^T\int_M
\lambda(x)
\bigl(b(u_{\mathrm{ref}})-b(u)\bigr)
\varphi\,dV_g\,dt.
\end{align}
Consequently, $u$ admits a representative satisfying
\[
u\in C([0,T];L^2(M)),
\qquad
u(0)=u_{\mathrm{in}}
\quad\text{strongly in }L^2(M).
\]
\end{lemma}

\begin{proof}
The boundedness of $a$ gives
\[
\|\chi_m\|_{L^2(Q_T)}
\leq
\|a\|_{L^\infty(\mathbb R)}
\|\nabla_gu_m\|_{L^2(Q_T)}.
\]
The right-hand side is uniformly bounded by the Galerkin estimates.
Therefore, after extraction,
\[
\chi_m\rightharpoonup\chi
\quad\text{weakly in }L^2(Q_T).
\]

Since $q>6$, the assumed strong convergence in $L^r(Q_T)$ for every $r<q$
gives, in particular,
\[
u_m\to u
\quad\text{strongly in }L^6(Q_T).
\]
Hence, for the standard double-well potential,
\[
W'(u_m)=u_m^3-u_m
\longrightarrow
u^3-u=W'(u)
\quad\text{strongly in }L^2(Q_T).
\]
Moreover,
\[
\Delta_gu_m\rightharpoonup\Delta_gu
\quad\text{weakly in }L^2(Q_T),
\]
and, by the definition of $\chi$,
\[
\chi_m\rightharpoonup\chi
\quad\text{weakly in }L^2(Q_T).
\]

We next identify the limit of the anchoring term. Since $u_m\to u$ a.e. in
$Q_T$ and $b$ is continuous,
\[
b(u_m)\to b(u)
\quad\text{a.e. in }Q_T.
\]
The growth assumption
\[
|b(r)|
\leq
c_2\bigl(1+|r|^{q-1}\bigr)
\]
implies
\[
|b(r)|^{q'}
\leq
C\bigl(1+|r|^q\bigr),
\qquad
q'=\frac{q}{q-1}.
\]
Since $(u_m)$ is bounded in $L^q(Q_T)$, it follows that
$(b(u_m))$ is bounded in $L^{q'}(Q_T)$. Therefore,
\[
b(u_m)\rightharpoonup b(u)
\quad\text{weakly in }L^{q'}(Q_T).
\]
Indeed, every weakly convergent subsequence has limit $b(u)$ by the
almost-everywhere convergence, and hence the entire sequence converges weakly
to $b(u)$.
Let
\(
\varphi\in C^1([0,T];\mathcal X)\) ,
\(
\varphi(T)=0.
\)
We test the Galerkin equation with $P_m\varphi$ and integrate in time.
Because $P_m$ is the spectral projection,
\[
P_m\varphi\to\varphi
\quad\text{strongly in }C^1([0,T];\mathcal X).
\]
In particular,
\[
\Delta_gP_m\varphi\to\Delta_g\varphi
\quad\text{strongly in }C([0,T];L^2(M)),
\]
and
\[
P_m\varphi\to\varphi
\quad\text{strongly in }C([0,T];L^q(M)).
\]
Furthermore,
\[
P_mu_{\mathrm{in}}\to u_{\mathrm{in}}
\quad\text{strongly in }L^2(M).
\]
The preceding convergences therefore permit passage to the limit in every
term of the integrated Galerkin equation and yield
\eqref{eq:preliminary-limit-formulation}. Equivalently, $u$ satisfies
\eqref{eq:preliminary-limit-equation}.

We now establish the time regularity. Set
\[
V_1:=H_D^2(M),
\qquad
V_2:=L^q(M),
\qquad
H:=L^2(M).
\]
Then
\[
V_1\cap V_2\hookrightarrow H
\]
densely and continuously. We have
\[
u\in L^2(0,T;V_1)\cap L^q(0,T;V_2).
\]
Moreover,
\[
\mathfrak m_\chi
=
W'(u)-\mu\Delta_gu-\nu\chi
\in L^2(Q_T),
\]
and therefore
\[
\Delta_g\mathfrak m_\chi
\in L^2(0,T;V_1').
\]
Indeed, for every $\psi\in V_1$,
\[
\left|
\left\langle
\Delta_g\mathfrak m_\chi,\psi
\right\rangle_{V_1',V_1}
\right|
=
\left|
\int_M
\mathfrak m_\chi\Delta_g\psi\,dV_g
\right|
\leq
\|\mathfrak m_\chi\|_{L^2(M)}
\|\Delta_g\psi\|_{L^2(M)}.
\]
On the other hand, since
\[
u_{\mathrm{ref}}\in L^q(M),
\qquad
u\in L^q(Q_T),
\]
the growth condition on $b$ gives
\[
b(u_{\mathrm{ref}})\in L^{q'}(M),
\qquad
b(u)\in L^{q'}(Q_T).
\]
Thus
\[
f_\chi
:=
\lambda(x)\bigl(b(u_{\mathrm{ref}})-b(u)\bigr)
\in L^{q'}(0,T;V_2').
\]
Consequently,
\[
\partial_tu
=
\Delta_g\mathfrak m_\chi+f_\chi
\in
L^2(0,T;V_1')
+
L^{q'}(0,T;V_2').
\]

The standard chain rule for this intersection evolution space, obtained by
Steklov regularisation in time, implies that $u$ has a representative in
$C([0,T];H)$ and that, for all $0\leq s\leq t\leq T$,
\begin{align}
\label{eq:preliminary-energy-identity}
\frac12\|u(t)\|_{L^2(M)}^2
-
\frac12\|u(s)\|_{L^2(M)}^2
&=
\int_s^t\int_M
\mathfrak m_\chi\,\Delta_gu\,dV_g\,d\tau
\nonumber\\
&\quad+
\int_s^t\int_M
\lambda(x)
\bigl(b(u_{\mathrm{ref}})-b(u)\bigr)
u\,dV_g\,d\tau.
\end{align}
The two integrals on the right-hand side are finite. Indeed,
\[
\left|
\int_M
\mathfrak m_\chi\Delta_gu\,dV_g
\right|
\leq
\|\mathfrak m_\chi\|_{L^2(M)}
\|\Delta_gu\|_{L^2(M)}
\in L^1(0,T),
\]
while
\begin{align*}
\left|
\int_M
\lambda(x)
\bigl(b(u_{\mathrm{ref}})-b(u)\bigr)u\,dV_g
\right|
&\leq
\|\lambda\|_{L^\infty(M)}
\Bigl(
\|b(u_{\mathrm{ref}})\|_{L^{q'}(M)}
\|u\|_{L^q(M)}
\\
&\qquad\qquad
+
\|b(u)\|_{L^{q'}(M)}
\|u\|_{L^q(M)}
\Bigr),
\end{align*}
and the latter expression belongs to $L^1(0,T)$.

It remains to identify the initial value. The distributional formulation
\eqref{eq:preliminary-limit-formulation} contains the initial datum
$u_{\mathrm{in}}$. On the other hand, integration by parts in time for the
strongly continuous representative gives the same formula with $u(0)$ in
place of $u_{\mathrm{in}}$. Therefore,
\[
(u(0)-u_{\mathrm{in}},\varphi(0))_{L^2(M)}=0
\]
for every admissible test function $\varphi$. Since $\varphi(0)$ is arbitrary
in the dense space $\mathcal X\subset L^2(M)$, we conclude that
\[
u(0)=u_{\mathrm{in}}
\quad\text{strongly in }L^2(M).
\]
The proof is complete.
\end{proof}

We shall need a simultaneous time--space regularization of the limiting
function that is compatible with the comparison argument used below.

\begin{lemma}[Smooth approximation adapted to the comparison argument]
\label{lem:smooth-approximation}
Let $q>6$, let
\[
q'=\frac{q}{q-1},
\]
and set
\[
\mathcal E_T
:=
L^2(0,T;H_D^2(M))
\cap
L^q(Q_T),
\]
and
\[
\mathcal E_T'
:=
L^2(0,T;H_D^2(M)')
+
L^{q'}(Q_T),
\]
endowed with the usual sum norm
\[
\|F\|_{\mathcal E_T'}
:=
\inf_{F=f+g}
\left(
\|f\|_{L^2(0,T;H_D^2(M)')}
+
\|g\|_{L^{q'}(Q_T)}
\right).
\]

Assume that
\[
u\in\mathcal E_T,
\qquad
\partial_tu\in\mathcal E_T',
\]
and that
\[
u\in C([0,T];L^2(M)),
\qquad
u(0)=u_{\mathrm{in}}
\quad\text{in }L^2(M),
\]
for some $u_{\mathrm{in}}\in L^2(M)$.

Then there exists a family $(u^\zeta)_{\zeta>0}$ such that
\[
u^\zeta\in C^\infty([0,T];D(A^k))
\qquad\text{for every }k\in\mathbb N,
\]
and, in particular,
\[
u^\zeta\in C^1([0,T];\mathcal X).
\]
As $\zeta\downarrow0$, one has
\begin{equation}
\label{eq:smooth-approximation-energy}
\|u^\zeta-u\|_{L^2(0,T;H_D^2(M))}
+
\|u^\zeta-u\|_{L^q(Q_T)}
\longrightarrow0,
\end{equation}
\begin{equation}
\label{eq:smooth-approximation-initial}
u^\zeta(0)\longrightarrow u_{\mathrm{in}}
\qquad\text{strongly in }L^2(M),
\end{equation}
and
\begin{equation}
\label{eq:smooth-approximation-time-derivative}
\partial_tu^\zeta
\longrightarrow
\partial_tu
\qquad\text{in }\mathcal E_T'.
\end{equation}
Moreover,
\begin{equation}
\label{eq:smooth-approximation-W}
W'(u^\zeta)\longrightarrow W'(u)
\qquad\text{strongly in }L^2(Q_T),
\end{equation}
\begin{equation}
\label{eq:smooth-approximation-b}
b(u^\zeta)\longrightarrow b(u)
\qquad\text{strongly in }L^{q'}(Q_T),
\end{equation}
and
\begin{equation}
\label{eq:smooth-approximation-pairing}
\int_0^T
\left|
\left\langle
\partial_tu^\zeta(t),
u(t)-u^\zeta(t)
\right\rangle
\right|\,dt
\longrightarrow0.
\end{equation}
In \eqref{eq:smooth-approximation-pairing}, the pairing is understood
through any decomposition
\[
\partial_tu^\zeta=f^\zeta+g^\zeta,
\]
with
\[
f^\zeta\in L^2(0,T;H_D^2(M)'),
\qquad
g^\zeta\in L^{q'}(Q_T).
\]
\end{lemma}

\begin{proof}
Choose $\delta\in(0,T)$. We first extend $u$ from $(0,T)$ to
$(-\delta,T+\delta)$ by reflection:
\[
u_\delta(t)
=
\begin{cases}
u(-t),
&-\delta<t<0,
\\
u(t),
&0\leq t\leq T,
\\
u(2T-t),
&T<t<T+\delta.
\end{cases}
\]
Since
\[
u\in
L^2(0,T;H_D^2(M))
\cap
L^q(Q_T),
\]
the reflected function satisfies
\[
u_\delta
\in
L^2(-\delta,T+\delta;H_D^2(M))
\cap
L^q((-\delta,T+\delta)\times M).
\]

Choose a decomposition
\[
\partial_tu=f+g,
\]
where
\[
f\in L^2(0,T;H_D^2(M)'),
\qquad
g\in L^{q'}(Q_T).
\]
Define the reflected functions
\[
f_\delta(t)
=
\begin{cases}
-f(-t),
&-\delta<t<0,
\\
f(t),
&0<t<T,
\\
-f(2T-t),
&T<t<T+\delta,
\end{cases}
\]
and
\[
g_\delta(t)
=
\begin{cases}
-g(-t),
&-\delta<t<0,
\\
g(t),
&0<t<T,
\\
-g(2T-t),
&T<t<T+\delta.
\end{cases}
\]
Then
\[
f_\delta
\in
L^2(-\delta,T+\delta;H_D^2(M)'),
\]
and
\[
g_\delta
\in
L^{q'}((-\delta,T+\delta)\times M).
\]

Because
\[
u\in C([0,T];L^2(M)),
\]
the reflected function $u_\delta$ is continuous in $L^2(M)$ at the
reflection points $t=0$ and $t=T$. Its distributional time derivative
therefore contains no Dirac masses at these points, and
\[
\partial_tu_\delta=f_\delta+g_\delta
\]
in the sense of distributions on $(-\delta,T+\delta)$ with values in
\[
H_D^2(M)'+L^{q'}(M).
\]

Let $\rho_\zeta$ be a standard one-dimensional mollifier supported in
$(-\zeta,\zeta)$, where $0<\zeta<\delta/2$, and define
\[
\widetilde u^\zeta(t)
:=
(\rho_\zeta*u_\delta)(t),
\qquad
0\leq t\leq T.
\]
Then
\[
\widetilde u^\zeta
\longrightarrow u
\quad\text{in }
L^2(0,T;H_D^2(M))
\cap
L^q(Q_T).
\]
Moreover,
\[
\partial_t\widetilde u^\zeta
=
\rho_\zeta*f_\delta
+
\rho_\zeta*g_\delta,
\]
and
\[
\rho_\zeta*f_\delta
\longrightarrow f
\quad\text{in }L^2(0,T;H_D^2(M)'),
\]
while
\[
\rho_\zeta*g_\delta
\longrightarrow g
\quad\text{in }L^{q'}(Q_T).
\]
Consequently,
\[
\partial_t\widetilde u^\zeta
\longrightarrow
\partial_tu
\quad\text{in }\mathcal E_T'.
\]

The strong continuity of $u$ at $t=0$ and the reflected extension also give
\[
\widetilde u^\zeta(0)
\longrightarrow
u_{\mathrm{in}}
\quad\text{strongly in }L^2(M).
\]

We now regularize in space. Let
\[
S_\zeta:=e^{-\zeta A}
\]
be the heat semigroup generated by the chosen realization of $A=-\Delta_g$,
and set
\[
u^\zeta:=S_\zeta\widetilde u^\zeta.
\]
For every $\zeta>0$, the semigroup $S_\zeta$ maps $L^2(M)$ into
$D(A^k)$ for every $k\in\mathbb N$. Since
$\widetilde u^\zeta$ is smooth in time, it follows that
\[
u^\zeta\in C^\infty([0,T];D(A^k))
\qquad\text{for every }k\in\mathbb N.
\]
In particular,
\[
u^\zeta\in C^1([0,T];\mathcal X).
\]

The semigroup $(S_\zeta)_{\zeta\geq0}$ is strongly continuous and uniformly
bounded for $0<\zeta\leq1$ on $H_D^2(M)$ and on $L^q(M)$. Therefore
\[
u^\zeta-u
=
S_\zeta(\widetilde u^\zeta-u)
+
(S_\zeta u-u),
\]
and hence
\[
\|u^\zeta-u\|_{L^2(0,T;H_D^2(M))}
+
\|u^\zeta-u\|_{L^q(Q_T)}
\longrightarrow0.
\]
This proves \eqref{eq:smooth-approximation-energy}.

At the initial time,
\[
u^\zeta(0)-u_{\mathrm{in}}
=
S_\zeta
\bigl(
\widetilde u^\zeta(0)-u_{\mathrm{in}}
\bigr)
+
(S_\zeta u_{\mathrm{in}}-u_{\mathrm{in}}).
\]
The first term tends to zero because $S_\zeta$ is uniformly bounded on
$L^2(M)$, whereas the second tends to zero by the strong continuity of the
semigroup on $L^2(M)$. Thus
\[
u^\zeta(0)\longrightarrow u_{\mathrm{in}}
\quad\text{strongly in }L^2(M).
\]

Similarly,
\[
\partial_tu^\zeta
=
S_\zeta(\rho_\zeta*f_\delta)
+
S_\zeta(\rho_\zeta*g_\delta).
\]
The semigroup is strongly continuous on $H_D^2(M)'$ and on $L^{q'}(M)$.
Combining this with the convergence of the time mollifications gives
\[
\partial_tu^\zeta
\longrightarrow
\partial_tu
\quad\text{in }\mathcal E_T',
\]
which proves \eqref{eq:smooth-approximation-time-derivative}.

We next establish the nonlinear convergences. Since $q>6$ and $Q_T$ has
finite measure, \eqref{eq:smooth-approximation-energy} implies
\[
u^\zeta\longrightarrow u
\quad\text{strongly in }L^6(Q_T).
\]
For
\[
W'(r)=r^3-r,
\]
we have
\[
|r^3-s^3|
\leq
C\bigl(|r|^2+|s|^2\bigr)|r-s|.
\]
Therefore
\[
\begin{aligned}
\|(u^\zeta)^3-u^3\|_{L^2(Q_T)}
&\leq
C
\left(
\|u^\zeta\|_{L^6(Q_T)}^2
+
\|u\|_{L^6(Q_T)}^2
\right)
\|u^\zeta-u\|_{L^6(Q_T)}
\\
&\longrightarrow0.
\end{aligned}
\]
It follows that
\[
W'(u^\zeta)\longrightarrow W'(u)
\quad\text{strongly in }L^2(Q_T).
\]

We now consider the general anchoring law $b$. By continuity of $b$ and
the strong convergence of $u^\zeta$ in $L^q(Q_T)$,
\[
b(u^\zeta)\longrightarrow b(u)
\quad\text{in measure on }Q_T.
\]
The growth condition
\[
|b(r)|
\leq
c_2\bigl(1+|r|^{q-1}\bigr)
\]
implies
\[
|b(r)|^{q'}
\leq
C\bigl(1+|r|^q\bigr).
\]
Since
\[
u^\zeta\longrightarrow u
\quad\text{strongly in }L^q(Q_T),
\]
the family
\[
\bigl(|b(u^\zeta)|^{q'}\bigr)_{\zeta>0}
\]
is uniformly integrable. The same is true of
\[
|b(u^\zeta)-b(u)|^{q'}.
\]
Vitali's convergence theorem therefore gives
\[
b(u^\zeta)\longrightarrow b(u)
\quad\text{strongly in }L^{q'}(Q_T).
\]

It remains to prove \eqref{eq:smooth-approximation-pairing}. Write
\[
\partial_tu^\zeta
=
f^\zeta+g^\zeta,
\]
where
\[
f^\zeta
:=
S_\zeta(\rho_\zeta*f_\delta),
\qquad
g^\zeta
:=
S_\zeta(\rho_\zeta*g_\delta).
\]
The preceding convergence shows, in particular, that
\[
\sup_{0<\zeta\leq1}
\left(
\|f^\zeta\|_{L^2(0,T;H_D^2(M)')}
+
\|g^\zeta\|_{L^{q'}(Q_T)}
\right)
<\infty.
\]
Consequently,
\begin{align*}
\int_0^T
\left|
\left\langle
\partial_tu^\zeta,
u-u^\zeta
\right\rangle
\right|\,dt
&\leq
\|f^\zeta\|_{L^2(0,T;H_D^2(M)')}
\|u-u^\zeta\|_{L^2(0,T;H_D^2(M))}
\\
&\qquad\quad+
\|g^\zeta\|_{L^{q'}(Q_T)}
\|u-u^\zeta\|_{L^q(Q_T)}.
\end{align*}
Both terms tend to zero by
\eqref{eq:smooth-approximation-energy}. This proves
\eqref{eq:smooth-approximation-pairing} and completes the proof.
\end{proof}

\begin{theorem}[Existence and strong second-order compactness]
\label{thm:existence-weak-solution}
Let $q>6$, and assume that
\[
u_{\mathrm{in}}\in L^2(M),
\qquad
u_{\mathrm{ref}}\in L^q(M).
\]
Assume further that $\lambda$ satisfies \eqref{ass-lambda}, that $a$ satisfies
Assumption~\ref{ass:aclass}, and that $b$ satisfies the standing continuity,
monotonicity, coercivity, and growth assumptions.

Then there exist a subsequence of the Galerkin solutions $(u_m)$, not
relabeled, and a function
\[
u\in
L^\infty(0,T;L^2(M))
\cap
L^2(0,T;H_D^2(M))
\cap
L^q(Q_T),
\qquad
\partial_tu\in L^{q'}(0,T;\mathcal X'),
\]
where
\(
q'=\frac{q}{q-1},
\)
such that
\begin{align}
u_m
&\rightharpoonup u
&&\text{weakly in }L^2(0,T;H_D^2(M)),
\label{conv-m-weak-H2}
\\
u_m
&\overset{*}{\rightharpoonup}u
&&\text{weakly-* in }L^\infty(0,T;L^2(M)),
\label{conv-m-weakstar-L2}
\\
u_m
&\longrightarrow u
&&\text{strongly in }L^2(0,T;H^1(M)),
\label{conv-m-strong-H1}
\\
u_m
&\longrightarrow u
&&\text{strongly in }L^6(Q_T),
\label{conv-m-strong-L6}
\\
\Delta_gu_m
&\longrightarrow\Delta_gu
&&\text{strongly in }L^2(Q_T).
\label{conv-m-strong-Delta}
\end{align}
Consequently,
\[
u_m\longrightarrow u
\qquad\text{strongly in }L^2(0,T;H_D^2(M)).
\]
Moreover,
\[
a(\Delta_gu_m)|\nabla_gu_m|_g
\longrightarrow
a(\Delta_gu)|\nabla_gu|_g
\qquad\text{strongly in }L^2(Q_T).
\]
The limit $u$ satisfies, for every
\[
\varphi\in C^1([0,T];\mathcal X)
\]
and every $t\in[0,T]$,
\begin{align}
\label{weak-form-limit}
&\int_M u(t)\varphi(t)\,dV_g
-
\int_M u_{\mathrm{in}}\varphi(0)\,dV_g
-
\int_0^t\int_M
u\,\partial_t\varphi\,dV_g\,ds
\nonumber\\
&\qquad=
\int_0^t\int_M
\Bigl(
W'(u)
-\mu\Delta_gu
-\nu a(\Delta_gu)|\nabla_gu|_g
\Bigr)
\Delta_g\varphi\,dV_g\,ds
\nonumber\\
&\qquad\quad+
\int_0^t\int_M
\lambda(x)
\bigl(
b(u_{\mathrm{ref}})-b(u)
\bigr)
\varphi\,dV_g\,ds.
\end{align}
Furthermore,
\[
u\in C([0,T];L^2(M)),
\qquad
u(0)=u_{\mathrm{in}}
\quad\text{strongly in }L^2(M).
\]
Hence $u$ is a weak solution of
\eqref{eq:active-CH-manifold}--\eqref{bc-navier}.
\end{theorem}

\begin{proof}
The uniform estimates of
Proposition~\ref{prop:Galerkin-existence} give
\[
(u_m)
\quad\text{bounded in}\quad
L^\infty(0,T;L^2(M))
\cap
L^2(0,T;H_D^2(M))
\cap
L^q(Q_T),
\]
and
\[
(\partial_tu_m)
\quad\text{bounded in}\quad
L^{q'}(0,T;\mathcal X').
\]
Since $q'>1$ and $\mathcal X'$ is reflexive, after extracting a subsequence
there exist functions $u$ and $\eta$ such that
\[
u_m\rightharpoonup u
\quad\text{weakly in }L^2(0,T;H_D^2(M)),
\]
\[
u_m\overset{*}{\rightharpoonup}u
\quad\text{weakly-* in }L^\infty(0,T;L^2(M)),
\]
and
\[
\partial_tu_m\rightharpoonup\eta
\quad\text{weakly in }L^{q'}(0,T;\mathcal X').
\]

We identify $\eta$ as the distributional time derivative of $u$. Let
\(
\psi\in C_c^\infty((0,T);\mathcal X).
\)
Then
\[
\begin{aligned}
\int_0^T
\langle\eta,\psi\rangle_{\mathcal X',\mathcal X}\,dt
&=
\lim_{m\to\infty}
\int_0^T
\langle\partial_tu_m,\psi\rangle_{\mathcal X',\mathcal X}\,dt
=
-\lim_{m\to\infty}
\int_0^T
(u_m,\partial_t\psi)_{L^2(M)}\,dt
\\
&=
-\int_0^T
(u,\partial_t\psi)_{L^2(M)}\,dt.
\end{aligned}
\]
Thus
\[
\eta=\partial_tu
\]
in the sense of distributions, and therefore
\[
\partial_tu\in L^{q'}(0,T;\mathcal X').
\]
Since
\[
H_D^2(M)\Subset H_D^1(M)\hookrightarrow\mathcal X',
\]
the Aubin--Lions lemma gives, after passing to a further subsequence,
\[
u_m\longrightarrow u
\qquad\text{strongly in }L^2(0,T;H_D^1(M)).
\]
In particular,
\[
u_m\longrightarrow u
\qquad\text{strongly in }L^2(Q_T),
\]
and, after extraction,
\[
u_m(t,x)\longrightarrow u(t,x)
\qquad\text{for a.e. }(t,x)\in Q_T.
\]
Interpolation between the strong $L^2(Q_T)$-convergence and the uniform
$L^q(Q_T)$-bound gives
\[
u_m\longrightarrow u
\qquad\text{strongly in }L^r(Q_T)
\quad\text{for every }1\leq r<q.
\]
In particular, since $q>6$,
\[
u_m\longrightarrow u
\qquad\text{strongly in }L^6(Q_T).
\]
Consequently, for the standard double-well potential,
\[
W'(u_m)=u_m^3-u_m
\longrightarrow
u^3-u=W'(u)
\qquad\text{strongly in }L^2(Q_T).
\]
We next consider the anchoring term. Since $u_m\to u$ a.e. in $Q_T$ and
$b$ is continuous,
\[
b(u_m)\longrightarrow b(u)
\qquad\text{a.e. in }Q_T.
\]
The growth assumption
\[
|b(r)|
\leq
c_2\bigl(1+|r|^{q-1}\bigr)
\]
implies
\[
|b(r)|^{q'}
\leq
C\bigl(1+|r|^q\bigr).
\]
The uniform $L^q(Q_T)$-bound for $(u_m)$ therefore shows that
$(b(u_m))$ is bounded in $L^{q'}(Q_T)$. It follows that
\[
b(u_m)\rightharpoonup b(u)
\qquad\text{weakly in }L^{q'}(Q_T).
\]We introduce the active product
\[
\chi_m:=a(\Delta_g u_m)|\nabla_g u_m|_g.
\]Since $a$ is bounded and $(\nabla_gu_m)$ is bounded in $L^2(Q_T)$,
the sequence $(\chi_m)$ is bounded in $L^2(Q_T)$. Hence, after passing to
a further subsequence,
\[
\chi_m\rightharpoonup\chi
\qquad\text{weakly in }L^2(Q_T)
\]
for some $\chi\in L^2(Q_T)$.

The convergences obtained above are precisely the hypotheses of
Lemma~\ref{lem:trace-preliminary-flux}. Therefore $u$ satisfies the
preliminary equation
\[
\partial_tu
=
\Delta_g
\bigl(
W'(u)-\mu\Delta_gu-\nu\chi
\bigr)
+
\lambda(x)
\bigl(
b(u_{\mathrm{ref}})-b(u)
\bigr)
\]
in the sense of distributions on $(0,T)\times M$. In particular,
\begin{equation}
\label{u-reg}
u\in C([0,T];L^2(M)),
\qquad
u(0)=u_{\mathrm{in}}
\quad\text{strongly in }L^2(M),
\end{equation}
and
\[
\partial_tu
\in
L^2(0,T;H_D^2(M)')
+
L^{q'}(Q_T).
\]

Thus the hypotheses of
Lemma~\ref{lem:smooth-approximation} are satisfied. It remains to prove the
strong convergence of $\Delta_gu_m$ and to identify the preliminary active product
$\chi$. We split the proof into five steps.

\medskip
\noindent
\medskip
\noindent
\textbf{Step 1: regularized comparison functions.}

By \eqref{u-reg}, we have
\[
u\in C([0,T];L^2(M)),
\qquad
u(0)=u_{\mathrm{in}},
\]
and, moreover,
\[
u\in
L^2(0,T;H_D^2(M))
\cap
L^q(Q_T),
\qquad
\partial_tu\in
L^2(0,T;H_D^2(M)')
+
L^{q'}(Q_T).
\]
We may therefore apply Lemma~\ref{lem:smooth-approximation}. It yields a
family
\[
u^\zeta\in C^1([0,T];\mathcal X),
\qquad
\zeta>0,
\]
such that, as $\zeta\downarrow0$,
\begin{align}
u^\zeta
&\longrightarrow u
&&\text{strongly in }
L^2(0,T;H_D^2(M))
\cap
L^q(Q_T),
\label{approx-space}
\\
u^\zeta(0)
&\longrightarrow u_{\mathrm{in}}
&&\text{strongly in }L^2(M),
\label{approx-initial}
\\
W'(u^\zeta)
&\longrightarrow W'(u)
&&\text{strongly in }L^2(Q_T),
\label{approx-potential}
\\
b(u^\zeta)
&\longrightarrow b(u)
&&\text{strongly in }L^{q'}(Q_T).
\label{approx-anchor}
\end{align}
In addition,
\begin{equation}
\label{approx-zeta}
\int_0^T
\left|
\left\langle
\partial_tu^\zeta(t),
u(t)-u^\zeta(t)
\right\rangle
\right|\,dt
\longrightarrow0.
\end{equation}

Fix $\zeta>0$. Since the approximation supplied by
Lemma~\ref{lem:smooth-approximation} is smooth with values in
$D(A^k)$ for every $k\in\mathbb N$, both
\[
u^\zeta,\ \partial_tu^\zeta
\in C([0,T];\mathcal X).
\]
The spectral projections $P_m$ are uniformly bounded on $\mathcal X$ and
converge strongly to the identity on $\mathcal X$. Consequently,
\[
P_mu^\zeta
\longrightarrow
u^\zeta
\qquad\text{strongly in }C^1([0,T];\mathcal X)
\]
as $m\to\infty$.

For fixed $m\in\mathbb N$ and $\zeta>0$, define
\[
w_m^\zeta
:=
u_m-P_mu^\zeta.
\]
Since
\[
u_m\in\mathcal V_m,
\qquad
P_mu^\zeta\in\mathcal V_m,
\]
we have
\[
w_m^\zeta\in\mathcal V_m.
\]
Thus $w_m^\zeta$ is an admissible test function in the Galerkin equation.

\medskip

\noindent

\medskip
\noindent
\textbf{Step 2: one-sided comparison inequality.}

Fix $\zeta>0$ and $t\in[0,T]$. Testing the Galerkin equation with
\[
w_m^\zeta=u_m-P_mu^\zeta\in\mathcal V_m
\]
and integrating over $(0,t)$ gives
\begin{align}
\label{comparison-single-1}
\int_0^t
(\partial_su_m,w_m^\zeta)_{L^2(M)}\,ds
&=
\int_0^t\int_M
W'(u_m)\Delta_gw_m^\zeta\,dV_g\,ds
\nonumber\\
&\quad
-\mu\int_0^t\int_M
\Delta_gu_m\,\Delta_gw_m^\zeta\,dV_g\,ds
\nonumber\\
&\quad
-\nu\int_0^t\int_M
a(\Delta_gu_m)|\nabla_gu_m|_g
\Delta_gw_m^\zeta\,dV_g\,ds
\nonumber\\
&\quad
+\int_0^t\int_M
\lambda(x)
\bigl(
b(u_{\mathrm{ref}})-b(u_m)
\bigr)
w_m^\zeta\,dV_g\,ds .
\end{align}

Since
\[
\Delta_gw_m^\zeta
=
\Delta_gu_m-\Delta_gP_mu^\zeta,
\]
we have
\[
\Delta_gu_m
=
\Delta_gw_m^\zeta+\Delta_gP_mu^\zeta.
\]
Consequently,
\begin{align}
-\mu\int_0^t\int_M
\Delta_gu_m\,\Delta_gw_m^\zeta\,dV_g\,ds
&=
-\mu\int_0^t\int_M
|\Delta_gw_m^\zeta|^2\,dV_g\,ds
\nonumber\\
&\quad
-\mu\int_0^t\int_M
\Delta_gP_mu^\zeta\,\Delta_gw_m^\zeta\,dV_g\,ds.
\label{comparison-linear-term}
\end{align}

We next use the monotonicity of the classifier. Adding and subtracting
$a(\Delta_gP_mu^\zeta)$, we obtain
\begin{align*}
&
-\nu
a(\Delta_gu_m)|\nabla_gu_m|_g\Delta_gw_m^\zeta
\\
&\quad=
-\nu
\bigl(
a(\Delta_gu_m)-a(\Delta_gP_mu^\zeta)
\bigr)
|\nabla_gu_m|_g
\bigl(
\Delta_gu_m-\Delta_gP_mu^\zeta
\bigr)
\\
&\qquad
-\nu
a(\Delta_gP_mu^\zeta)
|\nabla_gu_m|_g
\Delta_gw_m^\zeta.
\end{align*}
Since $a$ is non-decreasing,
\[
\bigl(
a(\Delta_gu_m)-a(\Delta_gP_mu^\zeta)
\bigr)
\bigl(
\Delta_gu_m-\Delta_gP_mu^\zeta
\bigr)
\geq0
\]
pointwise, and $|\nabla_gu_m|_g\geq0$. Therefore,
\begin{align}
\label{comparison-active-term}
&-\nu\int_0^t\int_M
a(\Delta_gu_m)|\nabla_gu_m|_g
\Delta_gw_m^\zeta\,dV_g\,ds
\nonumber\\
&\qquad\leq
-\nu\int_0^t\int_M
a(\Delta_gP_mu^\zeta)|\nabla_gu_m|_g
\Delta_gw_m^\zeta\,dV_g\,ds.
\end{align}

The monotonicity of the anchoring law yields a similar one-sided estimate.
Indeed, since $\lambda\geq0$ and
\[
\bigl(
b(u_m)-b(P_mu^\zeta)
\bigr)
\bigl(
u_m-P_mu^\zeta
\bigr)
\geq0,
\]
we have
\begin{align}
\label{comparison-anchoring-term}
&\int_0^t\int_M
\lambda(x)
\bigl(
b(u_{\mathrm{ref}})-b(u_m)
\bigr)
w_m^\zeta\,dV_g\,ds
\nonumber\\
&\qquad\leq
\int_0^t\int_M
\lambda(x)
\bigl(
b(u_{\mathrm{ref}})-b(P_mu^\zeta)
\bigr)
w_m^\zeta\,dV_g\,ds.
\end{align}

Combining
\eqref{comparison-single-1}--\eqref{comparison-anchoring-term}, we obtain
\begin{align}
\label{comparison-single-2}
&\int_0^t
(\partial_su_m,w_m^\zeta)_{L^2(M)}\,ds
+
\mu\int_0^t\int_M
|\Delta_gw_m^\zeta|^2\,dV_g\,ds
\nonumber\\
&\qquad\leq
\int_0^t\int_M
W'(u_m)\Delta_gw_m^\zeta\,dV_g\,ds
\nonumber\\
&\qquad\quad
-\nu\int_0^t\int_M
a(\Delta_gP_mu^\zeta)|\nabla_gu_m|_g
\Delta_gw_m^\zeta\,dV_g\,ds
\nonumber\\
&\qquad\quad
-\mu\int_0^t\int_M
\Delta_gP_mu^\zeta\,\Delta_gw_m^\zeta\,dV_g\,ds
\nonumber\\
&\qquad\quad
+\int_0^t\int_M
\lambda(x)
\bigl(
b(u_{\mathrm{ref}})-b(P_mu^\zeta)
\bigr)
w_m^\zeta\,dV_g\,ds.
\end{align}

We now rewrite the time-derivative term. Since
\[
\partial_sw_m^\zeta
=
\partial_su_m-\partial_sP_mu^\zeta,
\]
we have
\begin{align}
\label{time-parts-single}
\int_0^t
(\partial_su_m,w_m^\zeta)_{L^2(M)}\,ds
&=
\frac12
\|w_m^\zeta(t)\|_{L^2(M)}^2
-
\frac12
\|w_m^\zeta(0)\|_{L^2(M)}^2
\nonumber\\
&\quad
+
\int_0^t
(\partial_sP_mu^\zeta,w_m^\zeta)_{L^2(M)}\,ds.
\end{align}
By the initial condition for the Galerkin approximation,
\[
w_m^\zeta(0)
=
P_mu_{\mathrm{in}}-P_mu^\zeta(0).
\]
Therefore,
\begin{align}
\label{time-parts-single-expanded}
\int_0^t
(\partial_su_m,w_m^\zeta)_{L^2(M)}\,ds
&=
\frac12
\|w_m^\zeta(t)\|_{L^2(M)}^2
-
\frac12
\|P_mu_{\mathrm{in}}-P_mu^\zeta(0)\|_{L^2(M)}^2
\nonumber\\
&\quad
+
\int_0^t
(\partial_sP_mu^\zeta,w_m^\zeta)_{L^2(M)}\,ds.
\end{align}

Substituting \eqref{time-parts-single-expanded} into
\eqref{comparison-single-2} and dropping the non-negative terminal term
\[
\frac12\|w_m^\zeta(t)\|_{L^2(M)}^2,
\]
we arrive at
\begin{align}
\label{comparison-single-3}
&\mu\int_0^t\int_M
|\Delta_gw_m^\zeta|^2\,dV_g\,ds
\nonumber\\
&\qquad\leq
\frac12
\|P_mu_{\mathrm{in}}-P_mu^\zeta(0)\|_{L^2(M)}^2
-
\int_0^t
(\partial_sP_mu^\zeta,w_m^\zeta)_{L^2(M)}\,ds
\nonumber\\
&\qquad\quad
+
\int_0^t\int_M
W'(u_m)\Delta_gw_m^\zeta\,dV_g\,ds
\nonumber\\
&\qquad\quad
-\nu\int_0^t\int_M
a(\Delta_gP_mu^\zeta)|\nabla_gu_m|_g
\Delta_gw_m^\zeta\,dV_g\,ds
\nonumber\\
&\qquad\quad
-\mu\int_0^t\int_M
\Delta_gP_mu^\zeta\,\Delta_gw_m^\zeta\,dV_g\,ds
\nonumber\\
&\qquad\quad
+
\int_0^t\int_M
\lambda(x)
\bigl(
b(u_{\mathrm{ref}})-b(P_mu^\zeta)
\bigr)
w_m^\zeta\,dV_g\,ds.
\end{align}

\medskip
\noindent

\medskip
\noindent
\textbf{Step 3: passage to the limit as $m\to\infty$ for fixed $\zeta$.}

Fix $\zeta>0$. From the strong convergence of $u_m$ in
$L^2(0,T;H^1(M))$ and the convergence
\[
P_mu^\zeta\longrightarrow u^\zeta
\qquad\text{strongly in }C^1([0,T];\mathcal X),
\]
we obtain
\begin{equation}
\label{comparison-w-strong-H1}
w_m^\zeta
=
u_m-P_mu^\zeta
\longrightarrow
u-u^\zeta
\qquad\text{strongly in }L^2(0,T;H^1(M)).
\end{equation}
Moreover,
\begin{equation}
\label{comparison-w-weak-H2}
\Delta_gw_m^\zeta
=
\Delta_gu_m-\Delta_gP_mu^\zeta
\rightharpoonup
\Delta_g(u-u^\zeta)
\qquad\text{weakly in }L^2(Q_T).
\end{equation}
For fixed $\zeta$, we also have
\begin{align}
P_mu^\zeta
&\longrightarrow u^\zeta
&&\text{strongly in }L^2(0,T;H_D^2(M))
\cap L^q(Q_T),
\label{comparison-projection-space}
\\
\partial_tP_mu^\zeta
&\longrightarrow\partial_tu^\zeta
&&\text{strongly in }L^2(Q_T).
\label{comparison-projection-time}
\end{align}

We next identify the limit of the coefficient in the active term. Since
\[
\Delta_gP_mu^\zeta
\longrightarrow
\Delta_gu^\zeta
\qquad\text{strongly in }L^2(Q_T),
\]
the continuity and boundedness of $a$ imply
\[
a(\Delta_gP_mu^\zeta)
\longrightarrow
a(\Delta_gu^\zeta)
\qquad\text{strongly in }L^p(Q_T)
\]
for every finite $p\geq1$. Together with
\[
\nabla_gu_m\longrightarrow\nabla_gu
\qquad\text{strongly in }L^2(Q_T),
\]
this gives
\begin{equation}
\label{comparison-active-coefficient-limit}
a(\Delta_gP_mu^\zeta)|\nabla_gu_m|_g
\longrightarrow
a(\Delta_gu^\zeta)|\nabla_gu|_g
\qquad\text{strongly in }L^2(Q_T).
\end{equation}
Indeed,
\begin{align*}
&
\left\|
a(\Delta_gP_mu^\zeta)|\nabla_gu_m|_g
-
a(\Delta_gu^\zeta)|\nabla_gu|_g
\right\|_{L^2(Q_T)}
\\
&\qquad\leq
\|a\|_{L^\infty}
\|\nabla_gu_m-\nabla_gu\|_{L^2(Q_T)}
\\
&\qquad\quad+
\left\|
\bigl(
a(\Delta_gP_mu^\zeta)-a(\Delta_gu^\zeta)
\bigr)
|\nabla_gu|_g
\right\|_{L^2(Q_T)},
\end{align*}
and the second term tends to zero by dominated convergence.

For the anchoring term, the convergence
\[
P_mu^\zeta\longrightarrow u^\zeta
\qquad\text{strongly in }L^q(Q_T)
\]
and the continuity and growth assumptions on $b$ imply
\begin{equation}
\label{comparison-anchor-coefficient-limit}
b(P_mu^\zeta)
\longrightarrow
b(u^\zeta)
\qquad\text{strongly in }L^{q'}(Q_T).
\end{equation}
Indeed, the convergence holds in measure, while
\[
|b(r)|^{q'}
\leq
C\bigl(1+|r|^q\bigr),
\]
so the conclusion follows from Vitali's theorem.

Furthermore, $(w_m^\zeta)$ is bounded in $L^q(Q_T)$ and converges almost
everywhere to $u-u^\zeta$. Consequently,
\begin{equation}
\label{comparison-w-weak-Lq}
w_m^\zeta
\rightharpoonup
u-u^\zeta
\qquad\text{weakly in }L^q(Q_T).
\end{equation}
Combining
\eqref{comparison-anchor-coefficient-limit} and
\eqref{comparison-w-weak-Lq}, we obtain, for every $t\in[0,T]$,
\begin{align}
\label{anchoring-limit}
&\int_0^t\int_M
\lambda(x)
\bigl(
b(u_{\mathrm{ref}})-b(P_mu^\zeta)
\bigr)
w_m^\zeta\,dV_g\,ds
\nonumber\\
&\qquad\longrightarrow
\int_0^t\int_M
\lambda(x)
\bigl(
b(u_{\mathrm{ref}})-b(u^\zeta)
\bigr)
(u-u^\zeta)\,dV_g\,ds.
\end{align}

We may now take the limit superior in
\eqref{comparison-single-3}. The initial term satisfies
\[
P_mu_{\mathrm{in}}-P_mu^\zeta(0)
\longrightarrow
u_{\mathrm{in}}-u^\zeta(0)
\qquad\text{strongly in }L^2(M).
\]
Moreover, by
\eqref{comparison-projection-time} and
\eqref{comparison-w-strong-H1},
\[
\int_0^t
(\partial_sP_mu^\zeta,w_m^\zeta)_{L^2(M)}\,ds
\longrightarrow
\int_0^t
\left\langle
\partial_su^\zeta,u-u^\zeta
\right\rangle\,ds.
\]
The strong convergence of $W'(u_m)$ in $L^2(Q_T)$, together with
\eqref{comparison-w-weak-H2}, gives
\[
\int_0^t\int_M
W'(u_m)\Delta_gw_m^\zeta\,dV_g\,ds
\longrightarrow
\int_0^t\int_M
W'(u)\Delta_g(u-u^\zeta)\,dV_g\,ds.
\]
Similarly, \eqref{comparison-active-coefficient-limit} and
\eqref{comparison-w-weak-H2} yield
\begin{align*}
&\int_0^t\int_M
a(\Delta_gP_mu^\zeta)|\nabla_gu_m|_g
\Delta_gw_m^\zeta\,dV_g\,ds
\\
&\qquad\longrightarrow
\int_0^t\int_M
a(\Delta_gu^\zeta)|\nabla_gu|_g
\Delta_g(u-u^\zeta)\,dV_g\,ds.
\end{align*}
Finally,
\[
\int_0^t\int_M
\Delta_gP_mu^\zeta\Delta_gw_m^\zeta\,dV_g\,ds
\longrightarrow
\int_0^t\int_M
\Delta_gu^\zeta\Delta_g(u-u^\zeta)\,dV_g\,ds.
\]

Consequently,
\begin{equation}
\label{comparison-single-4}
\mu
\limsup_{m\to\infty}
\int_0^t\int_M
|\Delta_gw_m^\zeta|^2\,dV_g\,ds
\leq
r_\zeta(t),
\end{equation}
where
\begin{align}
\label{rzeta-single}
r_\zeta(t)
:={}&
\frac12
\|u_{\mathrm{in}}-u^\zeta(0)\|_{L^2(M)}^2
\nonumber\\
&+
\left|
\int_0^t
\left\langle
\partial_su^\zeta,
u-u^\zeta
\right\rangle ds
\right|
\nonumber\\
&+
\left|
\int_0^t\int_M
W'(u)\Delta_g(u-u^\zeta)\,dV_g\,ds
\right|
\nonumber\\
&+
\nu
\left|
\int_0^t\int_M
a(\Delta_gu^\zeta)|\nabla_gu|_g
\Delta_g(u-u^\zeta)\,dV_g\,ds
\right|
\nonumber\\
&+
\mu
\left|
\int_0^t\int_M
\Delta_gu^\zeta\Delta_g(u-u^\zeta)\,dV_g\,ds
\right|
\nonumber\\
&+
\left|
\int_0^t\int_M
\lambda(x)
\bigl(
b(u_{\mathrm{ref}})-b(u^\zeta)
\bigr)
(u-u^\zeta)\,dV_g\,ds
\right|.
\end{align}

We finally show that the remainder vanishes as $\zeta\downarrow0$. The
first two terms tend to zero by
\eqref{approx-initial} and \eqref{approx-zeta}. Moreover,
\[
W'(u)\in L^2(Q_T),
\qquad
\Delta_g(u-u^\zeta)\longrightarrow0
\quad\text{in }L^2(Q_T),
\]
and therefore the double-well term tends to zero. Since $a$ is bounded,
\[
\begin{aligned}
&
\left|
\int_0^t\int_M
a(\Delta_gu^\zeta)|\nabla_gu|_g
\Delta_g(u-u^\zeta)\,dV_g\,ds
\right|
\\
&\qquad\leq
\|a\|_{L^\infty}
\|\nabla_gu\|_{L^2(Q_T)}
\|\Delta_g(u-u^\zeta)\|_{L^2(Q_T)}
\longrightarrow0.
\end{aligned}
\]
The linear fourth-order term tends to zero because
$(\Delta_gu^\zeta)$ is bounded in $L^2(Q_T)$ and
\[
\Delta_g(u-u^\zeta)\longrightarrow0
\qquad\text{in }L^2(Q_T).
\]

For the anchoring term, write
\[
b(u_{\mathrm{ref}})-b(u^\zeta)
=
b(u_{\mathrm{ref}})-b(u)
+
b(u)-b(u^\zeta).
\]
Then
\begin{align*}
&
\left|
\int_0^t\int_M
\lambda(x)
\bigl(
b(u_{\mathrm{ref}})-b(u^\zeta)
\bigr)
(u-u^\zeta)\,dV_g\,ds
\right|
\\
&\qquad\leq
\|\lambda\|_{L^\infty}
\Bigl(
\|b(u_{\mathrm{ref}})-b(u)\|_{L^{q'}(Q_T)}
+
\|b(u)-b(u^\zeta)\|_{L^{q'}(Q_T)}
\Bigr)
\|u-u^\zeta\|_{L^q(Q_T)},
\end{align*}
which tends to zero by
\eqref{approx-space} and \eqref{approx-anchor}. Hence
\begin{equation}
\label{rzeta-vanishing}
\sup_{t\in[0,T]}r_\zeta(t)
\longrightarrow0
\qquad\text{as }\zeta\downarrow0.
\end{equation}

\medskip
\noindent

\medskip
\noindent
\textbf{Step 4: strong second-order compactness.}

For \(t\in(0,T]\), set
\[
Q_t:=(0,t)\times M.
\]
For every fixed \(\zeta>0\), we have
\[
\Delta_g(u_m-u)
=
\Delta_gw_m^\zeta
+
\Delta_g(P_mu^\zeta-u^\zeta)
+
\Delta_g(u^\zeta-u).
\]
Hence
\begin{align}
\label{laplacian-comparison-triangle}
\|\Delta_g(u_m-u)\|_{L^2(Q_t)}
&\leq
\|\Delta_gw_m^\zeta\|_{L^2(Q_t)}
+
\|\Delta_g(P_mu^\zeta-u^\zeta)\|_{L^2(Q_t)}
\nonumber\\
&\quad+
\|\Delta_g(u^\zeta-u)\|_{L^2(Q_t)}.
\end{align}

For fixed \(\zeta>0\),
\[
P_mu^\zeta\longrightarrow u^\zeta
\qquad
\text{strongly in }L^2(0,T;H_D^2(M)),
\]
and therefore
\[
\|\Delta_g(P_mu^\zeta-u^\zeta)\|_{L^2(Q_t)}
\longrightarrow0
\qquad\text{as }m\to\infty.
\]
Moreover, \eqref{comparison-single-4} gives
\[
\limsup_{m\to\infty}
\|\Delta_gw_m^\zeta\|_{L^2(Q_t)}^2
\leq
\mu^{-1}r_\zeta(t).
\]
Taking the limit superior in
\eqref{laplacian-comparison-triangle}, we obtain
\begin{equation}
\label{laplacian-comparison-limsup}
\limsup_{m\to\infty}
\|\Delta_g(u_m-u)\|_{L^2(Q_t)}
\leq
\mu^{-1/2}r_\zeta(t)^{1/2}
+
\|\Delta_g(u^\zeta-u)\|_{L^2(Q_t)}.
\end{equation}

By \eqref{rzeta-vanishing},
\[
\sup_{t\in[0,T]}r_\zeta(t)
\longrightarrow0
\qquad\text{as }\zeta\downarrow0,
\]
whereas
\[
u^\zeta\longrightarrow u
\qquad
\text{strongly in }L^2(0,T;H_D^2(M)).
\]
Letting \(\zeta\downarrow0\) in
\eqref{laplacian-comparison-limsup}, we conclude that
\[
\Delta_gu_m\longrightarrow\Delta_gu
\qquad
\text{strongly in }L^2(Q_t)
\]
for every \(t\in(0,T]\). In particular, taking \(t=T\), we obtain
\begin{equation}
\label{strong-laplacian-final}
\Delta_gu_m\longrightarrow\Delta_gu
\qquad
\text{strongly in }L^2(Q_T).
\end{equation}

Since we already know that
\[
u_m\longrightarrow u
\qquad
\text{strongly in }L^2(Q_T),
\]
the elliptic estimate
\[
\|z\|_{H^2(M)}
\leq
C\Bigl(
\|\Delta_gz\|_{L^2(M)}
+
\|z\|_{L^2(M)}
\Bigr),
\qquad
z\in H_D^2(M),
\]
gives
\begin{equation}
\label{strong-H2-final}
u_m\longrightarrow u
\qquad
\text{strongly in }L^2(0,T;H_D^2(M)).
\end{equation}

\medskip
\noindent
\textbf{Step 5: passage to the limit in the weak formulation.}

Let
\[
\varphi\in C^1([0,T];\mathcal X).
\]
Since the Galerkin equation may only be tested with functions belonging to
\(\mathcal V_m\), we use \(P_m\varphi\) as a time-dependent test function.
After integration over \((0,t)\), we obtain
\begin{align}
\label{weak-form-m-integrated}
&\int_M u_m(t)P_m\varphi(t)\,dV_g
-
\int_M P_mu_{\mathrm{in}}\,P_m\varphi(0)\,dV_g
-
\int_0^t\int_M
u_m\,\partial_s(P_m\varphi)\,dV_g\,ds
\nonumber\\
&\qquad=
\int_0^t\int_M
\Bigl(
W'(u_m)
-\mu\Delta_gu_m
-\nu a(\Delta_gu_m)|\nabla_gu_m|_g
\Bigr)
\Delta_g(P_m\varphi)\,dV_g\,ds
\nonumber\\
&\qquad\quad+
\int_0^t\int_M
\lambda(x)
\bigl(
b(u_{\mathrm{ref}})-b(u_m)
\bigr)
P_m\varphi\,dV_g\,ds.
\end{align}

For fixed \(\varphi\in C^1([0,T];\mathcal X)\), the spectral projections
satisfy
\[
P_m\varphi\longrightarrow\varphi
\qquad
\text{strongly in }C^1([0,T];\mathcal X).
\]
In particular,
\[
P_m\varphi\longrightarrow\varphi
\quad\text{in }\;
C([0,T];H_D^2(M))
\cap
C([0,T];L^q(M)),
\]
\[
\Delta_g(P_m\varphi)
\longrightarrow
\Delta_g\varphi
\qquad
\text{in }C([0,T];L^2(M)),
\]
and
\[
\partial_t(P_m\varphi)
\longrightarrow
\partial_t\varphi
\qquad
\text{in }C([0,T];L^2(M)).
\]
Moreover,
\[
P_mu_{\mathrm{in}}
\longrightarrow
u_{\mathrm{in}}
\qquad
\text{strongly in }L^2(M).
\]

To pass to the endpoint term at time \(t\), we use compactness in time.
Since
\[
L^2(M)\Subset\mathcal X',
\]
the sequence \((u_m)\) is bounded in
\(L^\infty(0,T;L^2(M))\), and
\((\partial_tu_m)\) is bounded in
\(L^{q'}(0,T;\mathcal X')\) with \(q'>1\), the
Aubin--Lions--Simon compactness theorem yields, after passing to a further
subsequence if necessary,
\begin{equation}
\label{strong-C-Xdual}
u_m\longrightarrow u
\qquad
\text{strongly in }C([0,T];\mathcal X').
\end{equation}
Consequently,
\[
\int_Mu_m(t)P_m\varphi(t)\,dV_g
\longrightarrow
\int_Mu(t)\varphi(t)\,dV_g
\]
for every \(t\in[0,T]\). The initial and time-derivative terms in
\eqref{weak-form-m-integrated} also converge to their expected limits.

For the nonlinear chemical-potential terms, we already know that
\[
W'(u_m)\longrightarrow W'(u)
\qquad
\text{strongly in }L^2(Q_T),
\]
and, by \eqref{strong-laplacian-final},
\[
\Delta_gu_m\longrightarrow\Delta_gu
\qquad
\text{strongly in }L^2(Q_T).
\]
The latter convergence implies
\[
a(\Delta_gu_m)\longrightarrow a(\Delta_gu)
\qquad
\text{in measure on }Q_T.
\]
Since \(a\) is bounded, it follows that
\[
a(\Delta_gu_m)\longrightarrow a(\Delta_gu)
\qquad
\text{strongly in }L^p(Q_T)
\]
for every finite \(p\geq1\).

Combining this with
\[
\nabla_gu_m\longrightarrow\nabla_gu
\qquad
\text{strongly in }L^2(Q_T),
\]
we obtain
\begin{equation}
\label{active-product-strong-limit}
a(\Delta_gu_m)|\nabla_gu_m|_g
\longrightarrow
a(\Delta_gu)|\nabla_gu|_g
\qquad
\text{strongly in }L^2(Q_T).
\end{equation}
Indeed,
\begin{align*}
&
\left\|
a(\Delta_gu_m)|\nabla_gu_m|_g
-
a(\Delta_gu)|\nabla_gu|_g
\right\|_{L^2(Q_T)}
\\
&\qquad\leq
\|a\|_{L^\infty(\mathbb R)}
\|\nabla_gu_m-\nabla_gu\|_{L^2(Q_T)}
\\
&\qquad\quad+
\left\|
\bigl(
a(\Delta_gu_m)-a(\Delta_gu)
\bigr)
|\nabla_gu|_g
\right\|_{L^2(Q_T)}.
\end{align*}
The first term tends to zero by the strong \(H^1\)-convergence. The second
tends to zero by Vitali's theorem, since
\(a(\Delta_g u_m)\to a(\Delta_g u)\) in measure, \(a\) is bounded, and
\[
\left|
\bigl(a(\Delta_g u_m)-a(\Delta_g u)\bigr)|\nabla_g u|_g
\right|^2
\leq
4\|a\|_{L^\infty(\mathbb R)}^2|\nabla_g u|_g^2
\in L^1(Q_T).
\]
Finally,
\[
b(u_m)\rightharpoonup b(u)
\qquad
\text{weakly in }L^{q'}(Q_T),
\]
whereas
\[
P_m\varphi\longrightarrow\varphi
\qquad
\text{strongly in }L^q(Q_T).
\]
Therefore,
\[
\lambda(x)
\bigl(
b(u_{\mathrm{ref}})-b(u_m)
\bigr)
P_m\varphi
\]
converges to
\[
\lambda(x)
\bigl(
b(u_{\mathrm{ref}})-b(u)
\bigr)
\varphi
\]
in the corresponding weak--strong duality.

Passing to the limit in
\eqref{weak-form-m-integrated}, we obtain
\eqref{weak-form-limit}. Together with
\[
u\in C([0,T];L^2(M)),
\qquad
u(0)=u_{\mathrm{in}},
\]
this shows that \(u\) is a weak solution of
\eqref{eq:active-CH-manifold}--\eqref{bc-navier}.
The proof is complete.
\end{proof}

\subsection{Uniqueness and stability for the smooth-classifier problem}
\label{subsec:stability-smooth-classifier}

In this subsection we establish uniqueness and continuous dependence for the
smooth-classifier problem. For prescribed initial and reference states
$u_{\mathrm{in}}$ and $u_{\mathrm{ref}}$, respectively, the equation reads
\[
\partial_tu
=
\Delta_g\Bigl(
u^3-u
-\mu\Delta_gu
-\nu a(\Delta_gu)|\nabla_gu|_g
\Bigr)
+
\lambda(x)\bigl(b(u_{\mathrm{ref}})-b(u)\bigr).
\]

We work in the surface case $d=2$, which is the principal geometric setting
motivating the model. The dimensional restriction enters through the estimate
of the cubic double-well term. More precisely, on a compact
two-dimensional manifold, Agmon's inequality \cite{Agmon} gives
\[
\|z\|_{L^\infty(M)}^2
\leq
C
\|z\|_{L^2(M)}
\|z\|_{H^2(M)},
\qquad
z\in H_D^2(M).
\]
Consequently, if
\[
z\in
L^\infty(0,T;L^2(M))
\cap
L^2(0,T;H_D^2(M)),
\]
then
\[
t\longmapsto \|z(t)\|_{L^\infty(M)}^4
\]
belongs to $L^1(0,T)$. Without the cubic double-well contribution, the
argument below does not require this two-dimensional estimate and extends
under the corresponding higher-dimensional energy bounds.

\begin{theorem}[Uniqueness and stability for the smooth-classifier problem]
\label{thm:uniqueness-beta}
Assume that $d=2$. Let $u$ and $v$ be weak solutions corresponding,
respectively, to the data pairs
\[
(u_{\mathrm{in}},u_{\mathrm{ref}})
\qquad\text{and}\qquad
(v_{\mathrm{in}},v_{\mathrm{ref}}),
\]
where
\[
u_{\mathrm{in}},v_{\mathrm{in}}\in L^2(M),
\qquad
u_{\mathrm{ref}},v_{\mathrm{ref}}\in L^q(M),
\qquad
q>6.
\]
Assume additionally that
\[
b(u_{\mathrm{ref}})-b(v_{\mathrm{ref}})
\in L^2(M).
\]
Then there exists a constant $C_T>0$ such that
\begin{align}
\label{eq:smooth-stability-estimate}
&\sup_{t\in[0,T]}
\|u(t)-v(t)\|_{L^2(M)}^2
+
\int_0^T
\|\Delta_g(u-v)(t)\|_{L^2(M)}^2\,dt
\nonumber\\
&\qquad\leq
C_T
\left(
\|u_{\mathrm{in}}-v_{\mathrm{in}}\|_{L^2(M)}^2
+
\|b(u_{\mathrm{ref}})
-b(v_{\mathrm{ref}})\|_{L^2(M)}^2
\right).
\end{align}
The constant $C_T$ depends on $T$, the geometry of $(M,g)$, the parameters
$\mu$ and $\nu$, the bound $\lambda^*$, $\|a\|_{L^\infty(\mathbb R)}$, and
the norms of $u$ and $v$ in
\[
L^\infty(0,T;L^2(M))
\cap
L^2(0,T;H_D^2(M)),
\]
but not otherwise on the particular solutions.

In particular, if
\[
u_{\mathrm{in}}=v_{\mathrm{in}}
\qquad\text{and}\qquad
u_{\mathrm{ref}}=v_{\mathrm{ref}},
\]
then $u=v$. Hence the smooth-classifier problem admits at most one weak
solution for each prescribed pair
$(u_{\mathrm{in}},u_{\mathrm{ref}})$.
\end{theorem}

\begin{proof}
Set
\[
w:=u-v,
\qquad
w_{\mathrm{in}}
:=
u_{\mathrm{in}}-v_{\mathrm{in}},
\]
and
\[
B_{\mathrm{ref}}
:=
b(u_{\mathrm{ref}})-b(v_{\mathrm{ref}}).
\]
Subtracting the weak formulations satisfied by $u$ and $v$, we obtain
\begin{align}
\label{eq:diff-weak-manifold}
\langle\partial_tw,\varphi\rangle
&=
\int_M
\bigl(W'(u)-W'(v)\bigr)
\Delta_g\varphi\,dV_g
-
\mu\int_M
\Delta_gw\,\Delta_g\varphi\,dV_g
\nonumber\\
&\quad
-
\nu\int_M
\Bigl(
a(\Delta_gu)|\nabla_gu|_g
-
a(\Delta_gv)|\nabla_gv|_g
\Bigr)
\Delta_g\varphi\,dV_g
\nonumber\\
&\quad
+
\int_M
\lambda(x)
\Bigl(
B_{\mathrm{ref}}
-
\bigl(b(u)-b(v)\bigr)
\Bigr)
\varphi\,dV_g.
\end{align}

We first record that the right-hand side defines a functional on
\[
V:=H_D^2(M).
\]
Indeed, since $d=2$,
\[
V\hookrightarrow L^\infty(M)
\]
continuously. For the double-well term,
\[
W'(u)-W'(v)
=
w\bigl(u^2+uv+v^2\bigr)-w,
\]
and therefore
\[
\|W'(u)-W'(v)\|_{L^2(M)}
\leq
C
\Bigl(
1+\|u\|_{L^\infty(M)}^2
+\|v\|_{L^\infty(M)}^2
\Bigr)
\|w\|_{L^2(M)}.
\]
Consequently,
\begin{align*}
\left|
\int_M
\bigl(W'(u)-W'(v)\bigr)
\Delta_g\varphi\,dV_g
\right|
&\leq
C
\Bigl(
1+\|u\|_{L^\infty}^2
+\|v\|_{L^\infty}^2
\Bigr)
\\
&\qquad\qquad\times
\|w\|_{L^2}
\|\varphi\|_V.
\end{align*}
For the active term, the boundedness of $a$ gives
\[
\left|
a(\Delta_gu)|\nabla_gu|_g
-
a(\Delta_gv)|\nabla_gv|_g
\right|
\leq
\|a\|_{L^\infty}
\bigl(
|\nabla_gu|_g+|\nabla_gv|_g
\bigr),
\]
and hence
\begin{align*}
&
\left|
\int_M
\Bigl(
a(\Delta_gu)|\nabla_gu|_g
-
a(\Delta_gv)|\nabla_gv|_g
\Bigr)
\Delta_g\varphi\,dV_g
\right|
\\
&\qquad\leq
C\|a\|_{L^\infty}
\bigl(
\|\nabla_gu\|_{L^2}
+
\|\nabla_gv\|_{L^2}
\bigr)
\|\varphi\|_V.
\end{align*}

For the reference-state contribution,
\[
\left|
\int_M
\lambda(x)B_{\mathrm{ref}}\varphi\,dV_g
\right|
\leq
\lambda^*
\|B_{\mathrm{ref}}\|_{L^2(M)}
\|\varphi\|_{L^2(M)}.
\]
Moreover, since
\[
b(u),b(v)\in L^{q'}(M),
\qquad
V\hookrightarrow L^q(M),
\]
we have
\[
\left|
\int_M
\lambda(x)
\bigl(b(u)-b(v)\bigr)\varphi\,dV_g
\right|
\leq
C\lambda^*
\|b(u)-b(v)\|_{L^{q'}(M)}
\|\varphi\|_V.
\]
Thus the difference equation may be interpreted in $V'$.

The energy calculation below is justified by Steklov regularisation in
time. More precisely, one tests the time-averaged difference equation by
the corresponding time average of $w$, integrates over $(0,t)$, and then
lets the averaging parameter tend to zero. This is legitimate because
\[
w\in
L^2(0,T;H_D^2(M))
\cap
L^q(Q_T),
\]
and every term in \eqref{eq:diff-weak-manifold} is integrable against $w$.
We consequently obtain
\[
\int_0^t
\langle\partial_tw,w\rangle\,ds
=
\frac12\|w(t)\|_{L^2(M)}^2
-
\frac12\|w_{\mathrm{in}}\|_{L^2(M)}^2.
\]

Equivalently, testing \eqref{eq:diff-weak-manifold} by $w$ in this
regularised sense yields, for a.e. $t\in(0,T)$,
\begin{align}
\label{eq:energy-w-manifold}
\frac12\frac{d}{dt}\|w(t)\|_{L^2(M)}^2
+
\mu\|\Delta_gw(t)\|_{L^2(M)}^2
&=
\int_M
\bigl(W'(u)-W'(v)\bigr)
\Delta_gw\,dV_g
\nonumber\\
&\quad
-
\nu\int_M
\Bigl(
a(\Delta_gu)|\nabla_gu|_g
-
a(\Delta_gv)|\nabla_gv|_g
\Bigr)
\Delta_gw\,dV_g
\nonumber\\
&\quad
+
\int_M
\lambda(x)B_{\mathrm{ref}}w\,dV_g
\nonumber\\
&\quad
-
\int_M
\lambda(x)
\bigl(b(u)-b(v)\bigr)w\,dV_g.
\end{align}

We now estimate the terms on the right-hand side.

For the double-well contribution,
\[
|W'(u)-W'(v)|
\leq
C
\bigl(
1+|u|^2+|v|^2
\bigr)|w|.
\]
Hence
\[
\left|
\int_M
\bigl(W'(u)-W'(v)\bigr)
\Delta_gw\,dV_g
\right|
\leq
K(t)
\|w\|_{L^2(M)}
\|\Delta_gw\|_{L^2(M)},
\]
where
\[
K(t)
:=
C\Bigl(
1+\|u(t)\|_{L^\infty(M)}^2
+\|v(t)\|_{L^\infty(M)}^2
\Bigr).
\]
Young's inequality gives
\begin{equation}
\label{eq:double-well-stability-bound}
\left|
\int_M
\bigl(W'(u)-W'(v)\bigr)
\Delta_gw\,dV_g
\right|
\leq
\frac{\mu}{4}
\|\Delta_gw\|_{L^2(M)}^2
+
C K(t)^2
\|w\|_{L^2(M)}^2.
\end{equation}

By Agmon's inequality,
\[
\|u(t)\|_{L^\infty(M)}^4
\leq
C
\|u(t)\|_{L^2(M)}^2
\|u(t)\|_{H^2(M)}^2,
\]
and similarly for $v$. Since
\[
u,v\in
L^\infty(0,T;L^2(M))
\cap
L^2(0,T;H_D^2(M)),
\]
we conclude that
\[
K^2\in L^1(0,T).
\]

We next estimate the active contribution. We decompose
\begin{align*}
&a(\Delta_gu)|\nabla_gu|_g
-
a(\Delta_gv)|\nabla_gv|_g
\\
&\qquad=
\bigl(
a(\Delta_gu)-a(\Delta_gv)
\bigr)|\nabla_gu|_g
+
a(\Delta_gv)
\bigl(
|\nabla_gu|_g-|\nabla_gv|_g
\bigr).
\end{align*}
Since $a$ is non-decreasing and
\[
\Delta_gu-\Delta_gv=\Delta_gw,
\]
we have
\[
\bigl(
a(\Delta_gu)-a(\Delta_gv)
\bigr)\Delta_gw
\geq0.
\]
Because $|\nabla_gu|_g\geq0$, it follows that
\begin{equation}
\label{eq:favourable-classifier-term}
-\nu
\int_M
\bigl(
a(\Delta_gu)-a(\Delta_gv)
\bigr)
|\nabla_gu|_g
\Delta_gw\,dV_g
\leq0.
\end{equation}

For the remaining part, we use
\[
\left|
|\nabla_gu|_g-|\nabla_gv|_g
\right|
\leq
|\nabla_gw|_g
\]
and the boundedness of $a$ to obtain
\begin{align*}
&-\nu
\int_M
a(\Delta_gv)
\bigl(
|\nabla_gu|_g-|\nabla_gv|_g
\bigr)
\Delta_gw\,dV_g
\\
&\qquad\leq
\nu
\|a\|_{L^\infty}
\|\nabla_gw\|_{L^2(M)}
\|\Delta_gw\|_{L^2(M)}.
\end{align*}
Combining Young's inequality with
\[
\|\nabla_gw\|_{L^2(M)}^2
\leq
\varepsilon
\|\Delta_gw\|_{L^2(M)}^2
+
C_\varepsilon
\|w\|_{L^2(M)}^2,
\]
and choosing $\varepsilon>0$ sufficiently small, we infer
\begin{align}
\label{eq:active-stability-bound}
&-\nu
\int_M
\Bigl(
a(\Delta_gu)|\nabla_gu|_g
-
a(\Delta_gv)|\nabla_gv|_g
\Bigr)
\Delta_gw\,dV_g
\nonumber\\
&\qquad\leq
\frac{\mu}{4}
\|\Delta_gw\|_{L^2(M)}^2
+
C\|w\|_{L^2(M)}^2.
\end{align}

The state-dependent anchoring contribution is dissipative. Since $b$ is
non-decreasing,
\[
\bigl(b(u)-b(v)\bigr)(u-v)\geq0,
\]
and therefore
\begin{equation}
\label{eq:anchoring-dissipation}
-\int_M
\lambda(x)
\bigl(b(u)-b(v)\bigr)w\,dV_g
\leq0.
\end{equation}

The reference-state contribution satisfies
\begin{align}
\label{eq:reference-stability-bound}
\int_M
\lambda(x)B_{\mathrm{ref}}w\,dV_g
&\leq
\lambda^*
\|B_{\mathrm{ref}}\|_{L^2(M)}
\|w\|_{L^2(M)}
\nonumber\\
&\leq
C\|B_{\mathrm{ref}}\|_{L^2(M)}^2
+
C\|w\|_{L^2(M)}^2.
\end{align}

Combining
\eqref{eq:double-well-stability-bound},
\eqref{eq:active-stability-bound},
\eqref{eq:anchoring-dissipation}, and
\eqref{eq:reference-stability-bound}
in \eqref{eq:energy-w-manifold}, we obtain
\begin{equation}
\label{eq:smooth-stability-differential}
\frac{d}{dt}\|w(t)\|_{L^2(M)}^2
+
\mu\|\Delta_gw(t)\|_{L^2(M)}^2
\leq
C\bigl(1+K(t)^2\bigr)
\|w(t)\|_{L^2(M)}^2
+
C\|B_{\mathrm{ref}}\|_{L^2(M)}^2.
\end{equation}
Since $K^2\in L^1(0,T)$, Gronwall's inequality gives
\[
\sup_{t\in[0,T]}
\|w(t)\|_{L^2(M)}^2
\leq
C_T
\left(
\|w_{\mathrm{in}}\|_{L^2(M)}^2
+
\|B_{\mathrm{ref}}\|_{L^2(M)}^2
\right).
\]
Integrating \eqref{eq:smooth-stability-differential} over $(0,T)$ and using
the preceding estimate yields
\[
\int_0^T
\|\Delta_gw(t)\|_{L^2(M)}^2\,dt
\leq
C_T
\left(
\|w_{\mathrm{in}}\|_{L^2(M)}^2
+
\|B_{\mathrm{ref}}\|_{L^2(M)}^2
\right).
\]
Together, these two bounds give
\eqref{eq:smooth-stability-estimate}.

Finally, if
\[
u_{\mathrm{in}}=v_{\mathrm{in}}
\qquad\text{and}\qquad
u_{\mathrm{ref}}=v_{\mathrm{ref}},
\]
then
\[
w_{\mathrm{in}}=0,
\qquad
B_{\mathrm{ref}}=0.
\]
Estimate \eqref{eq:smooth-stability-estimate} therefore gives $w=0$, and
hence $u=v$. This proves uniqueness.
\end{proof}

\section{The steep-classifier limit and the sign-graph inclusion}
\label{sec:steep-classifier}

Fix \(T>0\) and an admissible data pair
\((u_{\mathrm{in}},u_{\mathrm{ref}})\). We retain the standard
double-well potential
\[
    W(u)=\frac14(1-u^2)^2,
    \qquad
    W'(u)=u^3-u.
\]
We analyse the singular regime in which the smooth classifier becomes
increasingly steep. The essential issue is its behaviour on the zero set of
the limiting Laplacian, where the scale \(N\Delta_g u_N\) need not vanish
and a multivalued constitutive law is therefore required.

For \(N\in\mathbb N\), set
\begin{equation}
    a_N(r):=\frac{2}{\pi}\arctan(Nr),
    \qquad r\in\mathbb R.
    \label{eq:steep-classifier}
\end{equation}
Then $a_N$ is smooth, odd, non-decreasing, $|a_N|\leq 1$, and
\[
    a_N(r)\longrightarrow \operatorname{sgn}_0(r)
    \qquad\text{for every }r\in\mathbb R,
\]
where
\[
    \operatorname{sgn}_0(r):=
    \begin{cases}
        1, & r>0,\\
        0, & r=0,\\
       -1, & r<0.
    \end{cases}
\]
Thus $\operatorname{sgn}_0$ is the pointwise limit when the argument
$r$ is fixed. In the singular problem, however, the arguments
$\Delta_g u_N$ vary simultaneously with $N$. The sequentially closed
constitutive limit is therefore the maximal monotone graph
\begin{equation}
    \operatorname{Sign}(r):=
    \begin{cases}
        \{1\}, & r>0,\\
        [-1,1], & r=0,\\
        \{-1\}, & r<0.
    \end{cases}
    \label{eq:maximal-sign-graph}
\end{equation}
Equivalently,
\begin{equation}
    \xi\in \operatorname{Sign}(r)
    \quad\Longleftrightarrow\quad
    |\xi|\leq 1
    \ \text{ and }\ 
    \xi r=|r|.
\label{eq:sign-graph-equivalent}
\end{equation}
Moreover, \(\operatorname{Sign}=\partial|\cdot|\); hence
\(\operatorname{Sign}\) is a maximal monotone graph; see
\cite{Brezis1973}.
The graph is monotone: if
$\xi\in\operatorname{Sign}(r)$ and
$\eta\in\operatorname{Sign}(s)$, then
\begin{equation}
    (\xi-\eta)(r-s)\geq 0.
    \label{eq:sign-graph-monotonicity}
\end{equation}
For each $N$, let $u_N$ be a weak solution corresponding to the
classifier $a_N$ and to the fixed data pair
$(u_{\mathrm{in}},u_{\mathrm{ref}})$, that is,
\begin{equation}
\begin{aligned}
    \partial_t u_N
    ={}& \Delta_g\Bigl(
        W'(u_N)-\mu\Delta_g u_N
        -\nu a_N(\Delta_g u_N)|\nabla_g u_N|_g
        \Bigr)\\
    &\quad +\lambda(x)\bigl(b(u_{\mathrm{ref}})-b(u_N)\bigr)
    \qquad\text{in }Q_T.
\end{aligned}
\label{eq:steep-smooth-problem}
\end{equation}
with
\[
    u_N(0,\cdot)=u_{\mathrm{in}}.
\]
The boundary realization is the one specified in Section~\ref{sec:model}.

\begin{definition}[Weak solution of the sign-graph problem]
\label{def:sign-graph-solution}
A pair $(u,\xi)$ is called a weak solution of the limiting
sign-graph problem if
\[
\begin{aligned}
    &u\in L^\infty(0,T;L^2(M))
       \cap L^2(0,T;H_D^2(M))
       \cap L^q(Q_T),\\
    &\partial_tu\in L^{q'}(0,T;\mathcal X'),
    \qquad
    \xi\in L^\infty(Q_T),
\end{aligned}
\]
$u\in C([0,T];L^2(M))$, $u(0)=u_{\mathrm{in}}$, and
\begin{equation}
    \xi(t,x)\in\operatorname{Sign}(\Delta_g u(t,x))
    \qquad\text{for a.e. }(t,x)\in Q_T.
    \label{eq:graph-inclusion}
\end{equation}
Moreover, for every $\phi\in\mathcal X$ and for a.e. $t\in(0,T)$,
\begin{equation}
\begin{aligned}
    \langle \partial_tu(t),\phi\rangle_{\mathcal X',\mathcal X}
    ={}& \int_M
       \Bigl(
       W'(u)-\mu\Delta_g u
       -\nu\xi|\nabla_g u|_g
       \Bigr)\Delta_g\phi\,dV_g\\
    &+\int_M
       \lambda(x)\bigl(b(u_{\mathrm{ref}})-b(u)\bigr)
       \phi\,dV_g.
\end{aligned}
\label{eq:weak-sign-graph}
\end{equation}
\end{definition}

\begin{lemma}[Closure of the steep classifiers]
\label{lem:steep-graph-closure}
Let $N_k\to\infty$, let $z_k\to z$ strongly in $L^2(Q_T)$, and set
\[
    \xi_k:=a_{N_k}(z_k).
\]
Assume that, after extraction,
\[
    \xi_k\stackrel{*}{\rightharpoonup}\xi
    \qquad\text{weakly-* in }L^\infty(Q_T).
\]
Then
\[
    \xi(t,x)\in\operatorname{Sign}(z(t,x))
    \qquad\text{for a.e. }(t,x)\in Q_T.
\]
\end{lemma}

\begin{proof}
Since $\|\xi_k\|_{L^\infty(Q_T)}\leq1$ and the closed unit ball of
$L^\infty(Q_T)$ is weakly-* closed, we have
\[
\|\xi\|_{L^\infty(Q_T)}\leq1.
\] Passing to a further subsequence, we may
assume that $z_k\to z$ a.e. in $Q_T$. 
Fix $m\in\mathbb N$ and set
\[
    E_m^+:=\{(t,x)\in Q_T:z(t,x)\geq m^{-1}\}.
\]
For a.e. $(t,x)\in E_m^+$, one has $z_k(t,x)\geq (2m)^{-1}$ for all
sufficiently large $k$. Hence
\[
    N_k z_k(t,x)\longrightarrow +\infty,
    \qquad
    a_{N_k}(z_k(t,x))\longrightarrow 1.
\]
By dominated convergence,
$\xi_k\to 1$ strongly in $L^p(E_m^+)$ for every finite $p$. The
weak-* limit therefore satisfies $\xi=1$ a.e. on $E_m^+$. Taking the
union over $m$ gives $\xi=1$ a.e. on $\{z>0\}$. The same argument
shows that $\xi=-1$ a.e. on $\{z<0\}$. On $\{z=0\}$ we only know
$|\xi|\leq 1$, which is exactly the graph relation
$\xi\in\operatorname{Sign}(z)$.
\end{proof}

\begin{theorem}[Steep-classifier limit]
\label{thm:steep-classifier-limit}
Assume that
\[
u_{\mathrm{in}}\in L^2(M),
\qquad
u_{\mathrm{ref}}\in L^q(M),
\qquad q>6,
\]
and that $\lambda$ satisfies the standing assumptions. Let $u_N$ be weak
solutions of \eqref{eq:steep-smooth-problem}. Then there exist a
subsequence, not relabelled, a function $u$, and a function
$\xi\in L^\infty(Q_T)$ such that
\begin{align}
    u_N &\stackrel{*}{\rightharpoonup} u
    &&\text{in }L^\infty(0,T;L^2(M)),
    \label{eq:steep-weakstar-L2}\\
    u_N &\rightharpoonup u
    &&\text{in }L^2(0,T;H_D^2(M)),
    \label{eq:steep-weak-H2}\\
    u_N &\longrightarrow u
    &&\text{in }L^2(0,T;H^1(M)),
    \label{eq:steep-strong-H1}\\
    \Delta_g u_N &\longrightarrow \Delta_g u
    &&\text{in }L^2(Q_T),
    \label{eq:steep-strong-Laplacian}\\
    a_N(\Delta_g u_N)&\stackrel{*}{\rightharpoonup}\xi
    &&\text{in }L^\infty(Q_T).
    \label{eq:steep-classifier-weakstar}
\end{align}
In particular,
\begin{equation}
    u_N\longrightarrow u
    \qquad\text{strongly in }L^2(0,T;H_D^2(M)).
    \label{eq:steep-strong-H2}
\end{equation}
Moreover,
\begin{equation}
    \xi\in\operatorname{Sign}(\Delta_g u)
    \qquad\text{a.e. in }Q_T,
    \label{eq:steep-limit-graph}
\end{equation}
and
\begin{equation}
    a_N(\Delta_g u_N)|\nabla_g u_N|_g
    \rightharpoonup
    \xi|\nabla_g u|_g
    \qquad\text{weakly in }L^2(Q_T).
    \label{eq:steep-active-product-limit}
\end{equation}
Consequently, $(u,\xi)$ is a weak solution of the sign-graph problem
in the sense of Definition~\ref{def:sign-graph-solution}.
\end{theorem}

\begin{proof}
We divide the proof into four steps.

\medskip
\noindent
\emph{Step 1: uniform bounds and preliminary compactness.}
Since $\|a_N\|_{L^\infty(\mathbb R)}\leq 1$, the estimates proved for
the smooth-classifier problem are uniform in $N$. Thus
\[
    (u_N)_N
    \quad\text{is bounded in}\quad
    L^\infty(0,T;L^2(M))
    \cap L^2(0,T;H_D^2(M))
    \cap L^q(Q_T),
\]
and
\[
    (\partial_tu_N)_N
    \quad\text{is bounded in}\quad
    L^{q'}(0,T;\mathcal X').
\]
By weak compactness and the Aubin--Lions lemma, after extraction we
obtain \eqref{eq:steep-weakstar-L2}--\eqref{eq:steep-strong-H1}, and also
\[
    u_N\to u
    \qquad\text{strongly in }L^r(Q_T)
    \quad\text{for every }1\leq r<q.
\]
In particular,
\[
    W'(u_N)\to W'(u)
    \qquad\text{strongly in }L^2(Q_T).
\]

Set
\[
    \chi_N:=a_N(\Delta_g u_N)|\nabla_g u_N|_g.
\]
The sequence $(\chi_N)_N$ is bounded in $L^2(Q_T)$. Hence, after a
further extraction,
\[
    \chi_N\rightharpoonup\chi
    \qquad\text{weakly in }L^2(Q_T).
\]
After passing to a further subsequence, we also have
\[
u_N\longrightarrow u
\qquad\text{a.e. in }Q_T.
\]
The growth assumption on $b$ and the uniform $L^q(Q_T)$-bound imply
that $(b(u_N))_N$ is bounded in $L^{q'}(Q_T)$. Since $b$ is
continuous,
\[
b(u_N)\rightharpoonup b(u)
\qquad\text{weakly in }L^{q'}(Q_T).
\]

Passing to the limit in the time integrated weak formulation, before identifying $\chi$, we obtain the
preliminary equation
\[
    \partial_tu
    =\Delta_g\bigl(W'(u)-\mu\Delta_g u-\nu\chi\bigr)
    +\lambda(x)\bigl(b(u_{\mathrm{ref}})-b(u)\bigr).
\]

The argument of Lemma~\ref{lem:trace-preliminary-flux} therefore gives
\[
u\in C([0,T];L^2(M)),
\qquad
u(0)=u_{\mathrm{in}}.
\]

We may consequently apply the smooth-approximation lemma used in the
Minty argument. Thus there exist functions $u^\zeta$ such that, as
$\zeta\downarrow0$,
\begin{equation}
\begin{aligned}
    u^\zeta&\to u
    &&\text{in }L^2(0,T;H_D^2(M))\cap L^q(Q_T),\\
    u^\zeta(0)&\to u_{\mathrm{in}}
    &&\text{in }L^2(M),\\
    W'(u^\zeta)&\to W'(u)
    &&\text{in }L^2(Q_T),\\
    b(u^\zeta)&\to b(u)
    &&\text{in }L^{q'}(Q_T),
\end{aligned}
\label{eq:steep-smooth-approximation}
\end{equation}
and
\begin{equation}
    \int_0^T
    \left|
    \left\langle \partial_tu^\zeta,
    u-u^\zeta\right\rangle
    \right|dt
    \longrightarrow0.
    \label{eq:steep-time-pairing}
\end{equation}

\medskip
\noindent
\emph{Step 2: strong convergence of the Laplacians.}
Fix $\zeta>0$ and set
\[
w_N^\zeta:=u_N-u^\zeta.
\]
Let
\[
\mathcal Y:=H_D^2(M)\cap L^q(M).
\]
The uniform estimates imply that
\[
\mathfrak m_N
:=
W'(u_N)-\mu\Delta_g u_N
-\nu a_N(\Delta_g u_N)|\nabla_g u_N|_g
\]
is bounded in $L^2(Q_T)$, while
\[
\lambda(x)\bigl(b(u_{\mathrm{ref}})-b(u_N)\bigr)
\]
is bounded in $L^{q'}(Q_T)$. Hence the right-hand side of the weak
formulation defines an element of
\[
L^2(0,T;H_D^2(M)')
+
L^{q'}(0,T;L^{q'}(M)).
\]
Since $\mathcal X$ is dense in $\mathcal Y$ and the weak identity is
continuous with respect to the $\mathcal Y$-norm, it extends from
$\mathcal X$ to $\mathcal Y$. Consequently,
\[
w_N^\zeta\in
L^2(0,T;H_D^2(M))\cap L^q(Q_T)
\]
may be used as a test function after the usual Steklov regularisation
in time.

The equation itself yields, uniformly in $N$,
\[
    \partial_tu_N\in
    L^2(0,T;H_D^2(M)')+L^{q'}(Q_T),
\]
so that the duality with
$w_N^\zeta\in L^2(0,T;H_D^2(M))\cap L^q(Q_T)$ is meaningful.
Testing the equation for $u_N$ by $w_N^\zeta$, with the usual
Steklov regularisation in time, gives
\begin{align}
&\int_0^T\langle\partial_tu_N,w_N^\zeta\rangle\,dt
 +\mu\int_{Q_T}|\Delta_g w_N^\zeta|^2\,dV_gdt
\notag\\
&\quad=
 \int_{Q_T}W'(u_N)\Delta_gw_N^\zeta\,dV_gdt
 -\mu\int_{Q_T}\Delta_gu^\zeta\Delta_gw_N^\zeta\,dV_gdt
\notag\\
&\qquad
 -\nu\int_{Q_T}
 a_N(\Delta_gu_N)|\nabla_gu_N|_g
 \Delta_gw_N^\zeta\,dV_gdt
\notag\\
&\qquad
 +\int_{Q_T}\lambda(x)
 \bigl(b(u_{\mathrm{ref}})-b(u_N)\bigr)
 w_N^\zeta\,dV_gdt.
\label{eq:steep-Minty-start}
\end{align}
Since $a_N$ is non-decreasing,
\[
 \bigl(a_N(\Delta_gu_N)-a_N(\Delta_gu^\zeta)\bigr)
 \bigl(\Delta_gu_N-\Delta_gu^\zeta\bigr)\geq0.
\]
Multiplication by $|\nabla_gu_N|_g\geq0$ therefore yields
\begin{align}
&-\nu\int_{Q_T}
 a_N(\Delta_gu_N)|\nabla_gu_N|_g
 \Delta_gw_N^\zeta\,dV_gdt
\notag\\
&\qquad\leq
 -\nu\int_{Q_T}
 a_N(\Delta_gu^\zeta)|\nabla_gu_N|_g
 \Delta_gw_N^\zeta\,dV_gdt.
\label{eq:steep-active-Minty}
\end{align}
Likewise, the monotonicity of $b$ gives
\begin{align}
&\int_{Q_T}\lambda(x)
 \bigl(b(u_{\mathrm{ref}})-b(u_N)\bigr)
 w_N^\zeta\,dV_gdt
\notag\\
&\qquad\leq
 \int_{Q_T}\lambda(x)
 \bigl(b(u_{\mathrm{ref}})-b(u^\zeta)\bigr)
 w_N^\zeta\,dV_gdt.
\label{eq:steep-anchor-Minty}
\end{align}
Finally,
\[
\begin{aligned}
 \int_0^T\langle\partial_tu_N,w_N^\zeta\rangle\,dt
 ={}&\frac12\|w_N^\zeta(T)\|_{L^2(M)}^2
 -\frac12\|u_{\mathrm{in}}-u^\zeta(0)\|_{L^2(M)}^2\\
 &+\int_0^T
 \langle\partial_tu^\zeta,w_N^\zeta\rangle\,dt.
\end{aligned}
\]
Dropping the non-negative terminal term, we infer
\begin{align}
\mu\|\Delta_gw_N^\zeta\|_{L^2(Q_T)}^2
\leq{}&
 \frac12\|u_{\mathrm{in}}-u^\zeta(0)\|_{L^2(M)}^2
 -\int_0^T
 \langle\partial_tu^\zeta,w_N^\zeta\rangle\,dt
\notag\\
&+\int_{Q_T}W'(u_N)\Delta_gw_N^\zeta\,dV_gdt
\notag\\
&-\nu\int_{Q_T}
 a_N(\Delta_gu^\zeta)|\nabla_gu_N|_g
 \Delta_gw_N^\zeta\,dV_gdt
\notag\\
&-\mu\int_{Q_T}\Delta_gu^\zeta\Delta_gw_N^\zeta\,dV_gdt
\notag\\
&+\int_{Q_T}\lambda(x)
 \bigl(b(u_{\mathrm{ref}})-b(u^\zeta)\bigr)
 w_N^\zeta\,dV_gdt.
\label{eq:steep-Minty-inequality}
\end{align}

For fixed $\zeta$, one has
\begin{equation}
    a_N(\Delta_gu^\zeta)
    \longrightarrow
    \operatorname{sgn}_0(\Delta_gu^\zeta)
    \qquad\text{strongly in }L^p(Q_T)
\end{equation}
for every finite $p$. Moreover,
\begin{equation}
\begin{aligned}
&a_N(\Delta_gu^\zeta)|\nabla_gu_N|_g
\longrightarrow
\operatorname{sgn}_0(\Delta_gu^\zeta)|\nabla_gu|_g
\quad\text{strongly in }L^2(Q_T).
\end{aligned}
\label{eq:fixed-zeta-active-convergence}
\end{equation}
Indeed,
\begin{align*}
&\bigl\|
 a_N(\Delta_gu^\zeta)|\nabla_gu_N|_g
 -\operatorname{sgn}_0(\Delta_gu^\zeta)|\nabla_gu|_g
 \bigr\|_{L^2(Q_T)}\\
&\quad\leq
 \bigl\||\nabla_gu_N|_g-|\nabla_gu|_g\bigr\|_{L^2(Q_T)}
 +\bigl\|
 \bigl(a_N(\Delta_gu^\zeta)-
 \operatorname{sgn}_0(\Delta_gu^\zeta)\bigr)
 |\nabla_gu|_g
 \bigr\|_{L^2(Q_T)},
\end{align*}
where the first term tends to zero by
\eqref{eq:steep-strong-H1}, while the second tends to zero by dominated
convergence.

Since
\[
    \Delta_gw_N^\zeta
    \rightharpoonup
    \Delta_g(u-u^\zeta)
    \qquad\text{weakly in }L^2(Q_T),
\]
we may take the limit superior in
\eqref{eq:steep-Minty-inequality}. We obtain
\begin{equation}
    \mu\limsup_{N\to\infty}
    \|\Delta_g(u_N-u^\zeta)\|_{L^2(Q_T)}^2
    \leq r_\zeta,
    \label{eq:steep-limsup}
\end{equation}
where
\begin{align}
 r_\zeta:={}&
 \frac12\|u_{\mathrm{in}}-u^\zeta(0)\|_{L^2(M)}^2
 +\left|
 \int_0^T
 \langle\partial_tu^\zeta,u-u^\zeta\rangle\,dt
 \right|
\notag\\
&+\left|
 \int_{Q_T}W'(u)\Delta_g(u-u^\zeta)\,dV_gdt
 \right|
\notag\\
&+\nu\left|
 \int_{Q_T}
 \operatorname{sgn}_0(\Delta_gu^\zeta)|\nabla_gu|_g
 \Delta_g(u-u^\zeta)\,dV_gdt
 \right|
\notag\\
&+\mu\left|
 \int_{Q_T}\Delta_gu^\zeta\Delta_g(u-u^\zeta)\,dV_gdt
 \right|
\notag\\
&+\left|
 \int_{Q_T}\lambda(x)
 \bigl(b(u_{\mathrm{ref}})-b(u^\zeta)\bigr)
 (u-u^\zeta)\,dV_gdt
 \right|.
\label{eq:steep-r-zeta}
\end{align}

The first two terms in \eqref{eq:steep-r-zeta} tend to zero by
\eqref{eq:steep-smooth-approximation} and
\eqref{eq:steep-time-pairing}. Moreover,
\[
W'(u)\in L^2(Q_T),
\qquad
\Delta_g(u-u^\zeta)\to0
\quad\text{in }L^2(Q_T),
\]
so the double-well term tends to zero. The linear fourth-order term also
tends to zero, since $(\Delta_g u^\zeta)_\zeta$ is bounded in
$L^2(Q_T)$ and
\[
\left|
\int_{Q_T}
\Delta_gu^\zeta\Delta_g(u-u^\zeta)\,dV_g\,dt
\right|
\leq
\|\Delta_gu^\zeta\|_{L^2(Q_T)}
\|\Delta_g(u-u^\zeta)\|_{L^2(Q_T)}
\longrightarrow0.
\]
For the anchoring term, H\"older's inequality and
\eqref{eq:steep-smooth-approximation} give
\[
\begin{aligned}
&\left|
\int_{Q_T}\lambda(x)
\bigl(b(u_{\mathrm{ref}})-b(u^\zeta)\bigr)
(u-u^\zeta)\,dV_g\,dt
\right|\\
&\qquad\leq
\|\lambda\|_{L^\infty(M)}
\|b(u_{\mathrm{ref}})-b(u^\zeta)\|_{L^{q'}(Q_T)}
\|u-u^\zeta\|_{L^q(Q_T)}
\longrightarrow0.
\end{aligned}
\]
For the active term, it is enough to use
$|\operatorname{sgn}_0|\leq1$:
\[
\begin{aligned}
&\left|
\int_{Q_T}
\operatorname{sgn}_0(\Delta_gu^\zeta)|\nabla_gu|_g
\Delta_g(u-u^\zeta)\,dV_g\,dt
\right|\\
&\qquad\leq
\|\nabla_gu\|_{L^2(Q_T)}
\|\Delta_g(u-u^\zeta)\|_{L^2(Q_T)}
\longrightarrow0.
\end{aligned}
\]
Consequently,
\[
r_\zeta\longrightarrow0
\qquad\text{as }\zeta\downarrow0.
\]
Therefore,
\[
\limsup_{N\to\infty}
\|\Delta_g(u_N-u)\|_{L^2(Q_T)}
\leq
\mu^{-1/2}r_\zeta^{1/2}
+\|\Delta_g(u^\zeta-u)\|_{L^2(Q_T)}.
\]
Letting $\zeta\downarrow0$ yields
\eqref{eq:steep-strong-Laplacian}. The elliptic estimate, together
with the already known strong $L^2$-convergence, then gives
\eqref{eq:steep-strong-H2}.

\medskip
\noindent
\emph{Step 3: identification of the maximal monotone graph.}
Set
\[
    \xi_N:=a_N(\Delta_gu_N).
\]
Since $|\xi_N|\leq1$, there exists
$\xi\in L^\infty(Q_T)$ such that, after extraction,
\eqref{eq:steep-classifier-weakstar} holds. Applying
Lemma~\ref{lem:steep-graph-closure} with
$z_N=\Delta_gu_N$ and $z=\Delta_gu$ yields
\eqref{eq:steep-limit-graph}.

Furthermore, for every $\Psi\in L^2(Q_T)$,
\begin{align*}
&\int_{Q_T}
 \bigl(\xi_N|\nabla_gu_N|_g-
       \xi|\nabla_gu|_g\bigr)\Psi\,dV_gdt\\
&\quad=
 \int_{Q_T}\xi_N
 \bigl(|\nabla_gu_N|_g-|\nabla_gu|_g\bigr)
 \Psi\,dV_gdt
 +\int_{Q_T}(\xi_N-\xi)|\nabla_gu|_g\Psi\,dV_gdt.
\end{align*}
The first term tends to zero because
$\nabla_gu_N\to\nabla_gu$ strongly in $L^2(Q_T)$, while the second
one tends to zero by the weak-* convergence of $\xi_N$, since
$|\nabla_gu|_g\Psi\in L^1(Q_T)$. This proves
\eqref{eq:steep-active-product-limit}.

\medskip
\noindent
\emph{Step 4: passage to the limiting equation.}
The strong convergence of $W'(u_N)$, the strong convergence of the
Laplacians, the weak convergence
\eqref{eq:steep-active-product-limit}, and the weak convergence
$b(u_N)\rightharpoonup b(u)$ in $L^{q'}(Q_T)$ permit
passage to the limit in the weak formulation. Hence
$(u,\xi)$ satisfies \eqref{eq:weak-sign-graph} and is a weak
sign-graph solution.
\end{proof}

\begin{theorem}[Two-dimensional stability and uniqueness of the state]
\label{thm:sign-graph-uniqueness}
Assume $d=2$. Let $(u,\xi)$ and $(v,\eta)$ be weak sign-graph
solutions corresponding to the data pairs
$(u_{\mathrm{in}},u_{\mathrm{ref}})$ and
$(v_{\mathrm{in}},v_{\mathrm{ref}})$, respectively. Assume
\[
    u_{\mathrm{in}},v_{\mathrm{in}}\in L^2(M),
    \qquad
    u_{\mathrm{ref}},v_{\mathrm{ref}}\in L^q(M),
    \qquad
    b(u_{\mathrm{ref}})-b(v_{\mathrm{ref}})\in L^2(M).
\]
Then there exists $C_T>0$, depending only on the energy norms of
$u$ and $v$, on $\mu,\nu,\lambda^*$, and on the geometry of $M$, such
that
\begin{equation}
\begin{aligned}
&\operatorname*{ess\,sup}_{t\in[0,T]}
    \|u(t)-v(t)\|_{L^2(M)}^2
    +\int_0^T\|\Delta_g(u-v)\|_{L^2(M)}^2\,dt\\
&\qquad\leq C_T\Bigl(
       \|u_{\mathrm{in}}-v_{\mathrm{in}}\|_{L^2(M)}^2
       +\|b(u_{\mathrm{ref}})-b(v_{\mathrm{ref}})\|_{L^2(M)}^2
    \Bigr).
\end{aligned}
\label{eq:sign-graph-stability}
\end{equation}

In particular, if
$u_{\mathrm{in}}=v_{\mathrm{in}}$ and
$b(u_{\mathrm{ref}})=b(v_{\mathrm{ref}})$, then $u=v$. Thus the state
component $u$ is unique for a fixed initial/reference pair. The graph
variable $\xi$ need not be unique on the set $\{\Delta_gu=0\}$.
\end{theorem}

\begin{proof}
Set
\[
    w:=u-v,
    \qquad
    w_{\mathrm{in}}:=u_{\mathrm{in}}-v_{\mathrm{in}},
    \qquad
    B_{\mathrm{ref}}:=b(u_{\mathrm{ref}})-b(v_{\mathrm{ref}}).
\]
Subtracting the two weak formulations and testing by $w$, justified
by the same Steklov regularisation used for the smooth-classifier
problem, gives
\begin{align}
\frac12\frac{d}{dt}\|w\|_{L^2(M)}^2
+\mu\|\Delta_gw\|_{L^2(M)}^2
={}&
\int_M\bigl(W'(u)-W'(v)\bigr)\Delta_gw\,dV_g
\notag\\
&-\nu\int_M
\bigl(\xi|\nabla_gu|_g-\eta|\nabla_gv|_g\bigr)
\Delta_gw\,dV_g
\notag\\
&+\int_M\lambda(x)
B_{\mathrm{ref}}w\,dV_g
\notag\\
&-\int_M\lambda(x)
\bigl(b(u)-b(v)\bigr)w\,dV_g.
\label{eq:sign-graph-difference-energy}
\end{align}

For the double-well term,
\[
    W'(u)-W'(v)
    =w\bigl(u^2+uv+v^2\bigr)-w,
\]
so that
\[
    |W'(u)-W'(v)|
    \leq C\bigl(1+|u|^2+|v|^2\bigr)|w|.
\]
Hence
\begin{equation}
\begin{aligned}
\left|
\int_M\bigl(W'(u)-W'(v)\bigr)\Delta_gw\,dV_g
\right|
\leq
\frac{\mu}{4}\|\Delta_gw\|_{L^2(M)}^2
+C K(t)^2\|w\|_{L^2(M)}^2,
\end{aligned}
\label{eq:sign-graph-double-well-estimate}
\end{equation}
where
\[
    K(t):=1+\|u(t)\|_{L^\infty(M)}^2
             +\|v(t)\|_{L^\infty(M)}^2.
\]
By Agmon's inequality \cite{Agmon} in dimension two,
\[
    \|z\|_{L^\infty(M)}^2
    \leq C\|z\|_{L^2(M)}\|z\|_{H^2(M)},
\]
and therefore $K^2\in L^1(0,T)$.

We next estimate the active term. We write
\begin{align}
&\xi|\nabla_gu|_g-\eta|\nabla_gv|_g
\notag\\
&\qquad=
(\xi-\eta)|\nabla_gu|_g
+\eta\bigl(|\nabla_gu|_g-|\nabla_gv|_g\bigr).
\label{eq:sign-graph-active-decomposition}
\end{align}
Since
\[
    \xi\in\operatorname{Sign}(\Delta_gu),
    \qquad
    \eta\in\operatorname{Sign}(\Delta_gv),
\]
the monotonicity of the graph gives
\begin{equation}
    (\xi-\eta)\bigl(\Delta_gu-\Delta_gv\bigr)
    =(\xi-\eta)\Delta_gw\geq0
    \qquad\text{a.e. in }Q_T.
    \label{eq:pointwise-sign-monotonicity}
\end{equation}
Since $|\nabla_gu|_g\geq0$, it follows that
\begin{equation}
   -\nu\int_M
   (\xi-\eta)|\nabla_gu|_g\Delta_gw\,dV_g
   \leq0.
   \label{eq:favorable-sign-graph-term}
\end{equation}
For the remaining part, using $|\eta|\leq1$ and
\[
    \bigl||\nabla_gu|_g-|\nabla_gv|_g\bigr|
    \leq |\nabla_gw|_g,
\]
we obtain
\begin{align}
&-\nu\int_M
\eta\bigl(|\nabla_gu|_g-|\nabla_gv|_g\bigr)
\Delta_gw\,dV_g
\notag\\
&\qquad\leq
\nu\|\nabla_gw\|_{L^2(M)}
\|\Delta_gw\|_{L^2(M)}.
\end{align}
The interpolation estimate
\[
    \|\nabla_gw\|_{L^2(M)}^2
    \leq \varepsilon\|\Delta_gw\|_{L^2(M)}^2
    +C_\varepsilon\|w\|_{L^2(M)}^2
\]
and Young's inequality therefore imply
\begin{equation}
\begin{aligned}
&-\nu\int_M
\bigl(\xi|\nabla_gu|_g-\eta|\nabla_gv|_g\bigr)
\Delta_gw\,dV_g\\
&\qquad\leq
\frac{\mu}{4}\|\Delta_gw\|_{L^2(M)}^2
+C\|w\|_{L^2(M)}^2.
\end{aligned}
\label{eq:sign-graph-active-estimate}
\end{equation}

The nonlinear anchoring contribution is dissipative because
$b$ is monotone:
\begin{equation}
    -\int_M\lambda(x)
    \bigl(b(u)-b(v)\bigr)w\,dV_g
    \leq0.
    \label{eq:sign-graph-anchor-dissipation}
\end{equation}
The data-dependent part satisfies
\begin{equation}
\begin{aligned}
&\int_M\lambda(x)
\;B_{\mathrm{ref}}w\,dV_g\\
&\qquad\leq
C\|B_{\mathrm{ref}}\|_{L^2(M)}^2
+C\|w\|_{L^2(M)}^2.
\end{aligned}
\label{eq:sign-graph-reference-estimate}
\end{equation}
Combining
\eqref{eq:sign-graph-difference-energy}--
\eqref{eq:sign-graph-reference-estimate}, we obtain
\begin{equation}
\begin{aligned}
\frac{d}{dt}\|w(t)\|_{L^2(M)}^2
+\mu\|\Delta_gw(t)\|_{L^2(M)}^2
\leq{}&
C\bigl(1+K(t)^2\bigr)\|w(t)\|_{L^2(M)}^2\\
&+C\|B_{\mathrm{ref}}\|_{L^2(M)}^2.
\end{aligned}
\label{eq:sign-graph-Gronwall-inequality}
\end{equation}
Since $K^2\in L^1(0,T)$, Gronwall's lemma first gives the
$L^\infty(0,T;L^2(M))$ bound in
\eqref{eq:sign-graph-stability}. Integrating
\eqref{eq:sign-graph-Gronwall-inequality} once more and using this
bound gives the stated control of
$\Delta_gw$ in $L^2(Q_T)$. If
$w_{\mathrm{in}}=0$ and $B_{\mathrm{ref}}=0$, then both terms on the
right-hand side of \eqref{eq:sign-graph-stability} vanish, and hence $u=v$.
\end{proof}

\begin{corollary}[Convergence of the full steep-classifier family in two dimensions]
\label{cor:full-steep-convergence}
Assume $d=2$, fix an admissible pair
$(u_{\mathrm{in}},u_{\mathrm{ref}})$, and for each $N$ let $u_N$ be
the unique weak solution of the smooth-classifier problem
\eqref{eq:steep-smooth-problem} with this same pair.
Then there exists a unique state $u$ solving the sign-graph problem
such that
\[
    u_N\longrightarrow u
    \qquad\text{strongly in }L^2(0,T;H_D^2(M)).
\]
No extraction in $N$ is needed for the state variable.
\end{corollary}

\begin{proof}
Every subsequence of $(u_N)_N$ contains, by
Theorem~\ref{thm:steep-classifier-limit}, a further subsequence
converging strongly in $L^2(0,T;H_D^2(M))$ to the state component of
a weak sign-graph solution. Theorem~\ref{thm:sign-graph-uniqueness}
shows that every such state component is the same function $u$.
Relative compactness together with uniqueness of the cluster point
therefore implies convergence of the full sequence.
\end{proof}

\begin{remark}[Why the graph is necessary]
Even though
$\Delta_gu_N\to\Delta_gu$ strongly in $L^2(Q_T)$, one cannot in
general conclude that
$a_N(\Delta_gu_N)\to\operatorname{sgn}_0(\Delta_gu)$ on the zero set
of $\Delta_gu$. For example, if
$\Delta_gu_N=c/N$ on a set of positive measure, then
\[
    a_N(\Delta_gu_N)=\frac{2}{\pi}\arctan(c),
\]
while $\Delta_gu_N\to0$. The maximal graph
$\operatorname{Sign}(0)=[-1,1]$ records exactly this residual freedom.
The state $u$ is nevertheless unique in dimension two by
Theorem~\ref{thm:sign-graph-uniqueness}.
\end{remark}

% Reserve the fourteen equation numbers occupied by the duplicated
% steep-classifier preliminaries in the circulated draft.  This keeps the
% numerical sign identity and source-corrected mass balance at (5.135) and
% (5.137), respectively, for a transparent revision response.

%\addtocounter{equation}{14}

\section{Numerical experiments, exact-solution verification, and direct
sign-graph approximation}
\label{sec:numerics}

The purpose of this section is to illustrate and test the mechanisms that are specific to
the analysis. We address five separate questions.  First, we verify the complete active, anchored, split algorithm against a prescribed exact solution for which every
term is non-negligible.  Second, we solve the discrete maximal-monotone sign-graph problem directly and compare it with the smooth classifiers $a_N$ under coupled refinement of $N$, space, time, and quadrature.
Third, we attach a discretisation error budget to the paired spinodal experiment. Fourth, we test continuous dependence in several independent
directions and over ranges of $N$, $\nu$, and perturbation size.  Fifth, we consider missing-strip reconstruction as an active-response stress test, with hidden truth excluded from the evolution and with the active-minus-passive observable checked under one-factor refinements.

All computations in this section are performed on the closed flat torus $\mathbb T^2=[0,1)^2$. This deliberate restriction permits a direct constant-mode check and isolates the two-dimensional sign-graph and
stability mechanisms.  We do not claim that these experiments validate a particular surface discretisation on a general Riemannian manifold.  No computed state is clipped, and Fourier interpolation is used only for
cross-grid comparison and display.

Throughout, unless explicitly stated otherwise, we solve
\begin{align}
\label{num:eq:model}
\partial_tu
&=
\Delta\!\left(
u^3-u-\mu\Delta u
-\nu a_N(\Delta u)|\nabla u|
\right)
+\lambda(x)\left(u_{\rm ref}^7-u^7\right)+f,\\ \nonumber
a_N(r)&=\frac{2}{\pi}\arctan(Nr).
\end{align}
Thus the constraint law is the admissible map
$b=\beta_8$, $\beta_8(r)=|r|^6r=r^7$, with $q=8>6$.
The forcing $f$ is nonzero only in the exact-solution test.

\subsection{Fourier--Galerkin discretisation and smooth-classifier update}
\label{subsec:num-finiteN-scheme}

Let $n=2K+1$ be the odd number of retained modes in each direction and
\[
\mathcal K_K=\{-K,\ldots,K\}^2,
\qquad
V_K=\operatorname{span}\{e^{2\pi i k\cdot x}:k\in\mathcal K_K\}.
\]
Our Fourier normalization is
\begin{equation}
\label{num:eq:fourier-normalization}
u_h(x)=\sum_{k\in\mathcal K_K}\widehat u_k e^{2\pi i k\cdot x},
\qquad
\widehat u_k=\frac1{n^2}\sum_{j\in\{0,\ldots,n-1\}^2}
u_h(j/n)e^{-2\pi i k\cdot j/n}.
\end{equation}
Under this normalization, Parseval identifies the coefficient norm with the
volume-normalised discrete $L^2$-norm.  For a field on an $m\times m$
padded grid $Q$, our quadrature and average norm are
\[
\langle v\rangle^Q=\frac1{m^2}\sum_{q\in Q}v(q),
\qquad
\|v\|_{2,\rm av}^Q=\bigl(\langle|v|^2\rangle^Q\bigr)^{1/2}.
\]
We write
$P_h^Q$ for evaluation on a padded grid followed by Fourier truncation,
and $A_h=-\Delta_h$.  The padded size is the nearest integer to $pn$
having the same parity as $n$, where $p$ is the padding factor.  Hence
there is no ambiguous Nyquist mode.  Analytic reference, confidence, and
forcing fields are evaluated directly on the padded nodes.

Given $u_h^n$, set
\begin{equation}
\label{num:eq:lagged-data}
G_h^n=|\nabla_hu_h^n|,
\qquad
\mathcal W_h^n=P_h^Q\bigl((u_h^n)^3-u_h^n\bigr),
\qquad
S_n=\max\left\{4,1+3\max\{1,\|u_h^n\|_{\infty,Q}\}^2\right\}.
\end{equation}
The differential predictor $\widetilde u_h^{n+1}\in V_K$ is defined by
\begin{align}
\bigl(I+\Delta tS_nA_h+\Delta t\mu A_h^2\bigr)
\widetilde u_h^{n+1}
={}&
\bigl(I+\Delta tS_nA_h\bigr)u_h^n
-\Delta tA_h\mathcal W_h^n
\nonumber\\
&+\Delta t\nu A_hP_h^Q\!\left[
a_N(-A_h\widetilde u_h^{n+1})G_h^n
\right]
+\Delta tP_h^Qf^{n+1}.
\label{num:eq:finiteN-predictor}
\end{align}
Thus the classifier is implicit in the new Laplacian, whereas only the
nonnegative factor $G_h^n$ is lagged.  At convergence, with
$z_h^{n+1}=\Delta_h\widetilde u_h^{n+1}$, quadrature orthogonality gives
\begin{equation}
\label{num:eq:finiteN-sign}
\left(
P_h^Q[a_N(z_h^{n+1})G_h^n],z_h^{n+1}
\right)_{h,Q}
=
\left\langle
a_N(z_h^{n+1})z_h^{n+1}G_h^n
\right\rangle_{\mathbb T^2}^{Q}
\geq0.
\end{equation}
This is the discrete counterpart of the favourable sign used in the
$L^2$-estimate.  We note that a completely explicit classifier would not retain
\eqref{num:eq:finiteN-sign} and is not used.

The nonlinear source is then advanced on the padded grid by the monotone
pointwise solve
\begin{equation}
\label{num:eq:anchor}
r_h^{n+1}+\Delta t\lambda(x)(r_h^{n+1})^7
=
\widetilde u_h^{n+1}+\Delta t\lambda(x)u_{\rm ref}^7,
\qquad
u_h^{n+1}=P_h^Qr_h^{n+1}.
\end{equation}
The scalar derivative is $1+7\Delta t\lambda r^6\geq1$.  We use safeguarded Newton iteration in
the bracket $[-R,R]$, $R=\max\{1,|\mathrm{rhs}|\}$, initialized by the
predictor clipped to this bracket. Bisection replaces any Newton proposal outside the bracket. The cap is 30 iterations and the stopping rule is
\[
\frac{\|F_{\rm anc}\|_\infty}{1+\|\mathrm{rhs}\|_\infty}
\leq5\times10^{-13}.
\]

The constant-mode calculation is restricted explicitly to the closed torus.
Since $A_h1=0$ and $P_h^Q$ preserves constants, the converged algebraic step obeys
\begin{equation}
\label{num:eq:closed-mass-balance}
\left\langle u_h^{n+1}-u_h^n\right\rangle
=\Delta t\left\langle
\lambda\bigl[u_{\rm ref}^7-(r_h^{n+1})^7\bigr]+f^{n+1}
\right\rangle^Q,
\end{equation}
up to the reported anchoring, projection, and roundoff residuals.  Anchored
runs are not mass conservative.  Moreover,
\eqref{num:eq:closed-mass-balance} is not asserted for the Dirichlet/Navier realization, for which the constant function is not an admissible test function and a boundary flux must be included.

\paragraph{Nonlinear and linear solvers}
Let $c$ denote the coefficient vector in
\eqref{num:eq:finiteN-predictor}, let $F_h(c)$ be its modal residual, and
write
\[
D_k=1+\Delta tS_n|2\pi k|^2+\Delta t\mu|2\pi k|^4.
\]
Damped Newton iteration is initialized by the stabilized passive predictor,
$c_k^{(0)}=(\mathfrak r_h^n)_k/D_k$, where $\mathfrak r_h^n$ is the
right-hand side with the active term omitted.  At Newton iterate $j$, set
\[
\overline w_j=
\left\langle
a_N'(\Delta_hu_h^{(j)})G_h^n
\right\rangle^Q,
\qquad
a_N'(r)=\frac{2N}{\pi(1+N^2r^2)}.
\]
The inverse Fourier multiplier with symbol
\begin{equation}
\label{num:eq:newton-preconditioner}
\left[D_k+\Delta t\nu\overline w_j|2\pi k|^4\right]^{-1}
\end{equation}
is used as a left preconditioner for GMRES.  The linear tolerances are
\[
{\tt rtol}=\max\{10^{-6},\min(10^{-3},0.1R_j)\},
\qquad
{\tt atol}=10^{-12},
\qquad
R_j=\frac{\|F_h(c^{(j)})\|_{\ell^2}}
{\max\{1,\|c^{(j)}\|_{\ell^2}\}}.
\]
GMRES is restarted after 40 inner iterations and is allowed at most 80
restart cycles.  The Newton correction $\delta_j$ is accepted with step
$\alpha_j$ if
\begin{equation}
\label{num:eq:armijo}
\|F_h(c^{(j)}+\alpha_j\delta_j)\|_{\ell^2}
\leq
(1-10^{-4}\alpha_j)\|F_h(c^{(j)})\|_{\ell^2}.
\end{equation}
The search starts at $\alpha_j=1$, halves the step, and fails if $\alpha_j<2^{-16}$.  All finite-classifier experiments reported below use a Newton cap of 100.  A failed GMRES, line search, Newton solve, or anchoring solve terminates the run (an unconverged state is never accepted silently).

\subsection{A direct discrete sign-graph solver}
\label{subsec:num-direct-graph-scheme}

We next replace the smooth classifier in the predictor by the maximal monotone graph itself. This is a direct discrete calculation and does not use a large finite $N$ as an approximation. Suppress the time index and define
\[
B_h=I+\Delta tS_nA_h+\Delta t\mu A_h^2,
\qquad
K_h=\Delta_h=-A_h,
\qquad
\alpha=\Delta t\nu,
\]
with $\mathfrak r_h$ equal to the non-active right-hand side of \eqref{num:eq:finiteN-predictor}. The graph predictor is the unique minimizer in the primal variable of the strongly convex constrained problem
\begin{equation}
\label{num:eq:graph-minimization}
\min_{c\in V_K,\,z}
\left\{
\frac12(B_hc,c)_h-(\mathfrak r_h,c)_h
+\alpha\langle G_h^n|z|\rangle^Q
\right\},
\qquad z=K_hc.
\end{equation}
Its optimality condition is
\begin{equation}
\label{num:eq:discrete-graph-inclusion}
B_hc-\mathfrak r_h-A_hP_h^Qp=0,
\qquad
p=\alpha G_h^n\xi_h,
\qquad
\xi_h\in\operatorname{Sign}(K_hc)
\quad\text{on the padded nodes}.
\end{equation}

We solve \eqref{num:eq:graph-minimization} by a weighted
split-Bregman/ADMM iteration.  With scaled dual variable $y$, penalty $\rho=40\alpha$, and
$\operatorname{shrink}(v,\tau)=\operatorname{sgn}(v)
\max\{|v|-\tau,0\}$, one iteration is
\begin{align}
c^{j+1}
&=(B_h+\rho A_h^2)^{-1}
\left[\mathfrak r_h-\rho A_hP_h^Q(z^j-y^j)\right],
\label{num:eq:admm-c}\\
z^{j+1}
&=\operatorname{shrink}\!\left(
K_hc^{j+1}+y^j,\frac{\alpha G_h^n}{\rho}
\right),
\label{num:eq:admm-z}\\
y^{j+1}
&=y^j+K_hc^{j+1}-z^{j+1}.
\label{num:eq:admm-y}
\end{align}
The diagonal solve in \eqref{num:eq:admm-c} is exact in Fourier space. At the first step we take
$c^0=B_h^{-1}\mathfrak r_h$, $z^0=K_hc^0$, and $y^0=0$.
Subsequently the dual product from the preceding time step is clipped to its new admissible box and used as a warm
start. The cap is 15000 ADMM iterations.

The key quantity is the active product
\begin{equation}
\label{num:eq:active-product}
p_h=\rho y_h,
\qquad
\chi_h=\frac{p_h}{\alpha}=G_h^n\xi_h,
\qquad
|p_h|\leq\alpha G_h^n.
\end{equation}
Where $G_h^n>0$, one admissible multiplier is
$\xi_h=p_h/(\alpha G_h^n)$.  Where $G_h^n=0$, the state equation determines neither $\xi_h$ nor the pointwise representative of $\chi_h$. The graph relation nevertheless forces $\chi_h=G_h^n\xi_h=0$. On
$\{K_hc=0\}$, the graph itself is multivalued. We therefore compare the dynamically identifiable mean-free low modes of the active product and use $\xi_h$ only for a qualitative visualization near the zero set.

For reproducibility, the two stopping residuals are
\begin{align}
R_{\rm pri}^{j+1}
&=
\frac{\|K_hc^{j+1}-z^{j+1}\|_{2,\rm av}^{Q}}
{\max\{1,\|K_hc^{j+1}\|_{2,\rm av}^{Q}\}},
\label{num:eq:admm-primal}\\
R_{\rm dual}^{j+1}
&=
\frac{\|\rho A_hP_h^Q(z^{j+1}-z^j)\|_{\ell^2}}
{\max\{1,\|\mathfrak r_h\|_{\ell^2}\}}.
\label{num:eq:admm-dual}
\end{align}
Both must be below the stated ADMM tolerance.  Independently, we check
\begin{align}
R_{\rm state}
&=
\frac{\|B_hc-\mathfrak r_h-A_hP_h^Qp_h\|_{\ell^2}}
{\max\{1,\|B_hc\|_{\ell^2},\|\mathfrak r_h\|_{\ell^2}\}},
\label{num:eq:graph-state-residual}\\
R_{\rm box}
&=
\frac{\|(|p_h|-\alpha G_h^n)_+\|_\infty}
{\max\{10^{-30},\|\alpha G_h^n\|_\infty\}},
\label{num:eq:graph-box-residual}\\
R_{\rm comp}
&=
\frac{\left\langle
\left|\alpha G_h^n|K_hc|-p_hK_hc\right|
\right\rangle^Q}
{\max\{10^{-30},\langle\alpha G_h^n|K_hc|\rangle^Q\}},
\label{num:eq:graph-complementarity}\\
R_{\rm graph}
&=
\left\|
\xi_h-\Pi_{[-1,1]}(\xi_h+\gamma K_hc)
\right\|_{2,\rm av}^{Q},
\qquad
\gamma=\frac1{\max\{1,\|K_hc\|_{2,\rm av}^{Q}\}},
\label{num:eq:graph-projection-residual}
\end{align}
as well as the normalized complementarity part of the primal--dual gap,
\begin{equation}
\label{num:eq:graph-gap}
\mathcal G_h=
\frac{\langle\alpha G_h^n|K_hc|-p_hK_hc\rangle^Q}
{\max\{10^{-30},\langle\alpha G_h^n|K_hc|\rangle^Q\}}.
\end{equation}

\subsection{Diagnostics and numerical configuration}
\label{subsec:num-diagnostics}

For a retained Fourier state, we record both the ordinary and the
Laplacian-weighted high-mode fractions
\begin{equation}
\label{num:eq:modal-tails}
\rho_0(u_h)=
\frac{\sum_{k\in\mathcal T_K}|\widehat u_k|^2}
{\sum_k|\widehat u_k|^2},
\qquad
\rho_\Delta(u_h)=
\frac{\sum_{k\in\mathcal T_K}|2\pi k|^4|\widehat u_k|^2}
{\sum_k|2\pi k|^4|\widehat u_k|^2},
\end{equation}
where
$\mathcal T_K=\{k:\max(|k_1|,|k_2|)\geq\lceil3K/4\rceil\}$.
We also store \eqref{num:eq:closed-mass-balance}, the finite-$N$ Newton residual, the trajectory maxima of every graph residual, the constraint residual and projection defect, extrema, and the minimum of
$a_N(\Delta_hu_h)\Delta_hu_h$.
More precisely, with
$s_h^{n+1}=\lambda[(u_{\rm ref})^7-(r_h^{n+1})^7]+f^{n+1}$, the reported
one-step and cumulative source--mass defects and the constraint projection defect are
\begin{align}
d_{\rm mass}^{n+1}
&=\left|\langle u_h^{n+1}-u_h^n\rangle
-\Delta t\langle s_h^{n+1}\rangle^Q\right|,
\label{num:eq:mass-step-defect}\\
D_{\rm mass}^{n+1}
&=\left|\langle u_h^{n+1}-u_h^0\rangle
-\Delta t\sum_{j=0}^{n}\langle s_h^{j+1}\rangle^Q\right|,
\label{num:eq:mass-cumulative-defect}\\
d_{\rm proj}^{n+1}
&=\left\|P_h^Qr_h^{n+1}-r_h^{n+1}\right\|_{2,\rm av}^{Q}.
\label{num:eq:anchor-projection-defect}
\end{align}
When a trajectory-level value is quoted below, it is the maximum of the
corresponding recorded defect unless ``cumulative'' is stated explicitly.

For morphology we use
\begin{equation}
\label{num:eq:morphology}
I(u)=\langle|\nabla_hu|\rangle^Q,
\qquad
P(u)=\langle|u^2-1|\rangle^Q,
\qquad
\ell_c(u)=\frac{2\pi}{k_c(u)},
\end{equation}
where, with
$S_k=|\widehat{u-\langle u\rangle}_k|^2$ and
$\kappa_k=2\pi k$,
\begin{equation}
\label{num:eq:structure-factor}
k_c(u)=
\frac{\sum_{k\neq0}|\kappa_k|S_k}{\sum_{k\neq0}S_k}.
\end{equation}
The interface and purity quantities are diagnostics, not Lyapunov
functionals for the active anchored problem.

\begin{table}[tbp]
\centering
\scriptsize
\renewcommand{\arraystretch}{1.08}
\setlength{\tabcolsep}{2.5pt}
\begin{tabular}{@{}
>{\raggedright\arraybackslash}p{0.14\textwidth}
>{\raggedright\arraybackslash}p{0.18\textwidth}
>{\raggedright\arraybackslash}p{0.15\textwidth}
>{\raggedright\arraybackslash}p{0.22\textwidth}
>{\raggedright\arraybackslash}p{0.25\textwidth}@{}}
\hline
Study & Retained/padded modes & Time & Model/data parameters & Error-budget and solver controls\\
\hline
Prescribed exact solution
& $31/93,\ 47/141$;\newline
$63/189,\ 95/285$;\newline controls $63/127,\ 63/189$
& $T=10^{-3}$; $\Delta t=2\!\times\!10^{-4},10^{-4},5\!\times\!10^{-5},2.5\!\times\!10^{-5}$
& $\mu=3\!\times\!10^{-3}$, $\nu=5\!\times\!10^{-2}$, $\kappa=10^{-2}$, $20\leq\lambda\leq30$; analytic $u_\star,u_{\rm ref},f_\star$
& Newton tolerance $10^{-10}$, cap 100; comparison tolerance $10^{-8}$; semidiscrete grids $n=15,19,23,31,47,63,95,127$\\
\hline
Direct sign graph and finite $N$
& L0: $31/63$; L1: $47/95$; L2: $63/189$;\newline
controls $47/95,\ 63/127,\ 47/141$
& $T=5\!\times\!10^{-4}$; $\Delta t=10^{-4},5\!\times\!10^{-5},2.5\!\times\!10^{-5}$
& $\mu=3\!\times\!10^{-3}$; $\nu=2\!\times\!10^{-3},10^{-2}$; $N=8,32,128,512$ on all levels and $N=2$ on L1; $\lambda=0.05$
& ADMM tolerance $10^{-8}$, cap 15000, $\rho=40\Delta t\nu$; finite-$N$ tolerance $10^{-8}$; mean-free product cutoff $|k|\leq6$ with 4/8 sensitivity; padding-3 and $3\!\times\!10^{-9}$ controls; coupled L0--L1--L2 budget\\
\hline
Spinodal production P
& $n/m=47/95$; 20 paired seeds
& $T=8\!\times\!10^{-3}$, $\Delta t=5\!\times\!10^{-5}$
& $\mu=8\!\times\!10^{-4}$; inactive $\nu=0$, active $(\nu,N)=(2\!\times\!10^{-3},32)$; $\lambda=10^{-4}$
& seeds 20260901--20260920; tolerance $10^{-6}$; 50000 paired bootstrap draws\\
\hline
Spinodal audit T/H/A/R/F
& T: $47/95$; H: $63/127$; A: $47/141$; R: $47/95$; F: $63/189$
& T/F use $\Delta t=2.5\!\times\!10^{-5}$; otherwise $5\!\times\!10^{-5}$
& same coupled random fields and physical parameters as P
& first six seeds for T/H/A/R, first four for F; R/F tolerance $10^{-7}$, otherwise $10^{-6}$; one-factor shifts and joint fine audit\\
\hline
Stability production
& $n/m=31/63$
& $T=10^{-3}$, $\Delta t=5\!\times\!10^{-5}$
& $\mu=3\!\times\!10^{-3}$; $N=2,8,32,128,512$; $\nu=0,5\!\times\!10^{-4},10^{-3},2\!\times\!10^{-3},4\!\times\!10^{-3}$; $\lambda=0.2$
& four product-data directions; $\varepsilon=10^{-4},3\!\times\!10^{-4},10^{-3},3\!\times\!10^{-3}$; fixed $S=5$; tolerance $10^{-9}$\\
\hline
Stability refined
& $n/m=47/141$
& $T=10^{-3}$, $\Delta t=2.5\!\times\!10^{-5}$
& $N=2:\ \nu=0.002$;\newline
$N=32:\ \nu=0.0005,0.002,0.004$;\newline
$N=512:\ \nu=0.002$
& all four directions, $\varepsilon=10^{-3}$, tolerance $10^{-10}$\\
\hline
Reconstruction stress test
& P/T/R: $47/95$; H: $63/127$; A: $47/141$
& $T=3\!\times\!10^{-3}$; P/H/A/R use $\Delta t=2.5\!\times\!10^{-5}$; T uses $1.25\!\times\!10^{-5}$
& $\mu=3\!\times\!10^{-3}$; passive $\nu=0$, active $(\nu,N)=(10^{-2},32)$; $\lambda_{\rm known}=6$, $\lambda_{\rm miss}=10^{-3}$
& common $381^2$ evaluation grid; P/T/H/A tolerance $10^{-7}$, R tolerance $10^{-9}$; hidden truth used only after each solve\\
\hline
Common algebraic controls
& odd Fourier grids; direct analytic data on padded nodes
& right-endpoint forcing; first-order splitting
& $b(r)=r^7$; no clipping
& finite-classifier Newton cap 100; GMRES(40), 80 restart cycles; ${\tt atol}=10^{-12}$; Armijo constant $10^{-4}$, minimum step $2^{-16}$; anchor tolerance $5\!\times\!10^{-13}$, cap 30\\
\hline
\end{tabular}
\caption{Numerical configuration.  A pair $n/m$ denotes $n$ retained Fourier modes and $m$ padded nodes per direction.
The checked labels mean: production (P), time (T), spatial resolution (H), padding (A), nonlinear tolerance (R), and, for the spinodal study only, their joint fine configuration (F).}
\label{tab:num-parameters}
\end{table}

Computations use double precision with Python~3.12, NumPy~2.3, SciPy~1.17, pandas~2.2, and Matplotlib~3.10.  

\subsection{Verification against a prescribed exact solution}
\label{subsec:num-exact}

Reference-solution self-convergence can reproduce the same coding error on
every grid.  We therefore prescribe an exact solution and derive an external
forcing analytically.  Let $F$ be the zero-mean periodic primitive
determined by
\begin{equation}
\label{num:eq:flat-primitive}
\begin{aligned}
F'(s)&=\cos^7(2\pi s),\\
F(s)&=
\frac{35}{128\pi}\sin(2\pi s)
+\frac{7}{128\pi}\sin(6\pi s)\\
&\quad
+\frac{7}{640\pi}\sin(10\pi s)
+\frac{1}{896\pi}\sin(14\pi s).
\end{aligned}
\end{equation}
We take
\begin{align}
u_\star(t,x,y)
&=0.10+0.05e^{-20t}
+7.5e^{-15t}\bigl(F(x)+0.7F(y)\bigr),
\label{num:eq:exact-state}\\
u_{\rm ref}(x,y)
&=0.10+5F(x)-6F(y),
\qquad
\lambda(x,y)=25\bigl[1+0.2\cos(2\pi x)\sin(2\pi y)\bigr].
\label{num:eq:exact-data}
\end{align}
The state is genuinely two-dimensional and all its derivatives are finite
Fourier sums.  Since each gradient component contains a seventh power of a
cosine, simultaneous critical points have sufficiently high-order zeros for
$|\nabla u_\star|$ to be at least $C^6$ near every joint zero (and in
particular $C^2$, as required by the forcing).  For this
verification only, write
$a_\kappa(r)=2\arctan(\kappa r)/\pi$ and take
$(\mu,\nu,\kappa)=(3\times10^{-3},5\times10^{-2},10^{-2})$.  Define
\begin{equation}
\label{num:eq:exact-forcing}
f_\star
=\partial_tu_\star
-\Delta\!\left(
u_\star^3-u_\star-\mu\Delta u_\star
-\nu a_\kappa(\Delta u_\star)|\nabla u_\star|
\right)
-\lambda(u_{\rm ref}^7-u_\star^7).
\end{equation}
All derivatives in \eqref{num:eq:exact-forcing} are evaluated directly from
\eqref{num:eq:flat-primitive}, independently of the production differentiation
and classifier routines.  The auxiliary parameter $\kappa=10^{-2}$ is used
only to verify the smooth-classifier algorithm.
At $t=0$, the RMS magnitudes of
$u_t$, $\Delta W'(u)$, $\mu\Delta^2u$, the active term, the anchoring
term, and $f_\star$ are respectively
\[
8.68,\quad73.8,\quad82.7,\quad66.9,\quad14.1,\quad119.2.
\]
Thus the verification does not suppress the active or anchored parts of the
implementation.

For numerical state $u_h$, define the terminal errors
\begin{equation}
\label{num:eq:exact-errors}
E_0(T)=\frac{\|u_h(T)-u_\star(T)\|_{2,\rm av}}
{\|u_\star(T)\|_{2,\rm av}},
\qquad
E_\Delta(T)=
\frac{\|\Delta_hu_h(T)-\Delta u_\star(T)\|_{2,\rm av}}
{\|\Delta u_\star(T)\|_{2,\rm av}},
\end{equation}
and the right-endpoint space--time error
\begin{equation}
\label{num:eq:exact-h2-error}
E_{H^2,t}^2=
\frac{\sum_{j=1}^{J}\Delta t\left\langle
|e^j|^2+|\nabla_he^j|^2+|\Delta_he^j|^2
\right\rangle}
{\sum_{j=1}^{J}\Delta t\left\langle
|u_\star^j|^2+|\nabla u_\star^j|^2+|\Delta u_\star^j|^2
\right\rangle},
\qquad e^j=u_h^j-u_\star^j.
\end{equation}
Figure~\ref{fig:num-exact} summarizes the exact fields, convergence, and
term balance. The corresponding temporal values are listed in Table~\ref{tab:num-exact-temporal}.

\begin{figure}[tbp]
\centering
\includegraphics[width=\textwidth]{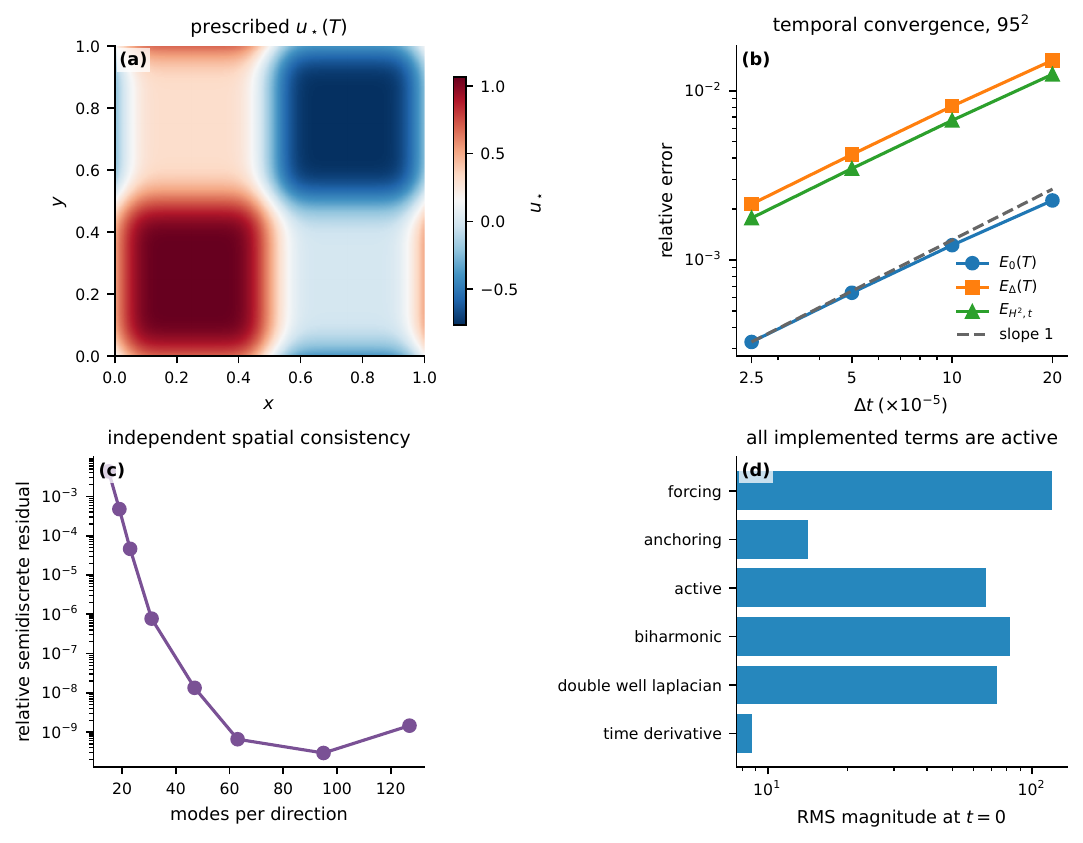}
\caption{Verification against the prescribed exact solution.
(a) $u_\star(T)$.  (b) Temporal convergence of the terminal and
space--time errors on the $95^2$ Fourier space.  (c) Independently formed
semidiscrete residual versus Fourier resolution.  (d) RMS magnitudes of all
terms in \eqref{num:eq:model} at $t=0$.}
\label{fig:num-exact}
\end{figure}

\begin{table}[tbp]
\centering
\small
\begin{tabular}{c|cc|cc|cc}
\hline
$\Delta t$ & $E_0(T)$ & order & $E_\Delta(T)$ & order
& $E_{H^2,t}$ & order\\
\hline
$2.0\times10^{-4}$ & $2.248\times10^{-3}$ & --
& $1.513\times10^{-2}$ & -- & $1.250\times10^{-2}$ & --\\
$1.0\times10^{-4}$ & $1.223\times10^{-3}$ & 0.878
& $8.105\times10^{-3}$ & 0.900 & $6.680\times10^{-3}$ & 0.904\\
$5.0\times10^{-5}$ & $6.401\times10^{-4}$ & 0.934
& $4.199\times10^{-3}$ & 0.949 & $3.468\times10^{-3}$ & 0.946\\
$2.5\times10^{-5}$ & $3.278\times10^{-4}$ & 0.966
& $2.138\times10^{-3}$ & 0.974 & $1.770\times10^{-3}$ & 0.971\\
\hline
\end{tabular}
\caption{Temporal errors for the complete forced active and anchored scheme
at $n=95$, padding factor 3, and $T=10^{-3}$.}
\label{tab:num-exact-temporal}
\end{table}

The three observed orders tend to one, as expected for the split update.  At
the finest step the derivative-sensitive terminal error is
\(2.14\times10^{-3}\), closing the temporal error budget in a norm aligned
with the compactness theorem.  At fixed \(\Delta t=2.5\times10^{-5}\), the
time-discrete errors are already unchanged to the displayed digits between
\(n=31\) and \(n=95\).  Spatial accuracy is therefore assessed independently
by substituting the exact Fourier coefficients into the semidiscrete
right-hand side.  If \(\mathcal F_h\) denotes that padded semidiscrete
right-hand side, we report
\begin{equation}
\label{num:eq:exact-semidiscrete-residual}
R_{\rm sd}(n)=
\frac{\left\|
\widehat{\partial_tu_\star}-\widehat{\mathcal F_h(u_\star)}
\right\|_{\ell^2(\mathcal K_K)}}
{\left\|\widehat{\partial_tu_\star}\right\|_{\ell^2(\mathcal K_K)}}.
\end{equation}
This relative residual decreases from
$7.75\times10^{-7}$ at $n=31$ to $1.34\times10^{-8}$ at $n=47$ and $6.59\times10^{-10}$ at $n=63$, after which differentiation and cancellation reach a numerical floor.

At $n=63$, changing the padding factor from 2 to 3 changes the final state by only $8.98\times10^{-11}$ relatively.  Relaxing the Newton tolerance from $10^{-10}$ to $10^{-8}$ changes it by $1.17\times10^{-9}$. The strict runs have source-corrected one-step mass defects below $1.1\times10^{-14}$. The maximum positive-time
Laplacian-tail fraction falls from $2.20\times10^{-8}$ at $n=31$ to $7.24\times10^{-16}$ at $n=95$. 
%These tests jointly verify the forcing, double-well, biharmonic, active, anchoring, projection, and constant-mode parts of the implementation against an independently differentiated exact state.

\subsection{Finite classifiers versus the direct sign graph}
\label{subsec:num-signgraph-experiment}

The theorem-specific experiment starts from
\begin{equation}
\label{num:eq:graph-initial-data}
\begin{aligned}
q(x,y)&=\cos(2\pi x)+0.70\cos(2\pi y)
+0.35\cos(2\pi(x+y))\\
&\quad+0.20\sin(2\pi(2x-y)),\\
u_{\rm in}&=u_{\rm ref}=\frac{0.20}{2.25}q,
\qquad \lambda=0.05.
\end{aligned}
\end{equation}
We use the coupled levels L0--L2 in Table~\ref{tab:num-parameters} and the
same source split for the smooth and graph predictors.  For each finite
\(N\), define the relative state distances
\begin{align}
E_N^{H^2}
&=
\left[
\frac{\sum_{j=1}^{J}\Delta t\left(
\|u_N^j-u_{\rm Sign}^j\|_{2,\rm av}^2
+\|\Delta_h(u_N^j-u_{\rm Sign}^j)\|_{2,\rm av}^2
\right)}
{\sum_{j=1}^{J}\Delta t\left(
\|u_{\rm Sign}^j\|_{2,\rm av}^2
+\|\Delta_hu_{\rm Sign}^j\|_{2,\rm av}^2
\right)}
\right]^{1/2},
\label{num:eq:graph-state-h2}\\
E_N^{\infty,0}
&=
\frac{\max_j\|u_N^j-u_{\rm Sign}^j\|_{2,\rm av}}
{\max_j\|u_{\rm Sign}^j\|_{2,\rm av}}.
\label{num:eq:graph-state-linf}
\end{align}
The first norm is equivalent to the periodic \(L^2(0,T;H^2)\)-norm.  Write
\(G_{N,h}^{j-1}=|\nabla_hu_{N,h}^{j-1}|\) for the lagged factor in the
finite-\(N\) predictor.  Because the multiplier need not converge strongly
on \(\{\Delta u=0\}\), and because the Laplacian annihilates the constant
mode of the active product, we compare only the dynamically identifiable
mean-free Fourier indices \(0<|k|\leq6\), where
\(|k|=(k_1^2+k_2^2)^{1/2}\).  The cutoff 6 was fixed before the final
comparison rerun. Cutoffs 4 and 8 are retained as sensitivity controls:
\begin{equation}
\label{num:eq:graph-product-error}
E_N^\chi=
\frac{\left\|\Pi_{0<|k|\leq6}\!\left[
a_N(\Delta_h\widetilde u_{N,h}^{\,j})G_{N,h}^{j-1}-\chi_h^j
\right]\right\|_{L^2_tL^2_x}}
{\|\Pi_{0<|k|\leq6}\chi_h\|_{L^2_tL^2_x}}.
\end{equation}

\begin{figure}[tbp]
\centering
\includegraphics[width=\textwidth]{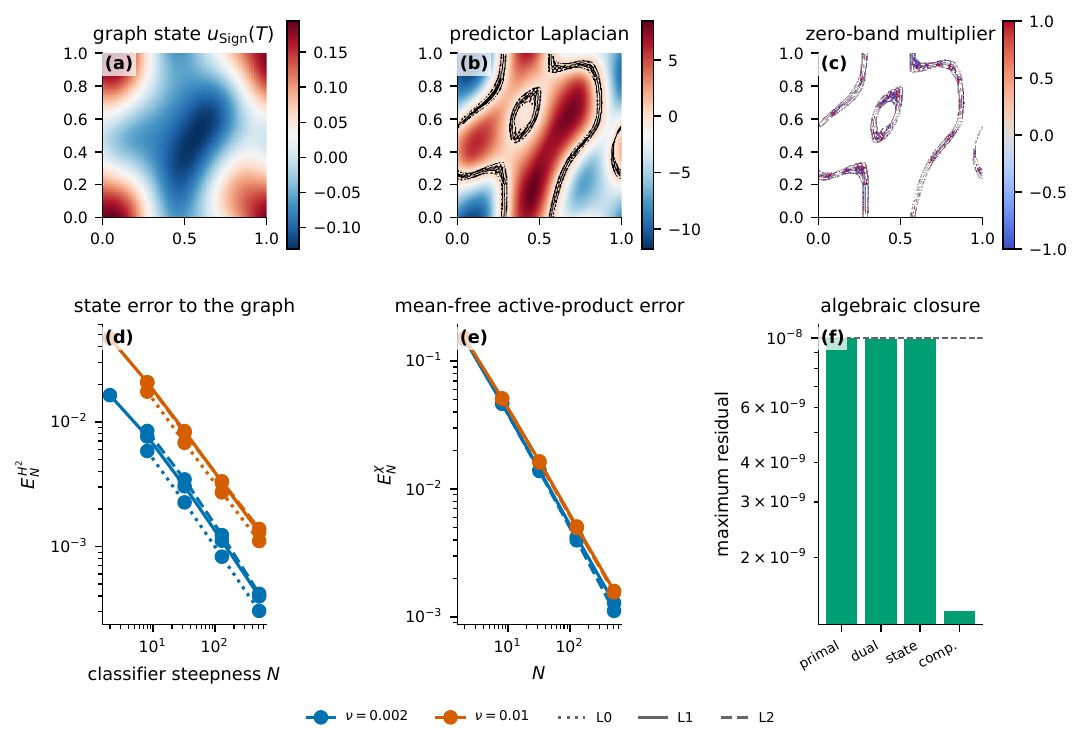}
\caption{Direct discrete sign-graph calculation.
(a) Graph state at \(T\) on the finest L2 level for \(\nu=10^{-2}\).
(b) Its predictor Laplacian and zero contour.
(c) Computed multiplier on
\(|\Delta_h\widetilde u_h|\leq10^{-4}\|\Delta_h\widetilde u_h\|_{2,\rm av}\),
this band contains \(7.49\%\) of the padded nodes, \(43.0\%\) of which have
\(|\xi_h|<0.99\).  Interior values correspond to zero split variables up to
the reported primal residual.
(d) \(E_N^{H^2}\) for two active strengths and three coupled levels.
(e) Mean-free low-mode active-product discrepancy \(E_N^\chi\).
(f) Maximum primal, dual, state, and complementarity residuals for the
padding-3 central graph trajectory.}
\label{fig:num-direct-graph}
\end{figure}

Figure~\ref{fig:num-direct-graph} displays the directly computed graph state,
one multiplier representative, and both finite-\(N\) comparison metrics.
At the central L1 discretisation, increasing \(N\) from 2 to 512 decreases
\(E_N^{H^2}\) from \(1.634\times10^{-2}\) to \(3.985\times10^{-4}\) for
\(\nu=2\times10^{-3}\), and from \(4.681\times10^{-2}\) to
\(1.304\times10^{-3}\) for \(\nu=10^{-2}\).  The corresponding
low-mode product errors fall from \(0.1473\) to \(1.287\times10^{-3}\) and
from \(0.1511\) to \(1.596\times10^{-3}\).  On the finest L2 level the
complete common-\(N\) data are
\[
\begin{array}{c|rrrr}
 &N=8&N=32&N=128&N=512\\ \hline
E_N^{H^2},\ \nu=0.002
 &8.436{\times}10^{-3}&3.448{\times}10^{-3}
 &1.233{\times}10^{-3}&4.170{\times}10^{-4}\\
E_N^{H^2},\ \nu=0.01
 &2.073{\times}10^{-2}&8.177{\times}10^{-3}
 &3.331{\times}10^{-3}&1.381{\times}10^{-3}\\
E_N^\chi,\ \nu=0.002
 &4.641{\times}10^{-2}&1.381{\times}10^{-2}
 &3.971{\times}10^{-3}&1.115{\times}10^{-3}\\
E_N^\chi,\ \nu=0.01
 &5.075{\times}10^{-2}&1.608{\times}10^{-2}
 &4.967{\times}10^{-3}&1.554{\times}10^{-3}
\end{array}
\]
Every row decreases monotonically.  Least-squares fits on these four points
give \(N^{-0.72}\) and \(N^{-0.65}\) for the state metric, and \(N^{-0.90}\)
and \(N^{-0.84}\) for the low-mode product metric, at the two respective
active strengths.  These comparisons use a direct solution of the discrete
graph inclusion.
The low-mode conclusion is insensitive to the auxiliary cutoff: for
cutoffs 4, 6, and 8 the fitted L2 product rates are, respectively, N$^{-0.912}$, \(N^{-0.897}\), and \(N^{-0.880}\) at \(\nu=0.002\), and \(N^{-0.859}\), \(N^{-0.839}\), and \(N^{-0.825}\) at \(\nu=0.01\).

The raw graph-to-graph $H^2$-equivalent discrepancies decrease from L0--L1 to L1--L2 as
\[
1.275\times10^{-2}\longrightarrow8.067\times10^{-3}
\quad(\nu=0.002),\qquad
2.039\times10^{-2}\longrightarrow1.309\times10^{-2}
\quad(\nu=0.01).
\]
The corresponding terminal \(L^2\) discrepancies decrease from
$2.855\times10^{-5}$ to $1.523\times10^{-5}$, and from
$7.289\times10^{-5}$ to $4.041\times10^{-5}$. Terminal low-mode product discrepancies decrease from $1.451\times10^{-2}$ to $6.452\times10^{-3}$, and from $1.076\times10^{-2}$ to $5.351\times10^{-3}$. Thus every coupled diagnostic decreases under refinement, but the raw graph state is not yet at a continuum numerical plateau.

\begin{table}[tbp]
\centering
\small
\begin{tabular}{c|ccc}
\hline
one-factor control
& relative space--time \(H^2\)
& terminal \(L^2\)
& terminal low-mode product\\
\hline
time step & \(9.189\times10^{-3}\) & \(4.165\times10^{-5}\)
& \(4.049\times10^{-3}\)\\
space & \(1.067\times10^{-2}\) & \(4.371\times10^{-6}\)
& \(3.307\times10^{-3}\)\\
padding & \(2.986\times10^{-3}\) & \(1.952\times10^{-6}\)
& \(1.494\times10^{-3}\)\\
algebraic & \(1.739\times10^{-6}\) & \(7.046\times10^{-10}\)
& \(1.411\times10^{-6}\)\\
\hline
\end{tabular}
\caption{One-factor error budget for the direct graph trajectory at
\(\nu=10^{-2}\).  The algebraic row compares identical padding-3
configurations at ADMM tolerances \(10^{-8}\) and \(3\times10^{-9}\).  The
product column uses the mean-free projection in
\eqref{num:eq:graph-product-error}.}
\label{tab:num-graph-budget}
\end{table}

Table~\ref{tab:num-graph-budget} gives the one-factor controls used to
separate constitutive and discretisation effects.  For \(\nu=10^{-2}\),
the raw state signal is clearly larger than every
one-factor \(H^2\) shift at \(N=8\), the low-mode product signal is resolved
through \(N=32\).  At larger \(N\) the constitutive discrepancy falls below
one or more discretisation controls.  Nevertheless, the paired
finite-\(N\)-to-graph curves themselves are stable from L1 to L2: at
\(N=512\), \(E_N^{H^2}\) changes by \(4.45\%\) and \(5.59\%\) for the two
active strengths, while \(E_N^\chi\) changes by \(15.4\%\) and \(2.76\%\),
respectively, percentages are relative to the L2 values.  We therefore interpret the result as resolved direct-discrete graph evidence through moderate \(N\), with a refinement-consistent convergence trend at larger \(N\), not as continuum closure of the graph trajectory or
strong multiplier convergence.

For the padding-3 central solve, the maxima over all time steps are
\[
R_{\rm pri}=1.000\times10^{-8},\quad
R_{\rm dual}=9.961\times10^{-9},\quad
R_{\rm state}=9.961\times10^{-9},\quad
R_{\rm comp}=1.347\times10^{-9},
\]
while
\[
R_{\rm graph}=1.000\times10^{-8},\qquad
\mathcal G_h=1.347\times10^{-9},\qquad
\max_jR_{\rm box}^j=3.61\times10^{-14},
\]
and the terminal box residual is \(R_{\rm box}(T)=3.56\times10^{-14}\).
The largest ADMM count is 5691 of the 15000 permitted iterations.
The strict algebraic rerun is three to five orders of magnitude below the
discretisation controls.
The plot of \(\xi_h\) near the small-Laplacian region illustrates an
admissible interior multiplier selection that is invisible in a pointwise
algebraic estimate
such as \(0\leq |r|-a_N(r)r\leq2/(\pi N)\).  The state and active-product
errors, rather than that bound alone, supply the numerical evidence relevant
to the sign-graph theorem.

\subsection{Spinodal coarsening with a discretisation-adjusted paired design}
\label{subsec:num-spinodal-budget}

For seed \(s\), let \(u_{\rm in}^{(s)}=u_{\rm ref}^{(s)}\) be a mean-zero
Gaussian Fourier field filtered by
\begin{equation}
\label{num:eq:spinodal-filter}
\widehat u_k\longmapsto
\exp\!\left[-\left(\frac{|k|^2}{7^2}\right)^2\right]\widehat u_k
\end{equation}
and rescaled to standard deviation $0.1$.  A master $47^2$ realization is used for every configuration, the $63^2$ initial state is obtained by exact Fourier embedding.  Thus refinement changes the discretisation, not the
random field. The inactive and active members of every pair have identical initial and reference data. Production uses 20 predetermined seeds, while the first six seeds form the time, space, padding, and tolerance audit subset
and the first four form the joint fine check.

For each observable $O\in\{\ell_c,I,P\}$, write
$d_s^C=O_{\rm active}^{C,s}(T)-O_{\rm inactive}^{C,s}(T)$, where P denotes
the production configuration and
$C\in\{\mathrm T,\mathrm H,\mathrm A,\mathrm R,\mathrm F\}$ an audit configuration. We bootstrap the paired production mean with 50000 draws. We use the ordinary paired bootstrap and the $2.5\%$ and $97.5\%$
percentiles with NumPy's linear quantile rule. For metric index $m=0,1,2$, the endpoint production seed is $20263000+m$, the check seed
is $20264000+113m+\sum_{\mathtt c\in C}{\rm ord}(\mathtt c)$, and the time-series seeds are fixed analogously in the accompanying driver.
To keep seed variability separate from discretisation sensitivity, set $\delta_s^C=d_s^C-d_s^{\rm P}$, and let
$q_\alpha^*(\overline\delta^{,C})$ denote the $\alpha$-quantile of the paired bootstrap distribution of its sample mean. Define the numerical envelope
\begin{equation}
\label{num:eq:spinodal-envelope}
\mathcal E_O=
\max_C\max\left\{
\left|q_{0.025}^*(\overline\delta^{,C})\right|,
\left|q_{0.975}^*(\overline\delta^{,C})\right|
\right\},
\end{equation}
using only seeds shared by P and \(C\).  If
\([L_O,U_O]\) is the paired seed interval, we report the predeclared
sensitivity interval
\([L_O-\mathcal E_O,U_O+\mathcal E_O]\).  This enlargement combines a sampling interval with a numerical sensitivity envelope (not presented as a formal simultaneous 95\% confidence interval).

\begin{figure}[tbp]
\centering
\includegraphics[width=\textwidth]{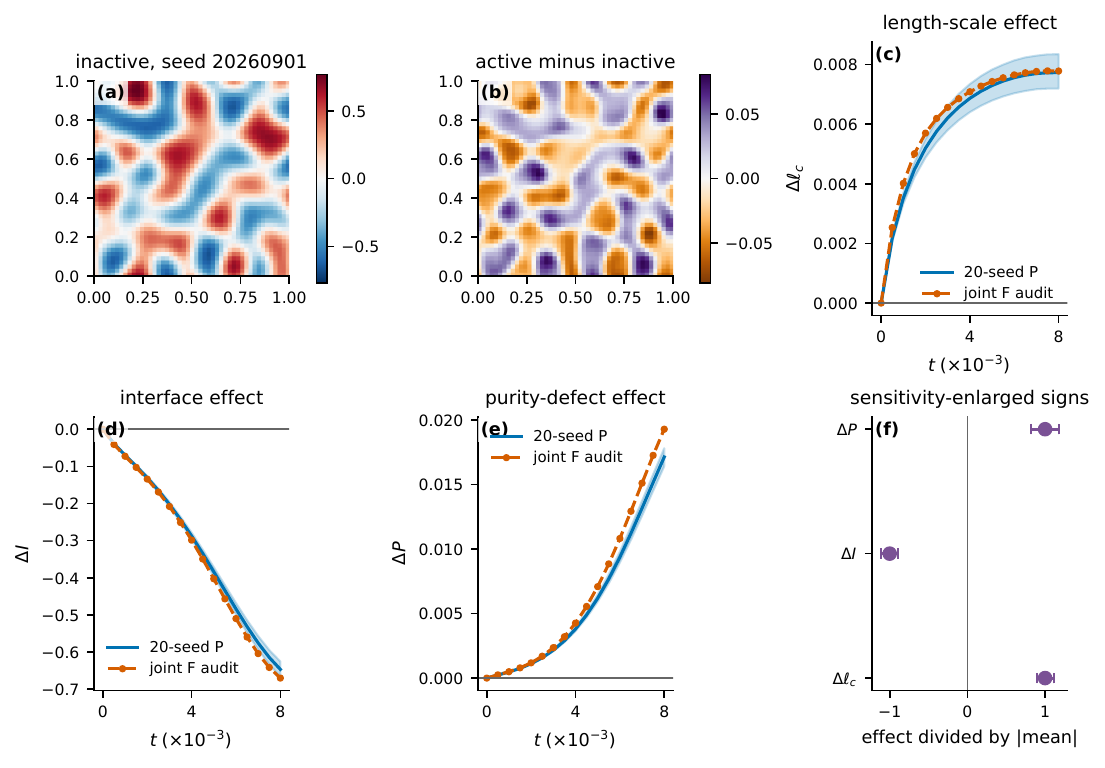}
\caption{Paired spinodal experiment with a numerical error budget.
(a) Terminal inactive field for seed 20260901.
(b) Terminal active-minus-inactive difference for the same seed, shown on its own symmetric colour scale.
(c--e) Mean paired active-minus-inactive effects for characteristic length, interface proxy, and purity defect. Bands are pointwise paired 95\% bootstrap intervals and the joint fine audit is shown separately.
(f) Terminal sensitivity intervals from
\eqref{num:eq:spinodal-envelope}, divided by the magnitude of each
production mean so that all three sign margins are visible.}
\label{fig:num-spinodal-budget}
\end{figure}

At \(T=8\times10^{-3}\), the production paired effects are
\begin{equation}
\label{num:eq:spinodal-results}
\begin{aligned}
\Delta\ell_c&=7.728\times10^{-3}
\quad[7.194\times10^{-3},\,8.354\times10^{-3}],\\
\Delta I&=-0.64665
\quad[-0.66650,\,-0.62533],\\
\Delta P&=1.7136\times10^{-2}
\quad[1.6402\times10^{-2},\,1.7816\times10^{-2}],
\end{aligned}
\end{equation}
where brackets denote the paired seed intervals.  The absolute numerical
envelopes are \(2.448\times10^{-4}\), \(5.125\times10^{-2}\), and
\(2.307\times10^{-3}\), respectively, or \(3.17\%\), \(7.92\%\), and
\(13.47\%\) of the production effects.  The resulting sensitivity intervals
are
\[
[0.006949,\,0.008599],\qquad
[-0.717749,\,-0.574083],\qquad
[0.014094,\,0.020124].
\]
All 20 length differences are positive, all 20 interface differences are
negative, and all 20 purity-defect differences are positive.

Time-step halving is the only material numerical shift: it changes the three
effects by \(9.79\times10^{-5}\), \(-4.72\times10^{-2}\), and
\(2.16\times10^{-3}\), respectively.  Spatial refinement changes them by at
most \(0.09\%\), threefold padding by at most \(0.042\%\), and the tighter
nonlinear tolerance by less than \(10^{-6}\) relatively. The joint fine configuration reproduces the temporal shift. Note, it is not  revealing a hidden space, padding, or solver interaction.

Across the recorded production and audit diagnostics, the maximum
finite-\(N\) residual is
\(9.95\times10^{-7}\), the cumulative source--mass defect is
\(8.25\times10^{-14}\), the anchoring projection defect is
\(2.63\times10^{-13}\), and the maximum recorded positive-time ordinary and
Laplacian-weighted modal tails are \(8.63\times10^{-7}\) and
\(5.16\times10^{-4}\).  The minimum classifier-sign density is
\(1.40\times10^{-11}>0\).  As a descriptive post-step proxy, between
\(98.38\%\) and \(99.44\%\) of active terminal nodes satisfy
\(|N\Delta_hu_h|\geq10\). We note that this is not substituted for the implicit predictor argument in any error estimate.

Figure~\ref{fig:num-spinodal-budget} displays the paired histories and the separate sampling and numerical-sensitivity information. The result is therefore a resolved coarsening effect: the active term increases the characteristic length and decreases interfacial content, while the purity defect increases. Thus this regime produces less interface but does not move the field closer to the pure values $\pm1$. This conclusion uses paired random data and a coupled numerical envelope. Note that a bootstrap interval alone would quantify seed variability but not discretisation bias.

\subsection{Multi-direction continuous-dependence stress test}
\label{subsec:num-stability-multidirection}

The analytical estimate measures perturbations in both the initial state and
the anchored product \(b(u_{\rm ref})\).  We therefore perturb these two data
components directly.  Let
\[
\phi_{k}^{\sin}=\sqrt2\sin(2\pi k\cdot x),
\qquad
\phi_{k}^{\cos}=\sqrt2\cos(2\pi k\cdot x),
\]
and define the unit-\(L^2\) fields
\begin{align}
\psi_1&=\frac{\phi_{(1,2)}^{\sin}+0.4\phi_{(3,-1)}^{\cos}}
{\sqrt{1+0.4^2}},
&
\psi_2&=\frac{\phi_{(2,1)}^{\cos}-0.35\phi_{(1,3)}^{\sin}}
{\sqrt{1+0.35^2}},
\nonumber\\
\psi_3&=\frac{\phi_{(4,3)}^{\sin}+0.3\phi_{(5,-2)}^{\cos}}
{\sqrt{1+0.3^2}}.
\label{num:eq:stability-directions}
\end{align}
The four normalized product-data directions are
\begin{equation}
\label{num:eq:four-data-directions}
(\psi_1,0),\qquad
(0,\psi_1),\qquad
2^{-1/2}(\psi_1,\psi_2),\qquad
2^{-1/2}(\psi_3,-\psi_2).
\end{equation}
For a direction \((\psi_{\rm in},\psi_b)\), we use
\[
u_{\rm in}^{\varepsilon}=u_{\rm in}+\varepsilon\psi_{\rm in},
\qquad
b(u_{\rm ref}^{\varepsilon})=b(u_{\rm ref})+\varepsilon\psi_b,
\]
where the real seventh root defines \(u_{\rm ref}^{\varepsilon}\).  Hence the
combined data distance is \(\varepsilon\) to roundoff, including the pure
reference direction.  The common baseline is
\begin{align*}
u_{\rm in}&=\frac{0.20}{2.25}\left[
\cos(2\pi x)+0.70\cos(2\pi y)+0.35\cos(2\pi(x+y))
+0.20\sin(2\pi(2x-y))\right],\\
u_{\rm ref}&=0.82+0.04\left[
\cos(2\pi x)+0.6\sin(2\pi y)+0.2\cos(2\pi(x+y))\right],
\qquad \lambda=0.2.
\end{align*}
In this stress test only, \(S=5\) is held fixed in both members of every
pair, so the response is not contaminated by a data-dependent change in
stabilization.  A targeted audit of the extreme tested parameters gives
\(\max\|u_h\|_{\infty,Q}<0.198\); hence the adaptive floor in
\eqref{num:eq:lagged-data} equals \(4<5\) throughout those controls.

To remove the tautological supremum at \(t=0\) for initial-data
perturbations, we report both the theorem-aligned quantity and a
positive-time stress-test quantity.  With
\(w^\varepsilon=u^\varepsilon-u\), set
\begin{align}
H_{\rm th}(\varepsilon)^2
&=\sup_{0\leq t\leq T}\|w^\varepsilon(t)\|_{2,\rm av}^2
+\int_0^T\|\Delta_hw^\varepsilon(t)\|_{2,\rm av}^2\,dt,
\label{num:eq:stability-theorem-metric}\\
H_{\rm dyn}(\varepsilon)^2
&=\sup_{T/4\leq t\leq T}\|w^\varepsilon(t)\|_{2,\rm av}^2
+\int_0^T\|\Delta_hw^\varepsilon(t)\|_{2,\rm av}^2\,dt,
\qquad
A_{\rm dyn}=\frac{H_{\rm dyn}}{\varepsilon}.
\label{num:eq:stability-dynamic-metric}
\end{align}
We also verify pointwise at every stored time that
\begin{equation}
\label{num:eq:stability-monotonicity}
\bigl[a_N(\Delta_hu^\varepsilon)-a_N(\Delta_hu)\bigr]
\bigl[\Delta_hu^\varepsilon-\Delta_hu\bigr]\geq0
\end{equation}
up to roundoff.

\begin{figure}[tbp]
\centering
\includegraphics[width=\textwidth]{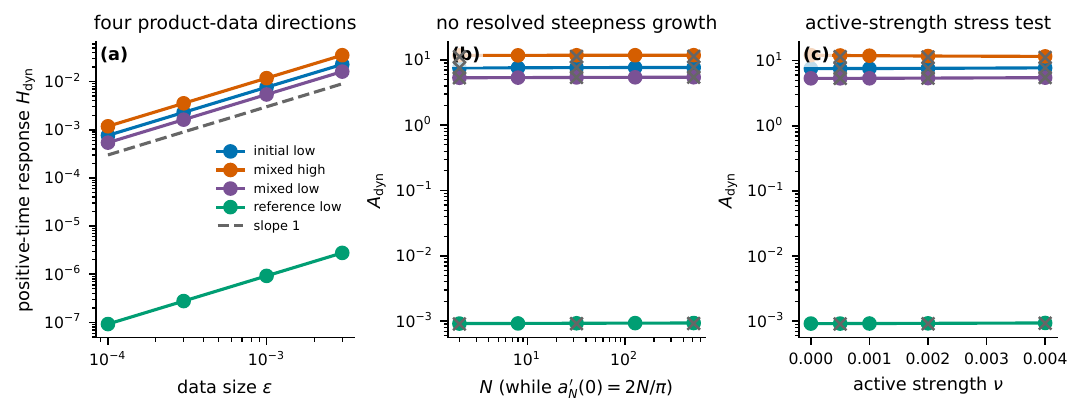}
\caption{Multi-direction continuous-dependence experiment.
(a) \(H_{\rm dyn}\) versus \(\varepsilon\) for four directions at
\((N,\nu)=(32,2\times10^{-3})\).  (b) Amplification \(A_{\rm dyn}\) versus
\(N\), despite \(a_N'(0)=2N/\pi\).  (c) \(A_{\rm dyn}\) versus \(\nu\).
Crosses in (b,c) are the refined configurations in
Table~\ref{tab:num-parameters}.}
\label{fig:num-stability-multidirection}
\end{figure}

For every direction, the fitted log--log slopes of
\(H_{\rm dyn}\) against \(\varepsilon\) are
\(0.9978\), \(1.0000\), \(0.9985\), and \(0.9989\) for the initial-only,
reference-only, low-mode mixed, and high-mode mixed directions,
respectively.  In the same order, the production ranges of \(A_{\rm dyn}\)
over \(N=2,\ldots,512\) are
\[
7.540\text{--}7.663,\qquad
9.156\times10^{-4}\text{--}9.334\times10^{-4},\qquad
5.332\text{--}5.430,\qquad
11.733\text{--}11.815.
\]
Their spans are \(1.61\%\), \(1.93\%\), \(1.81\%\), and \(0.70\%\) of
the respective directional means even though \(a_N'(0)\) increases by a
factor \(256\).  On the refined subset \(N\in\{2,32,512\}\), the
corresponding ranges are \(7.540\)--\(7.684\),
\(9.002\times10^{-4}\)--\(9.197\times10^{-4}\), \(5.332\)--\(5.441\),
and \(11.229\)--\(11.350\).

Over the production $\nu$-sweep, the four ranges are
\(7.567\)--\(7.753\),
\(9.182\times10^{-4}\)--\(9.383\times10^{-4}\),
\(5.352\)--\(5.483\), and \(11.630\)--\(12.131\). Their relative spans are
\(2.43\%\), \(2.18\%\), \(2.44\%\), and \(4.21\%\).  The refined values
at \(\nu=5\times10^{-4},2\times10^{-3},4\times10^{-3}\) preserve the same
directional trends.  The largest relative production-to-refined shift is
\(4.29\%\) for the high-mode mixed direction. Its largest absolute shift is
\(5.07\times10^{-1}\).  At the central pair
\((N,\nu)=(32,2\times10^{-3})\), the relative shifts are \(0.214\%\),
\(1.44\%\), \(0.200\%\), and \(4.02\%\), respectively.  We therefore
attach a conservative \(4.3\%\) coupled-discretisation envelope to the
high-mode curve; the worst-case envelopes for the other three directions
are \(0.28\%\), \(1.69\%\), and \(0.29\%\).

The minimum value in \eqref{num:eq:stability-monotonicity} is zero. The largest terminal Laplacian-weighted difference tail is $1.85\times10^{-2}$ (the largest terminal unweighted tail is $2.06\times10^{-4}$), the maximum nonlinear and projection defects divided by $\varepsilon$ are $9.67\times10^{-7}$ and $7.71\times10^{-12}$.
The computed product-data distance differs from $\varepsilon$ by at most $3.33\times10^{-16}$ relatively.

Figure~\ref{fig:num-stability-multidirection} shows these four directional responses and their refined controls.  The coupled refinements thus retain uniformly bounded amplification across the tested $N$-range. We note that this is potentially more informative than a single perturbation direction: it tests the theorem's monotonicity mechanism in a regime where
\(a_N'(0)=2N/\pi\) itself is not uniformly bounded. 

\subsection{Audited missing-strip active-response stress test}
\label{subsec:num-reconstruction-audit}

We finally consider reconstruction as a controlled test of whether the active response has a visually and numerically resolved effect (not as evidence for the graph limit).  The hidden target $u^\dagger$ is a smooth three-bar field with bar centres $0.19,0.50,0.81$, half-width $0.055$, and interface parameter $0.04$.
Let \(M(y)\) be a smooth horizontal-strip mask of half-width \(0.105\) and
transition parameter \(0.03\).  The corrupted observation is
\[
u_{\rm obs}=(1-M)u^\dagger,
\qquad
u_{\rm in}=u_{\rm ref}=u_{\rm obs},
\qquad
\lambda=6(1-M)+10^{-3}M.
\]
Thus \(u^\dagger\) is never supplied to the evolution, it is loaded only after a run to evaluate error.  We compare a passive solve, $\nu=0$, with
the fixed active stress value \((\nu,N)=(10^{-2},32)\).  This value is five
times the perturbative strength used in the spinodal study and was fixed
before the refined audit, not selected by minimizing truth error.

On the common \(381^2\) evaluation grid we use
\[
E_M(u)=
\frac{\langle M|u-u^\dagger|^2\rangle^{1/2}}
{\langle M|u^\dagger|^2\rangle^{1/2}},
\qquad P(u)=\langle|u^2-1|\rangle,
\qquad I(u)=\langle|\nabla_hu|\rangle.
\]
A predeclared coarse response scan at
\(\nu=0,0.002,0.005,0.01,0.02\) is monotone in all three diagnostics:
from the passive value to \(\nu=0.02\), \(E_M\) increases from
\(0.7911\) to \(0.8162\), \(P\) increases from \(0.8095\) to \(0.8428\),
and \(I\) decreases from \(5.415\) to \(4.531\).  

At the production configuration, the passive-to-active changes are
\begin{equation}
\label{num:eq:reconstruction-effects}
\Delta E_M=1.2856\times10^{-2},
\qquad
\Delta P=1.7408\times10^{-2},
\qquad
\Delta I=-4.5379\times10^{-1}.
\end{equation}
The full active-minus-passive field satisfies
\[
\|u_{\rm a}-u_{\rm p}\|_{2,\rm av}=3.0937\times10^{-2},
\qquad
\|u_{\rm a}-u_{\rm p}\|_\infty=5.4737\times10^{-2}.
\]
Consequently, the common-scale final states remain close, while their
difference is plainly resolved once plotted on its own symmetric scale.

\begin{figure}[!t]
\centering
\includegraphics[width=\textwidth]{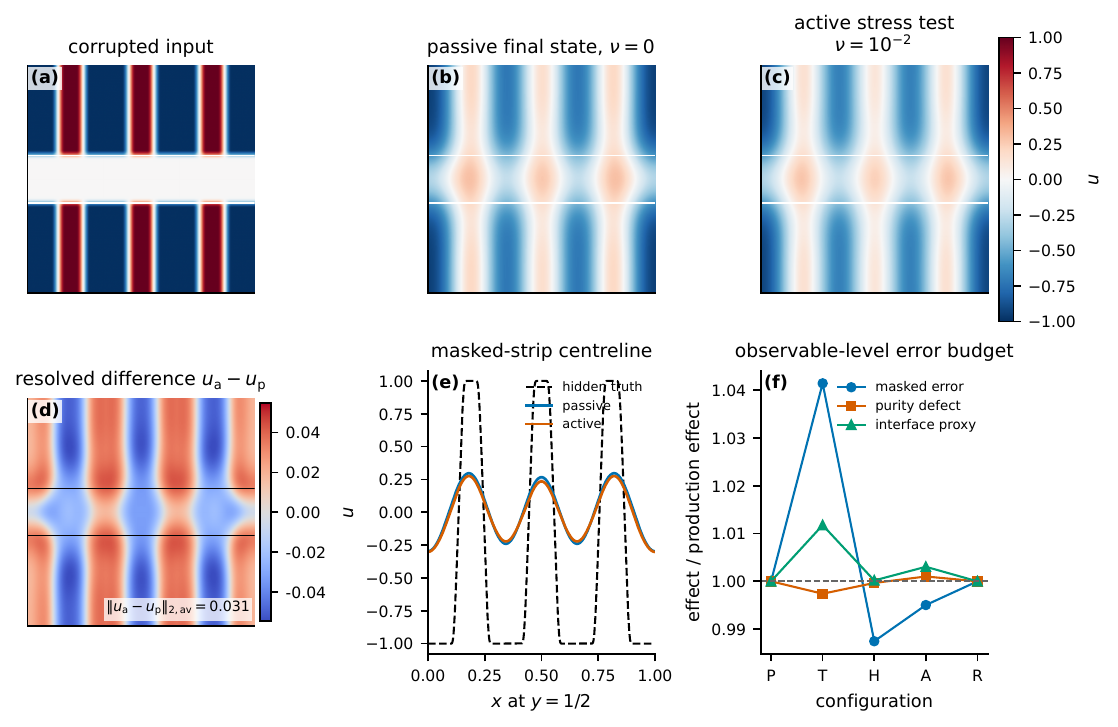}
\caption{Audited missing-strip reconstruction.
(a) Corrupted input, the white band contains no hidden target data.
(b,c) Passive and active final states on the same colour scale, white curves
mark the half-mask contour.  (d) Active-minus-passive field on its own
symmetric scale.  (e) Centreline through the masked strip, the dashed hidden
truth is shown only for post-solve evaluation.  (f) Each paired observable
effect divided by its production value for the production run (P), time-step
halving (T), spatial refinement (H), threefold padding (A), and a
hundredfold tighter nonlinear tolerance (R).}
\label{fig:num-reconstruction-audit}
\end{figure}

For \(\Delta E_M,\Delta P,\Delta I\), the largest one-factor shifts from the
production effects are, respectively,
\[
5.323\times10^{-4},\qquad
4.513\times10^{-5},\qquad
5.333\times10^{-3}.
\]
The corresponding signal-to-shift ratios are \(24.2\), \(386\), and
\(85.1\), and every control preserves all three signs.  The production field
signal is \(26.8\) times its largest one-factor \(L^2\) shift
\((1.154\times10^{-3})\).  Hence both the visual difference and the adverse
observable changes lie above the measured discretisation and solver
uncertainty.

Across the experiment, the maximum recorded active residual is \(9.91\times10^{-8}\), the source--mass residual is \(2.68\times10^{-16}\), and the anchoring projection defect is \(1.22\times10^{-5}\). The ordinary modal tail remains below \(2.49\times10^{-5}\). The maximum Laplacian-weighted tail is \(1.06\times10^{-2}\) on the time-refined control and falls below \(9.9\times10^{-4}\) on the spatially refined active run. The paired effects are nevertheless stable under that refinement. No clipping is applied.

Figure~\ref{fig:num-reconstruction-audit} shows the common-scale states, the
resolved difference, and the observable controls.  This is therefore a
resolved negative reconstruction result: in this parameter regime, activity
reduces interfacial content
while increasing masked error and phase-purity defect.  Its value is that the
same reduced-interface/coarsening response seen in the spinodal ensemble is
now visible in a deterministic geometry with a closed observable-level error
budget and is consistent with the local damping calculation below.

\subsection{Interpretation and limitations}
\label{subsec:num-interpretation}

The five numerical studies serve four distinct  roles. The prescribed exact solution verifies the complete smooth-classifier implementation, including its temporal order and derivative-sensitive error.  The weighted ADMM calculation solves the discrete graph inclusion itself and permits state and active-product comparison with finite $N$, coupled refinement is then used to distinguish a constitutive signal from numerical uncertainty.
The spinodal and stability studies test observable consequences and continuous dependence with error budgets tailored to those questions.  The reconstruction stress test separately resolves a deterministic visual effect and shows that it is adverse in the stated error and purity metrics.
We note that none of these computations proves convergence of the discrete graph problem to the continuum inclusion.

The active contribution should not be described as a universal sharpening or anti-diffusive mechanism.  If \(G=|\nabla u|\geq0\) is frozen locally,
treated as spatially constant, and \(|\Delta u|\) is small, then
\begin{equation}
\label{num:eq:frozen-damping}
a_N(\Delta u)\simeq\frac{2N}{\pi}\Delta u,
\qquad
\Delta\!\left[-\nu a_N(\Delta u)G\right]
\simeq-\frac{2\nu NG}{\pi}\Delta^2u.
\end{equation}
This has the sign of additional fourth-order damping.  A useful local nondimensional comparison is therefore
\(\Theta=2\nu NG/(\pi\mu)\), but it applies only where
\(N|\Delta_h\widetilde u_h|\ll1\).  At the central graph resolution the observed norm ratios
\(\nu\|\Pi_{k\ne0}\chi_h\|_2/
(\mu\|\Delta_h\widetilde u_h\|_2)\) are \(0.0841\) and
\(0.408\) for \(\nu=2\times10^{-3}\) and \(10^{-2}\), respectively.  This quantifies the moderate and stronger active regimes without substituting a post-step state diagnostic for the actual implicit classifier argument.
The resolved increase of spinodal length and reduction of interface content define a coarsening regime consistent with the local damping calculation.
Sharpening, should it occur for other data or parameters, would require a separately resolved transition-width or maximum-gradient verification.

The multiplier $\xi$ is not expected to be unique on
$\{\Delta u=0\}$, and it is not determined by the state equation where the lagged gradient vanishes. This is why the direct graph experiment emphasizes the unique primal state and the active product \(\chi=G\xi\) (and why no
strong multiplier-convergence claim is made). Finally, all experiments here are kept toroidal and therefore support no numerical conclusion about curvature.

\section*{Data availability}

No external datasets were used in this study. All numerical data reported
in the manuscript were generated from the computational procedures and
parameter settings described in the article.

\section*{Conflict of interest}

The authors declare that they have no financial or non-financial interests
that are directly or indirectly related to the work submitted for publication.

\end{document}